\documentclass[12pt,thmsa]{article}
\usepackage{amssymb}
\usepackage{amsfonts}
\usepackage{amsmath}
\usepackage[top=3.8cm,bottom=3.8cm,textwidth=16cm,centering,a4paper]{geometry}
\usepackage{diagbox}
\usepackage{slashbox}

\newtheorem{theorem}{Theorem}[section]

\newtheorem{proposition}[theorem]{Proposition}

\newtheorem{remark}[theorem]{Remark}

\newtheorem{lemma}[theorem]{Lemma}

\newtheorem{definition}[theorem]{Definition}
\newenvironment{proof}[1][Proof]{\noindent\textbf{#1.} }{\ \rule{0.5em}{0.5em}}

\begin{document}

\title{Existence and dynamical behaviour of vectorial standing waves with
prescribed mass for Hartree-Fock type systems}
\date{}
\author{Shuai Yao$^{a}$\thanks{%
E-mail address: shyao2019@163.com (S. Yao)}, Juntao Sun$^{a}$\thanks{%
E-mail address: jtsun@sdut.edu.cn (J. Sun)}, Tsung-fang Wu$^{b}$\thanks{%
E-mail address: tfwu@nuk.edu.tw (T.-F. Wu)} \\
$^{a}${\footnotesize \emph{School of Mathematics and Statistics, Shandong
University of Technology, Zibo 255049, PR China}}\\
$^{b}${\footnotesize \emph{Department of Applied Mathematics, National
University of Kaohsiung, Kaohsiung 811, Taiwan}}}
\maketitle

\begin{abstract}
In this paper, we investigate vectorial standing waves with prescribed mass
for the Hartree-Fock type system (HF system) with the double coupled
feature. Such system is viewed as an approximation of
the Coulomb system with two particles appeared in quantum mechanics. By exploring the interaction of
the double coupled terms, we prove the existence/nonexistence and symmetry
of vectorial energy ground states for the corresponding stationary problem. Furthermore,
we obtain the relation between vectorial energy ground states and
vectorial action ground states in some cases. Finally, we establish
conditions for global well-posedness and finite time blow-up to HF system
with the initial data, and prove orbital stability/strong instability of
standing waves.
\end{abstract}

\tableofcontents



\section{Introduction}

Our starting point is the system of Hartree-Fock equations:%
\begin{equation}
\begin{array}{ll}
i\partial _{t}\psi _{k}+\Delta \psi _{k}-\left( |x|^{-1}\ast
\sum\limits_{j=1}^{N}|\psi _{j}|^{2}\right) \psi _{k}+(V_{\text{ex}}\psi
)_{k}=0, & \text{ }\forall k=1,...,N,%
\end{array}
\label{e1-0}
\end{equation}%
where $\psi _{k}:\mathbb{R\times R}^{3}\rightarrow \mathbb{C}$ and $(V_{%
\text{ex}}\psi )_{k}$ is the $k$'th component of \textit{the exchange
potential} defined by%
\begin{equation}
\begin{array}{ll}
(V_{\text{ex}}\psi )_{k}=\sum\limits_{j=1}^{N}\psi _{j}(y)\int_{\mathbb{R}%
^{3}}\frac{\psi _{k}(y)\bar{\psi}_{j}(y)}{|x-y|}dy, & \text{ }\forall
k=1,...,N.%
\end{array}
\label{e1-1}
\end{equation}%
This system appeared in quantum mechanics in the study of a system of $N$
particles, which is an approximation of the Coulomb system. With respect to system (\ref{e1-0}), it has the advantage of
being consistent with the Pauli principle, see e.g. \cite{F1,Sl1}.

We note that the exchange potential (\ref{e1-1}) is the most difficult term
to be treated in system (\ref{e1-0}). A very simple approximation of this
potential was given by Slater \cite{Sl2} in the form:%
\begin{equation}
\begin{array}{ll}
(V_{\text{ex}}\psi )_{k}\approx C\left( \sum\limits_{j=1}^{N}|\psi
_{j}|^{2}\right) ^{1/3}\psi _{k}, & \text{ }\forall k=1,2,...,N,%
\end{array}
\label{1-1}
\end{equation}%
where $C>0$. For more different local approximations, we refer to \cite{KS,PY} and references therein.

For the single-partical states, i.e. $N=1,$ the exchange potential (\ref{1-1}%
) becomes $(V_{\text{ex}}\psi )_{1}=C|\psi _{1}|^{2/3}\psi _{1}$, and then
system (\ref{e1-0}) is the following Schr\"{o}dinger--Poisson--Slater
equation:

\begin{equation}
\begin{array}{ll}
i\partial _{t}\psi +\Delta \psi -\left( |x|^{-1}\ast |\psi |^{2}\right) \psi
+|\psi |^{2/3}\psi =0 & \text{ in }\mathbb{R}^{3},%
\end{array}
\label{1-3}
\end{equation}%
where we rename $\psi _{1}$ as $\psi ,$ and take $C=1$ for simplicity. S\'{a}%
nchez and Soler \cite{SS1} analyzed the asymptotic behaviour of solutions to
equation (\ref{1-3}). If the term $|\psi |^{2/3}\psi $ is replaced with $%
|\psi |^{p-2}\psi $ $(2<p<6)$ (or, more generally, $f(\psi ))$, then
equation (\ref{1-3}) becomes Schr\"{o}dinger--Poisson equation (also
called Schr\"{o}dinger--Maxwell equation), where the existence,
stability/instability and dynamics of standing waves $\psi (t,x)=e^{i\lambda
t}u(x)$ has been extensively studied in recent years, see e.g. \cite%
{BJL,CDSS,CW,FCW,R,R1,SS,SWF1,SWF2}.

For the two-particle states, i.e. $N=2,$ system (\ref{e1-0}) becomes much
more complicated. In order to further simplify it, we assume the exchange
potential%
\begin{equation}
V_{\text{ex}}\psi =C\binom{(\left\vert \psi _{1}\right\vert ^{p-2}+\beta
\left\vert \psi _{1}\right\vert ^{\frac{p}{2}-2}\left\vert \psi
_{2}\right\vert ^{\frac{p}{2}})\psi _{1}}{(\left\vert \psi _{2}\right\vert
^{p-2}+\beta \left\vert \psi _{1}\right\vert ^{\frac{p}{2}}\left\vert \psi
_{2}\right\vert ^{\frac{p}{2}-2})\psi _{2}},  \label{1-4}
\end{equation}%
where $C>0,\beta \in \mathbb{R}$ and $2<p<6$. Clearly, if $p=\frac{8}{3},$
then (\ref{1-4}) is written as%
\begin{equation*}
V_{\text{ex}}\psi =C\binom{(\left\vert \psi _{1}\right\vert ^{\frac{2}{3}%
}+\beta \left\vert \psi _{1}\right\vert ^{-\frac{2}{3}}\left\vert \psi
_{2}\right\vert ^{\frac{4}{3}})\psi _{1}}{(\left\vert \psi _{2}\right\vert ^{%
\frac{2}{3}}+\beta \left\vert \psi _{1}\right\vert ^{\frac{4}{3}}\left\vert
\psi _{2}\right\vert ^{-\frac{2}{3}})\psi _{2}},
\end{equation*}%
which can be viewed as an approximation of the exchange potential (\ref{1-1}%
). Taking, for simplicity, $C=1$, system (\ref{e1-0}) becomes the following
problem, depending on the parameter $\gamma >0:$%
\begin{equation}
\left\{
\begin{array}{ll}
i\partial _{t}\psi _{1}+\Delta \psi _{1}-\gamma \left( |x|^{-1}\ast
\sum\limits_{j=1}^{2}|\psi _{j}|^{2}\right) \psi _{1}+|\psi _{1}|^{p-2}\psi
_{1}+\beta |\psi _{2}|^{\frac{p}{2}}|\psi _{1}|^{\frac{p}{2}-2}\psi _{1}=0 &
\ \text {in}\ \mathbb{R}^{3}, \\
i\partial _{t}\psi _{2}+\Delta \psi _{2}-\gamma \left( |x|^{-1}\ast
\sum\limits_{j=1}^{2}|\psi _{j}|^{2}\right) \psi _{2}+|\psi _{2}|^{p-2}\psi
_{2}+\beta |\psi _{1}|^{\frac{p}{2}}|\psi _{2}|^{\frac{p}{2}-2}\psi _{2}=0 &
\ \text{in}\ \mathbb{R}^{3}.%
\end{array}%
\right.  \label{e1-2}
\end{equation}%
The typical characteristic of such system lies on the presence of the double
coupled terms, including a Coulomb interacting one and a cooperative pure
power one.

An important topic is to establish conditions for the well-posedness and
blow-up of system (\ref{e1-2}) with the initial data. When $\gamma =0,$
there has been a number of works on local existence, global existence and
scattering theory, see e.g. \cite{FP,S0,X}. However, for $\gamma >0$,
there have no any results on these aspects in the existing literature so far. We
observe that the presence of the double coupled terms yields substantial
difficulties and requires more detailed dispersive estimates. One objective
in this paper is to investigate the global well-posedness and finite time
blow-up of system (\ref{e1-2}) with the initial data.

Another interesting topic on system (\ref{e1-2}) is to study the standing
waves of the form%
\begin{equation*}
\psi _{1}(t,x)=e^{i\lambda t}u(x)\text{ and }\psi _{2}(t,x)=e^{i\lambda
t}v(x),
\end{equation*}%
where $\lambda \in \mathbb{R}$ and $(u,v)$ is a vector function to be found.
Thus $(u,v)$ must verify%
\begin{equation}
\left\{
\begin{array}{ll}
-\Delta u+\lambda u+\gamma \phi _{u,v}u=|u|^{p-2}u+\beta |v|^{\frac{p}{2}%
}|u|^{\frac{p}{2}-2}u & \quad \text{in}\quad \mathbb{R}^{3}, \\
-\Delta v+\lambda v+\gamma \phi _{u,v}v=|v|^{p-2}v+\beta |u|^{\frac{p}{2}%
}|v|^{\frac{p}{2}-2}v & \quad \text{in}\quad \mathbb{R}^{3},%
\end{array}%
\right.   \label{e1-3}
\end{equation}%
where $\phi _{u,v}:=\int_{\mathbb{R}^{3}}\frac{u^{2}(y)+v^{2}(y)}{|x-y|}dy.$
For a solution $\left( u,v\right) $ of system (\ref{e1-3}), we introduce
some concepts of its triviality.

\begin{definition}
\label{D1}A vector function $\left( u,v\right) $ is said to be\newline
$\left( i\right) $ nontrivial if either $u\neq 0$ or $v\neq 0;$\newline
$\left( ii\right) $ semitrivial if it is nontrivial but either $u=0$ or $%
v=0; $\newline
$\left( iii\right) $ vectorial if both of $u$ and $v$ are not zero;
\end{definition}

According to the role of the frequency $\lambda ,$ there are two different
ways to study solutions of system (\ref{e1-3}):\newline
$(i)$ the frequency $\lambda $ is a fixed and assigned parameter;\newline
$(ii)$ the frequency $\lambda $ is an unknown of the problem.

For case $(i)$, one can see that solutions of system (\ref{e1-3}) can be
obtained as critical points of the functional defined in $\mathbf{H}:=H^{1}(%
\mathbb{R}^{3})\times H^{1}(\mathbb{R}^{3})$ by%
\begin{eqnarray*}
E_{\lambda ,\gamma ,\beta }(u,v):&=&\frac{1}{2}\int_{\mathbb{R}%
^{N}}(|\nabla u|^{2}+|\nabla v|^{2})dx+\frac{\lambda }{2}\int_{\mathbb{R}%
^{N}}(|u|^{2}+|v|^{2})dx+\frac{\gamma }{4}\int_{\mathbb{R}^{3}}\phi
_{u,v}(u^{2}+v^{2})dx \\
&&-\frac{1}{p}\int_{\mathbb{R}^{3}}(|u|^{p}+|v|^{p}+2\beta |u|^{\frac{p}{2}%
}|v|^{\frac{p}{2}})dx.
\end{eqnarray*}%
When either $3<p<4$ and $\beta >0,$ or $4\leq p<6$ and $\beta \geq
2^{2q-1}-1,$ d'Avenia, Maia and Siciliano \cite{DMS} studied the existence
of radial vectorial solutions to system (\ref{e1-3}) with $\lambda =1\ $via
the method of Nehari-Pohozaev manifold defined in $\mathbf{H}%
_{r}:=H_{rad}^{1}(\mathbb{R}^{3})\times H_{rad}^{1}(\mathbb{R}^{3})$
developed by Ruiz \cite{R1}. Very recently, we \cite{SW} considered another
interesting case, i.e. $2<p<3,$ and proved the existence of radial vectorial
solutions for system (\ref{e1-3}) with $\lambda =1$ when either $\gamma >0$
small enough or $\beta >\beta (\gamma ).$ In addition, by developing a novel
constraint method, for $\beta >\beta (\gamma ),$ we obtained a vectorial
\textit{action ground state} in $\mathbf{H}$ when either $3\leq p<4$ and $%
0<\gamma <\gamma _{0},$ or $\frac{1+\sqrt{73}}{3}\leq p<4$ and $\gamma >0.$
Here the action ground state is defined in the following sense:

\begin{definition}
\label{d2} We say that $(u,v)$ is an action ground state of system (\ref%
{e1-3}) if it is a solution of system (\ref{e1-3}) having least energy among
all the solutions:%
\begin{equation*}
E_{\lambda ,\gamma ,\beta }(u,v)=\inf \{E_{\lambda ,\gamma ,\beta }(\phi
_{1},\phi _{2}):E_{\lambda ,\gamma ,\beta }^{\prime }(\phi _{1},\phi _{2})=0%
\text{ and }(\phi _{1},\phi _{2})\in \mathbf{H}\backslash \{(0,0)\}\}.
\end{equation*}
\end{definition}

Alternatively one can look for solutions of system (\ref{e1-3}) with $%
\lambda $ unknown. In this case $\lambda \in \mathbb{R}$ appears as a
Lagrange multiplier and $L^{2}$-norms of solutions are prescribed. This
study seems to be particularly meaningful from the physical point of view,
since solutions of system (\ref{e1-2}) with the initial data conserve their
mass along time. Moreover, this study often offers a good insight of the
dynamical properties of solutions for system (\ref{e1-2}), such as stability
or instability. For these reasons, another objective in this paper is to
focus on existence and dynamics of solutions for system (\ref{e1-3}) having
prescribed mass%
\begin{equation}
\int_{\mathbb{R}^{3}}(u^{2}+v^{2})dx=c>0,  \label{e1-4}
\end{equation}%
which has not been concerned before. From
the variational point of view, solutions of problem (\ref{e1-3})-(\ref{e1-4}%
) can be obtained as critical points of the energy functional $I_{\gamma
,\beta }:\mathbf{H\rightarrow }\mathbb{R}$ given by
\begin{eqnarray*}
I_{\gamma ,\beta }(u,v):&=&\frac{1}{2}\int_{\mathbb{R}^{N}}(|\nabla
u|^{2}+|\nabla v|^{2})dx+\frac{\gamma }{4}\int_{\mathbb{R}^{3}}\phi
_{u,v}(u^{2}+v^{2})dx \\
&&-\frac{1}{p}\int_{\mathbb{R}^{3}}(|u|^{p}+|v|^{p}+2\beta |u|^{\frac{p}{2}%
}|v|^{\frac{p}{2}})dx
\end{eqnarray*}%
on the constraint
\begin{equation*}
S_{c}:=\left\{ (u,v)\in \mathbf{H}:\int_{\mathbb{R}^{3}}(u^{2}+v^{2})dx=c%
\right\} .
\end{equation*}%
As $2<p<6,$ it is standard that $I_{\gamma ,\beta }$ is of class $C^{1}$ on $%
S_{c}.$ We will be particularly interested in \textit{energy ground states},
defined as follows:

\begin{definition}
\label{D2} We say that $(u,v)$ is an energy ground state of system (\ref%
{e1-3}) if it is a solution of system (\ref{e1-3}) having minimal energy
among all the solutions:%
\begin{equation*}
I_{\gamma ,\beta }^{\prime }|_{S_{c}}(u,v)=0\text{ and }I_{\gamma ,\beta
}(u,v)=\inf \{I_{\gamma ,\beta }(\phi _{1},\phi _{2}):I_{\gamma ,\beta
}^{\prime }|_{S(c)}(\phi _{1},\phi _{2})=0\text{ and }(\phi _{1},\phi
_{2})\in S_{c}\}.
\end{equation*}
\end{definition}

If $I_{\gamma ,\beta }$ has a global minimizer, then this definition
naturally extends the notion of ground states from linear quantum mechanics.
Moreover, it allows to deal with cases when $I_{\gamma ,\beta }$ is
unbounded from below on $S_{c}$. We also recall the notion of stability and
instability as follows:

\begin{definition}
\label{D3} $(i)$ The set of the ground states $\mathcal{Z}(c)$ is orbitally
stable if for any $\varepsilon >0$, there exists $\delta >0$ such that for
the initial data $(\psi _{1}(0),\psi _{2}(0))\in \mathbf{H}$ satisfying $%
\inf_{(u,v)\in \mathcal{Z}(c)}\Vert (\psi _{1}(0),\psi _{2}(0))-(u,v)\Vert _{%
\mathbf{H}}\leq \delta ,$ we have
\begin{equation*}
\sup_{t\in (0,T_{\max })}\inf_{(u,v)\in \mathcal{Z}(c)}\Vert (\psi
_{1}(t),\psi _{2}(t))-(u,v)\Vert _{\mathbf{H}}\leq \varepsilon ,
\end{equation*}%
where $(\psi _{1}(t),\psi _{2}(t))$ is the solution to system (\ref{e1-2})
with the initial data $(\psi _{1}(0),\psi _{2}(0))$ for $t\in \left[
0,T_{\max }\right) $, and $T_{\max }$ denotes the maximum existence time of
solution.\newline
$(ii)$ A standing wave $(e^{i\lambda t}u,e^{i\lambda t}v)$ is strongly
unstable if for any $\varepsilon >0$, there exists $(\psi _{1}(0),\psi
_{2}(0))\in \mathbf{H}$ such that $\Vert (\psi _{1}(0),\psi
_{2}(0))-(u,v)\Vert _{\mathbf{H}}<\varepsilon $ and the solution $(\psi
_{1}(t),\psi _{2}(t))$ blows up in finite time.
\end{definition}

We note that the definition of stability implicitly requires that system (%
\ref{e1-2}) has a unique global solution, at least for initial data $(\psi
_{1}(0),\psi _{2}(0))$ sufficiently close to $\mathcal{Z}(c).$

When $\gamma =0$, system (\ref{e1-3}) is reduced to the local weakly coupled
nonlinear Schr\"{o}dinger system
\begin{equation}
\left\{
\begin{array}{ll}
-\Delta u+\lambda u=|u|^{p-2}u+\beta |v|^{\frac{p}{2}}|u|^{\frac{p}{2}-2}u &
\quad \text{in}\quad \mathbb{R}^{3}, \\
-\Delta v+\lambda v=|v|^{p-2}v+\beta |u|^{\frac{p}{2}}|v|^{\frac{p}{2}-2}v &
\quad \text{in}\quad \mathbb{R}^{3}.%
\end{array}%
\right.  \label{e1-5}
\end{equation}%
In recent years, the existence and multiplicity of solutions to system (\ref%
{e1-5}) with prescribed mass have attracted much attention, see e.g. \cite%
{BS,CW1,GJ,GJ1,LZ,MS}. However, when $\gamma >0$, due to the interaction of
the double coupled terms, the situation becomes more different and
complicated. Some key challenges need to be solved, such as the
characterization of the geometric structure of the functional $I_{\gamma
,\beta }$, and the relation between two coupled constants $\gamma ,\beta $. In order to overcome these considerable
difficulties, new ideas and techniques have been explored. More details will be discussed in the
next subsection.

\subsection{Main results}

For sake of convenience, we set
\begin{equation*}
A(u,v):=\int_{\mathbb{R}^{3}}(|\nabla u|^{2}+|\nabla v|^{2})dx,\quad
B(u,v):=\int_{\mathbb{R}^{3}}\phi _{u,v}(u^{2}+v^{2})dx,
\end{equation*}%
and
\begin{equation*}
C(u,v):=\int_{\mathbb{R}^{3}}(|u|^{p}+|v|^{p}+2\beta |u|^{\frac{p}{2}}|v|^{%
\frac{p}{2}})dx.
\end{equation*}%
First of all, we consider the minimization problem:
\begin{equation*}
\sigma _{\gamma ,\beta }(c):=\inf_{(u,v)\in S_{c}}I_{\gamma ,\beta }(u,v).
\label{e1-6}
\end{equation*}

\begin{theorem}
\label{t1} Let $2<p<3$ and $c,\beta >0$. Then there exists $\gamma _{\ast
}:=\gamma _{\ast }(c)>0$ such that for $0<\gamma <\gamma _{\ast }$, $\sigma
_{\gamma ,\beta }(c)$ admits a minimizer $(\hat{u},\hat{v})\in S_{c}$,
namely, system (\ref{e1-3}) admits a vectorial energy ground state $(\hat{u},%
\hat{v})\in \mathbf{H}$. In particular, for $\frac{18}{7}<p<3,$ there exists
$c_{\ast }>0$ such that for every $c>c_{\ast }$ and $0<\gamma <\gamma _{\ast
},$ the vectorial energy ground state $(\hat{u},\hat{v})$ is non-radial.
\end{theorem}

For $3\leq p<\frac{10}{3}$ and $c>0,$ we define
\begin{eqnarray*}
\beta _{\star }:&=&\beta _{\star }(c,\gamma ,p) \\
&=&\inf_{(u,v)\in \mathbf{H}\backslash \{(0,0)\}}\left[\frac{\frac{p}{3p-8}\left(
\frac{\gamma (3p-8)}{2(10-3p)}\right) ^{\frac{10-3p}{2}}c^{-2(p-3)}A\left(
u,v\right) ^{\frac{3p-8}{2}}B\left( u,v\right) ^{\frac{10-3p}{2}}(\Vert
u\Vert _{2}^{2}+\Vert v\Vert _{2}^{2})^{2(p-3)}}{2\int_{\mathbb{R}^{3}}|u|^{\frac{p}{2}}|v|^{\frac{p%
}{2}}dx}\right.\\
&&-\left.\frac{\int_{\mathbb{R}%
^{3}}(|u|^{p}+|v|^{p})dx}{2\int_{\mathbb{R}^{3}}|u|^{\frac{p}{2}}|v|^{\frac{p%
}{2}}dx}\right].
\end{eqnarray*}
By (\ref{e2-21}) below, $\beta _{\star }$ is well-defined. In particular,
when $p=3$, $\beta _{\star }\geq \frac{3\sqrt{\gamma }}{2}-1,$
not depending on $c.$

\begin{theorem}
\label{t2} Let $3\leq p<\frac{10}{3}$ and $c,\gamma >0$. Then we have\newline
$(i)$ \textbf{(existence)} $\sigma _{\gamma ,\beta }(c)$ admits a minimizer $%
(\tilde{u},\tilde{v})\in S_{c}$ for $\beta >\max \{\beta _{\star },0\}$. In
addition, the corresponding Lagrange multiplier $\tilde{\lambda}$ is
positive.\newline
$(ii)$ \textbf{(non-existence)} If $\beta _{\star }>0$, then for $p=3$ and $%
\beta <\frac{3\sqrt{\gamma }}{2}-1,$ or for $3<p<\frac{10}{3}$ and $\beta
<\beta _{\star }$, $\sigma _{\gamma ,\beta }(c)$ has no minimizer.
\end{theorem}

\begin{remark}
\label{r1} $(a)$ We give a characterisation of the quantity $\beta
_{\star }$ in Theorem \ref{t2}, which a threshold for the existence and
non-existence of minimizer for $\sigma _{\gamma ,\beta }(c).$\newline
$(b)$ From Theorem \ref{t2}, we can see that the parameter $\beta $ can be
any positive value if $\gamma >0$ small. In particular, when $p=3$, if $%
0<\gamma \leq \frac{4}{9}$, then for any $\beta >0$, the existence result
still holds.
\end{remark}

For $2<p<10/3$, the main difficulty is due to the lack of compactness of
the bounded minimizing sequences. Indeed, the minimizing sequence could run off to spatial infinity and/or spread uniformly in space.
The usual strategy is to control the mass $c$ to exclude both the vanishing and the dichotomy. However, in Theorems \ref{t1} and \ref{t2}, we  control the coupled constant either $\gamma$ or $\beta$, not $c$, although there might be a connection between them. This leads us to have a more considered analysis in the proofs.

Due to Theorems \ref{t1} and \ref{t2}, a natural question is to investigate
the connection between the energy ground states and the action ground
states. This correspondence has been established in \cite{DST,JL0} for
scalar equations and in \cite{CW1} for systems. The next theorem establishes
the relation between vectorial energy ground states and vectorial action
ground states for system (\ref{e1-3}) when $p=3.$

\begin{theorem}
\label{t4} Let $p=3$ and $\gamma ,\lambda >0$. Then for $\beta >\max \{\beta
_{\star },0\}$, system (\ref{e1-3}) has a vectorial action ground state $%
(u_{\lambda },v_{\lambda })\in \mathbf{H}$. Moreover, we have\newline
$(i)$ There admits a unique $c=c(\lambda )>0$ such that $\Vert u_{\lambda
}\Vert _{2}^{2}+\Vert v_{\lambda }\Vert _{2}^{2}=c(\lambda ).$\newline
$(ii)$ $(u_{\lambda },v_{\lambda })$ is a vectorial action ground state of
system (\ref{e1-3}) if and only if $(u_{\lambda },v_{\lambda })$ is a
minimizer of $\sigma _{\gamma ,\beta }(c(\lambda ))$. In particular, there
holds
\begin{equation*}
I_{\gamma ,\beta }(u_{\lambda },v_{\lambda })=\sigma _{\gamma ,\beta
}(c(\lambda ))=\inf_{(u,v)\in S_{c(\lambda )}}I_{\gamma ,\beta }(u,v).
\end{equation*}
\end{theorem}

\begin{remark}
\label{r2} Regarding on the study of vectorial action ground states to
system (\ref{e1-3}), we note that the case of $p=3$ has been considered in \cite%
{SW} but the coupled constant $\gamma >0$ is requied to be small. In our
paper, according to Theorems \ref{t2} and \ref{t4}, we can obtain a
vectorial action ground state without restriction on $\gamma >0,$ which
improves the result in \cite{SW}.
\end{remark}

For $\frac{10}{3}\leq p<6$, $I_{\gamma ,\beta }$ is no longer bounded from below
on $S_{c}$. Then it will not be possible to find a global minimizer. In
order to seek for critical points of $I_{\gamma ,\beta }$ restricted to $%
S_{c},$ we will use the Pohozaev manifold $\mathcal{M}_{\gamma ,\beta }(c)$
that contains all the critical points of $I_{\gamma ,\beta }$ restricted to $%
S_{c}$. It is given by%
\begin{equation*}
\mathcal{M}_{\gamma ,\beta }(c):=\left\{ (u,v)\in S_{c}:\mathcal{P}_{\gamma
,\beta }(u,v)=0\right\} ,
\end{equation*}%
where $\mathcal{P}_{\gamma ,\beta }(u,v)=0$ is the Pohozaev type identity
corresponding to system (\ref{e1-3}). Consider the minimization problem:%
\begin{equation*}
m_{\gamma ,\beta }(c):=\inf_{(u,v)\in \mathcal{M}_{\gamma ,\beta
}(c)}I_{\gamma ,\beta }(u,v),
\end{equation*}%
and set
\begin{equation}
c_{\ast }:=\left( \frac{5}{3}\right) ^{\frac{3}{2}}\inf_{(u,v)\in S_{1}}%
\mathcal{R}_{10/3}(u,v)\text{ and }c^{\ast }:=\left( \frac{2(6-p)}{5p-12}%
\right) ^{\frac{3p-10}{4(p-3)}}\left( \frac{3p}{5p-12}\right) ^{\frac{1}{%
2(p-3)}}\inf_{(u,v)\in S_{1}}\mathcal{R}_{p}(u,v),\label{e1-7}
\end{equation}%
where
\begin{equation*}
\mathcal{R}_{p}(u,v):=\left[ \frac{A(u,v)^{3p-8}(\gamma B(u,v))^{10-3p}}{%
C(u,v)^{2}}\right] ^{\frac{1}{4(p-3)}}.
\end{equation*}

\begin{theorem}
\label{t5} Let $p=\frac{10}{3}$ and $\gamma ,\beta >0$. Then for each $%
c_{\ast }<c<c^{\ast }$, system (\ref{e1-3}) has a vectorial energy ground
state $(\bar{u},\bar{v})\in \mathbf{H}$. Moreover, the corresponding
Lagrange multiplier $\bar{\lambda}$ is positive. While for each $0<c<c_{\ast
}$, there are no critical points of $I_{\gamma ,\beta }$ restricted to $%
S_{c} $.
\end{theorem}

\begin{theorem}
\label{t6} Let $\frac{10}{3}<p<4$ and $\gamma ,\beta >0$, or $4\leq p<6$ and
$\gamma >0,\beta >2^{\frac{p-2}{2}}$. Then for each $0<c<c^{\ast }$, system (%
\ref{e1-3}) has a vectorial energy ground state $(\bar{u},\bar{v})\in
\mathbf{H}$. Moreover, there exists $c^{\star }<c^{\ast }$ such that for
each $0<c<c^{\star }$, the corresponding Lagrange multiplier $\bar{\lambda}$
is positive.
\end{theorem}

\begin{proposition}
\label{P11} Let $(\bar{u},\bar{v})\in
\mathbf{H}$ be the energy ground state obtained in
Theorem \ref{t5} or \ref{t6} with the positive Lagrange multiplier $\bar{%
\lambda}$, then $(\bar{u},\bar{v})$ has exponential decay:
\begin{equation*}
|\bar{u}(x)|\leq Ce^{-\zeta |x|}\text{ and }|\bar{v}(x)|\leq Ce^{-\zeta |x|}%
\text{ for every }x\in \mathbb{R}^{3},
\end{equation*}%
for some $C>0,$ $\zeta >0.$
\end{proposition}

In our approach we do not work in the radial function space $\mathbf{H}_{r}$ and do not need
to consider Palais-Smale sequences, so that we can avoid applying the mini-max approach based on a strong topological argument as in \cite{BS,G1,GJ}. A novel strategy to recover the compactness is to take advantage of the profile decomposition of the minimizing sequence and the monotonicity of ground state energy map. It provides a unified novel approach to study both mass critical and supercritical cases, even for Schr\"{o}dinger-Poisson systems.

Let us summarize the existence and non-existence results presented in our
work in the following table 1.
\begin{table*}[h]
\newcommand{\tabincell}[2]{\begin{tabular}{@{}#1@{}}#2\end{tabular}}
\centering
\fontsize{12}{14}\selectfont
\caption{}
\begin{tabular}{|c|c|c|c|}
\hline
$p$ & parameters conditions  & existence (yes)/non-existence (no) \\
\hline
$2<p<3$ &$0<\gamma<\gamma_{\ast}$  & Yes \\
\hline
$3\leq p<10/3$ & $\tabincell{c}{$\beta>\max\{\beta_{\star},0\}$\\ $\beta<\beta_{\star}$ if $\beta_{\star}>0$}$  & \tabincell{c}{Yes\\ No} \\
\hline
$p=10/3$ & \tabincell{c}{$0<c<c_{\ast}$ \\$c_{\ast}<c<c^{\ast}$} & \tabincell{c}{No \\Yes}\\
\hline
$10/3<p<6$ & $0<c<c^{\ast}$ & Yes\\
\hline
\end{tabular}\vspace{0cm}
\end{table*}

Next, we turn to study the dynamics of the corresponding standing waves for system (%
\ref{e1-2}) with the initial data.

\begin{theorem}[Local well-posedness theory]
\label{t8} Let $2<p<6$ and the initial data $(\psi _{1}(0),\psi _{2}(0))\in
\mathbf{H}$. Then there exists $T_{max}>0$ such that system (\ref{e1-2})
with initial data $(\psi _{1}(0),\psi _{2}(0))$ has a unique solution $(\psi
_{1}(t),\psi _{2}(t))$ with
\begin{equation*}
\psi _{1}(t),\psi _{2}(t)\in \mathcal{C}([0,T_{\max });H^{1}(\mathbb{R}%
^{3})).
\end{equation*}
\end{theorem}

Define the set of ground states $\mathcal{Z}(c)$ by
\begin{equation*}
\mathcal{Z}(c):=\left\{ (u,v)\in S_{c}:I_{\gamma ,\beta }(u,v)=\sigma
_{\gamma ,\beta }(c)\right\} .
\end{equation*}%
By Theorems \ref{t1} and \ref{t2}, $\mathcal{Z}(c)$ is not empty. With the
help of Theorem \ref{t8}, the following stability holds.

\begin{theorem}
\label{t3} Under the assumptions of Theorem \ref{t1} or \ref{t2}, the set $%
\mathcal{Z}(c)$ is orbitally stable.
\end{theorem}

\begin{theorem}[Global well-posedness theory]
\label{t9} Let the initial data $(\psi _{1}(0),\psi _{2}(0))\in \mathbf{H}$.
Then system (\ref{e1-2}) with the initial data $(\psi _{1}(0),\psi _{2}(0))$
is globally well-posed in $\mathbf{H}$ if any one of the following
conditions is satisfied:\newline
$(i)$ $2<p<\frac{10}{3}.$\newline
$(ii)$ $p=\frac{10}{3}$ and $\Vert \psi _{1}(0)\Vert _{2}^{2}+\Vert \psi
_{2}(0)\Vert _{2}^{2}$ is small.\newline
$(iii)$ $\frac{10}{3}<p<6$, $\mathcal{P}_{\gamma ,\beta }(\psi _{1}(0),\psi
_{2}(0))>0$ and $I_{\gamma ,\beta }(\psi _{1}(0),\psi _{2}(0))<m_{\gamma
,\beta }(c).$
\end{theorem}

\begin{theorem}[Finite time blow-up]
\label{t7} Under the assumptions of Theorem \ref{t5} or \ref{t6}. Let $(\psi
_{1}(t),\psi _{2}(t))$ be the solution to system (\ref{e1-2}) with the
initial data $(\psi _{1}(0),\psi _{2}(0))$. If
\begin{equation*}
(|x|\psi _{1}(0),|x|\psi _{2}(0))\in L^{2}(\mathbb{R}^{3})\times L^{2}(%
\mathbb{R}^{3}),\text{ }\mathcal{P}_{\gamma ,\beta }(\psi _{1}(0),\psi
_{2}(0))<0\text{ and }I_{\gamma ,\beta }(\psi _{1}(0),\psi
_{2}(0))<m_{\gamma ,\beta }(c),
\end{equation*}%
then $(\psi _{1}(t),\psi _{2}(t))$ blows up in finite time.
\end{theorem}

Finally, we prove that the associated standing waves are unstable.

\begin{theorem}
\label{t10} Let $(\bar{u},\bar{v})\in \mathbf{H}$ be the energy ground state
obtained in Theorem \ref{t5} or \ref{t6}, then the associated standing wave $%
(e^{i\bar{\lambda}t}\bar{u},e^{i\bar{\lambda}t}\bar{v})$ is strongly
unstable, where $\bar{\lambda}$ is the Lagrange multiplier.
\end{theorem}

The rest of this paper is organized as follows. After introducing some
preliminary results in Section 2, we prove Theorems \ref{t1}--\ref{t2} in
Section 3 and Theorem \ref{t4} in Section 4, respectively. We prove Theorems %
\ref{t5}--\ref{t6} and Proposition \ref{P11} in Section 5. Finally in
Section 6, we study dynamics of the corresponding standing waves and prove
Theorems \ref{t8}--\ref{t10}.

\section{Preliminary results}

In this section, we give some related inequalities and some known results.

\begin{lemma}[Gagliardo-Nirenberg inequality of Power type \protect\cite{W}]

\label{L2-1} Let $2<p<6$. Then there exists a sharp constant $\mathcal{S}%
_{p}>0$ such that
\begin{equation}
\Vert u\Vert _{p}\leq \mathcal{S}_{p}^{1/p}\Vert \nabla u\Vert _{2}^{\frac{%
3(p-2)}{2p}}\Vert u\Vert _{2}^{\frac{6-p}{2p}}\text{ for all }u\in H^{1}(%
\mathbb{R}^{3}),  \label{e2-1}
\end{equation}%
where $\mathcal{S}_{p}=\frac{p}{2\Vert U_{p}\Vert _{2}^{p-2}}$, and $U_{p}$
is the ground state solution of the following equation
\begin{equation*}
-\Delta u+\frac{6-p}{3(p-2)}u=\frac{4}{3(p-2)}|u|^{p-2}u\text{ in }\mathbb{R}%
^{3}.
\end{equation*}
\end{lemma}

\begin{lemma}[Gagliardo-Nirenberg inequality of Hartree type \protect\cite%
{MS1,Y}]
\label{L2-2} There exists the best constant $\mathcal{B}>0$ such that
\begin{equation}
\int_{\mathbb{R}^{3}}(|x|^{-1}\ast |u|^{2})|u|^{2}dx\leq \mathcal{B}\Vert
\nabla u\Vert _{2}\Vert u\Vert _{2}^{3}\text{ for all }u\in H^{1}(\mathbb{R}%
^{3}),  \label{e2-2}
\end{equation}%
where $\mathcal{B}=\frac{2}{\Vert Q\Vert _{2}^{2}}$ and $Q(x)$ is a positive
ground state solution of the following equation
\begin{equation*}
-\Delta u+3u=2(|x|^{-1}\ast |u|^{2})u\quad \text{in}\quad \mathbb{R}^{3}.
\end{equation*}
\end{lemma}

By using inequality (\ref{e2-1}), we have
\begin{eqnarray*}
\int_{\mathbb{R}^{3}}(|u|^{p}+|v|^{p})dx &\leq &\mathcal{S}_{p}\left( \Vert
u\Vert _{2}^{\frac{6-p}{2}}\Vert \nabla u\Vert _{2}^{\frac{3(p-2)}{2}}+\Vert
v\Vert _{2}^{\frac{6-p}{2}}\Vert \nabla v\Vert _{2}^{\frac{3(p-2)}{2}}\right)
\notag \\
&\leq &\mathcal{S}_{p}\left( \Vert u\Vert _{2}^{\frac{6-p}{2}}+\Vert v\Vert
_{2}^{\frac{6-p}{2}}\right) \left( \Vert \nabla u\Vert _{2}^{\frac{3(p-2)}{2}%
}+\Vert \nabla v\Vert _{2}^{\frac{3(p-2)}{2}}\right)  \notag \\
&\leq &4\mathcal{S}_{p}\left( \Vert u\Vert _{2}^{2}+\Vert v\Vert
_{2}^{2}\right) ^{\frac{6-p}{4}}A(u,v)^{\frac{3(p-2)}{4}},
\label{e2-3-1}
\end{eqnarray*}%
and then
\begin{eqnarray}
C(u,v) &\leq &(1+\beta )\int_{\mathbb{R}^{3}}(|u|^{p}+|v|^{p})dx  \notag \\
&\leq &4\mathcal{S}_{p}(1+\beta )\left( \Vert u\Vert _{2}^{2}+\Vert v\Vert
_{2}^{2}\right) ^{\frac{6-p}{4}}A(u,v)^{\frac{3(p-2)}{4}}.
\label{e2-3}
\end{eqnarray}

\begin{lemma}[Hardy-Littlewood-Sobolev inequality \protect\cite{LL}]\label{L2-10}
Assume that $1<a<b<+\infty$ is such that $\frac{1}{a}+\frac{1}{b}=\frac{5}{3}.$ Then there exists $\mathcal{C}_{HLS}>0$ such that
\begin{equation*}
\int_{\mathbb{R}^{3}}\int_{\mathbb{R}^{3}}
\frac{|f(x)||g(y)|}{|x-y|}\leq \mathcal{C}_{HLS}
\|f\|_{a}\|g\|_{b},\quad \forall f\in L^{a}(\mathbb{R}^{3}), g\in L^{b}(\mathbb{R}^{3}).
\end{equation*}
\end{lemma}

\begin{lemma}
\label{L2-3} Let $\beta >0$ and $g_{\beta }(s):=s^{\frac{p}{2}}+(1-s)^{\frac{p%
}{2}}+2\beta s^{\frac{p}{4}}(1-s)^{\frac{p}{4}}$ for $s\in \lbrack 0,1]$.%
\newline
$(i)$ If $2<p<4$, then there exists $s_{\beta }\in (0,1)$ such that $%
g_{\beta }(s_{\beta })=\max_{s\in \lbrack 0,1]}g_{\beta }(s)>1$. In
particular, if $\beta \geq \frac{p-2}{2}$, then $s_{\beta }=\frac{1}{2}.$%
\newline
$(ii)$ If $4\leq p<6$ and $\beta >2^{\frac{p-2}{2}}$, then there exists $%
s_{\beta }=\frac{1}{2}$ such that $g_{\beta }(\frac{1}{2})=\max_{s\in
\lbrack 0,1]}g_{\beta }(s)>1$.
\end{lemma}

\begin{proof}
The proof is similar to the argument in \cite[Lemma 2.4]{DMS}, and we omit
it here.
\end{proof}

\begin{lemma}
\label{L2-4} Let $2<p<4$ and $\beta >0,$ or $4\leq p<6$ and $\beta >2^{\frac{%
p-2}{2}}$. Then for each $\omega \in H^{1}(\mathbb{R}^{3})\backslash \{0\}$,
there exists $s_{\omega }\in (0,1)$ such that
\begin{equation*}
I_{\gamma ,\beta }\left( \sqrt{s_{\omega }}\omega ,\sqrt{1-s_{\omega }}%
\omega \right) <I_{\gamma ,\beta }(\omega ,0)=I_{\gamma ,\beta }(0,\omega ).
\end{equation*}
\end{lemma}

\begin{proof}
Let $(u,v)=(\sqrt{s}\omega ,\sqrt{1-s}\omega )$ for $\omega \in H^{1}(%
\mathbb{R}^{3})\backslash \{0\}$ and $s\in \lbrack 0,1]$. A direct
calculation shows that
\begin{equation*}
A(u,v)=s\Vert \nabla \omega \Vert _{2}^{2}+(1-s)\Vert \nabla \omega \Vert
_{2}^{2}=\Vert \nabla \omega \Vert _{2}^{2}
\end{equation*}%
and
\begin{equation*}
B(u,v)=\int_{\mathbb{R}^{3}}\phi _{\sqrt{s}\omega ,\sqrt{1-s}\omega
}(s\omega ^{2}+(1-s)\omega ^{2})dx=\int_{\mathbb{R}^{3}}\phi _{\omega
}\omega ^{2}dx.
\end{equation*}%
Moreover, by Lemma \ref{L2-3}, there exists $s_{\omega }\in (0,1)$ such that
\begin{equation*}
C(u,v)=\left[ s_{\omega }^{\frac{p}{2}}+\left( 1-s_{\omega }\right) ^{\frac{p%
}{2}}+2\beta s_{\omega }^{\frac{p}{4}}\left( 1-s_{\omega }\right) ^{\frac{p}{%
4}}\right] \int_{\mathbb{R}^{3}}|\omega |^{p}dx>\int_{\mathbb{R}^{3}}|\omega
|^{p}dx.
\end{equation*}%
Then we have
\begin{eqnarray*}
&&I_{\gamma ,\beta }\left( \sqrt{s_{\omega }}\omega ,\sqrt{1-s_{\omega }}%
\omega \right) \\
&=&\frac{1}{2}\Vert \nabla \omega \Vert _{2}^{2}+\frac{\lambda }{4}\int_{%
\mathbb{R}^{3}}\phi _{\omega }\omega ^{2}dx-\frac{1}{p}\left[ s_{\omega }^{%
\frac{p}{2}}+\left( 1-s_{\omega }\right) ^{\frac{p}{2}}+2\beta s_{\omega }^{%
\frac{p}{4}}\left( 1-s_{\omega }\right) ^{\frac{p}{4}}\right] \int_{\mathbb{R%
}^{3}}|\omega |^{p}dx \\
&<&\frac{1}{2}\Vert \nabla \omega \Vert _{2}^{2}+\frac{\lambda }{4}\int_{%
\mathbb{R}^{3}}\phi _{\omega }\omega ^{2}dx-\frac{1}{p}\int_{\mathbb{R}%
^{3}}|\omega |^{p}dx \\
&=&I_{\gamma ,\beta }(\omega ,0)=I_{\gamma ,\beta }(0,\omega ).
\end{eqnarray*}%
The proof is complete.
\end{proof}

Next, we show Brezis-Lieb type results for two coupled terms in the energy
functional.

\begin{lemma}[Brezis-Lieb Lemma]
\label{L2-5} If $(u_{n},v_{n})\rightharpoonup (u,v)$ in $\mathbf{H}$, then
\begin{equation*}
B(u_{n},v_{n})=B(u,v)+B(u_{n}-u,v_{n}-v)+o_{n}(1)
\end{equation*}%
and
\begin{equation*}
\int_{\mathbb{R}^{3}}|u_{n}|^{\frac{p}{2}}|v_{n}|^{\frac{p}{2}}dx=\int_{%
\mathbb{R}^{3}}|u|^{\frac{p}{2}}|v|^{\frac{p}{2}}dx+\int_{\mathbb{R}%
^{3}}|u_{n}-u|^{\frac{p}{2}}|v_{n}-v|^{\frac{p}{2}}dx+o_{n}(1),
\end{equation*}%
where $o_{n}(1)\rightarrow 0$ as $n\rightarrow \infty $.
\end{lemma}

\begin{proof}
We observe that
\begin{equation*}
B(u,v)=\int_{\mathbb{R}^{3}}\phi _{u,v}(|u|^{2}+|v|^{2})dx=\int_{\mathbb{R}%
^{3}}\phi _{z}|z|^{2}dx,
\end{equation*}%
where $z:=u+iv$ and $i$ is the imaginary unit. The first equality can be
proved by adopting the method in \cite[Lemma 2.4]{MS1}. The second one is a
direct consequence of \cite[Theorem 2]{BL} by taking $j:\mathbb{C}%
\rightarrow \mathbb{C}$ defined by $j(u+iv):=|u|^{\frac{p}{2}}|v|^{\frac{p}{2%
}}$ in that theorem. The proof is complete.
\end{proof}

Finally, we give the following Pohozaev identity, and we refer to \cite{DMS}
for the proof.

\begin{lemma}
\label{L2-6} Let $(u,v)$ be a weak solution to system (\ref{e1-3}). Then it
satisfies the Pohozaev identity
\begin{equation*}
\frac{1}{2}A(u,v)+\frac{3\lambda }{2}\int_{\mathbb{R}^{N}}(u^{2}+v^{2})dx+%
\frac{5\gamma }{4}B(u,v)=\frac{3}{p}C(u,v).
\end{equation*}%
Furthermore, it holds
\begin{equation}
\mathcal{P}_{\gamma ,\beta }(u,v):=A(u,v)+\frac{\gamma }{4}B(u,v)-\frac{%
3(p-2)}{2p}C(u,v)=0.  \label{e2-4}
\end{equation}
\end{lemma}

\section{The existence of global minimizer}

\begin{lemma}
\label{L3-1} Let $2<p<\frac{10}{3}$. Then the following statements are true.%
\newline
$(i)$ $I_{\gamma ,\beta }$ is bounded from below and coercive on $S_{c};$%
\newline
$(ii)$ $\sigma _{\gamma ,\beta }(c)$ satisfies the weak sub-additive
inequality:%
\begin{equation}
\sigma _{\gamma ,\beta }(c)\leq \sigma _{\gamma ,\beta }(c^{\prime })+\sigma
_{\gamma ,\beta }(c-c^{\prime })  \label{e2-5}
\end{equation}%
for any $0<c^{\prime }<c.$
\end{lemma}

\begin{proof}
$(i)$ For $(u,v)\in S_{c}$, by (\ref{e2-3}), one has
\begin{eqnarray*}
I_{\gamma ,\beta }(u,v) &\geq &\frac{1}{2}A(u,v)-\frac{1}{p}C(u,v)\geq \frac{1}{2}A(u,v)-\frac{4(1+\beta )\mathcal{S}_{p}}{p}c^{\frac{6-p}{4}%
}A(u,v)^{\frac{3(p-2)}{4}},
\end{eqnarray*}%
which implies that $I_{\gamma ,\beta }(u,v)$ is bounded from below and
coercive on $S_{c}$.\newline
$(ii)$ Following the idea of \cite{L}. According to the density of $%
C_{0}^{\infty }(\mathbb{R}^{3})$ into $H^{1}(\mathbb{R}^{3})$, for any $%
\varepsilon >0$, there exist $(\varphi _{1},\varphi _{2}),(\omega
_{1},\omega _{2})\in C_{0}^{\infty }(\mathbb{R}^{3})\times C_{0}^{\infty }(%
\mathbb{R}^{3})$ with $\Vert \varphi _{1}\Vert _{2}^{2}+\Vert \varphi
_{2}\Vert _{2}^{2}=c^{\prime }$ and $\Vert \omega _{1}\Vert _{2}^{2}+\Vert
\omega _{2}\Vert _{2}^{2}=c-c^{\prime }$ such that
\begin{equation*}
I_{\gamma ,\beta }(\varphi _{1},\varphi _{2})\leq \sigma _{\gamma ,\beta
}(c^{\prime })+\frac{\varepsilon }{2}\quad \mbox{and}\quad I_{\gamma ,\beta }(\omega _{1},\omega _{2})\leq \sigma _{\gamma ,\beta
}(c-c^{\prime })+\frac{\varepsilon }{2}.
\end{equation*}%
Denote by $(\omega _{1}^{n},\omega _{2}^{n})=(\omega _{1}(\cdot
+ne_{1}),\omega _{2}(\cdot +ne_{1})),$ where $e_{1}$ is some given unit
vector in $\mathbb{R}^{3}$. Since the distance between the supports of $%
(\varphi _{1},\varphi _{2})$ and $(\omega _{1}^{n},\omega _{2}^{n})$ is
strictly positive for $n$ large enough and goes to $+\infty $ as $%
n\rightarrow +\infty $, we deduce that $\text{supp}\varphi _{i}\cap \text{%
supp}\omega _{i}^{n}=\emptyset $ for $i=1,2$. Thus for $n$ large enough, we
have $\Vert \varphi _{1}+\omega _{1}^{n}\Vert _{2}^{2}+\Vert \varphi _{2}+\omega
_{2}^{n}\Vert _{2}^{2}=c$ and
\begin{eqnarray*}
\sigma _{\gamma ,\beta }(c) &\leq &I_{\gamma ,\beta }(\varphi _{1}+\omega
_{1}^{n},\varphi _{2}+\omega _{2}^{n}) \\
&=&I_{\gamma ,\beta }(\varphi _{1},\varphi _{2})+I_{\gamma ,\beta }(\omega
_{1}^{n},\omega _{2}^{n}) \\
&\leq &I_{\gamma ,\beta }(\varphi _{1},\varphi _{2})+I_{\gamma ,\beta
}(\omega _{1},\omega _{2}) \\
&\leq &\sigma _{\gamma ,\beta }(c^{\prime })+\sigma _{\gamma ,\beta
}(c-c^{\prime })+\varepsilon ,
\end{eqnarray*}%
where we have used the fact that $I_{\gamma ,\beta }$ are
translation-invariant. The proof is complete.
\end{proof}

\subsection{The case of $2<p<3$}

\begin{lemma}
\label{L3-2} Let $2<p<3$ and $\gamma ,\beta >0.$ Then there holds $\sigma
_{\gamma ,\beta }(c)<0$ for all $c>0$.
\end{lemma}

\begin{proof}
For fixed $(u,v)\in \mathbf{H}\backslash \{(0,0)\},$ we set%
\begin{equation}
(u^{s}(x),v^{s}(x)):=(s^{2}u(sx),s^{2}v(sx))\text{ for }s>0.  \label{e2-19}
\end{equation}%
It is clear that $\left\Vert u^{s}\right\Vert _{2}^{2}+\left\Vert
v^{s}\right\Vert _{2}^{2}=s(\Vert u\Vert _{2}^{2}+\Vert v\Vert _{2}^{2})$
and
\begin{equation*}
I_{\gamma ,\beta }(u^{s},v^{s})=\frac{s^{3}}{2}A(u,v)+\frac{s^{3}\gamma }{4}%
B(u,v)-\frac{s^{2p-3}}{p}C(u,v).
\end{equation*}%
Since $2p-3<3$, it follows that $I_{\gamma ,\beta }(u^{s},v^{s})<0$ for $s>0$
sufficiently small. This implies that there exists $c_{0}>0$ such that $%
\sigma _{\gamma ,\beta }(c)<0$ for $0<c\leq c_{0}$. By taking $c\in \left(
c_{0},2c_{0}\right] $ and using (\ref{e2-5}), we get
\begin{equation*}
\sigma _{\gamma ,\beta }(c)\leq \sigma _{\gamma ,\beta }(c_{0})+\sigma
_{\gamma ,\beta }(c-c_{0})\leq \sigma _{\gamma ,\beta }(c_{0})<0,
\end{equation*}%
since $c-c_{0}\leq c_{0}$. This means that $\sigma _{\gamma ,\beta }(c)<0$
for $c\in \left( c_{0},2c_{0}\right] $. Repeating this procedure, we have $%
\sigma _{\gamma ,\beta }(c)<0$ for all $c>0$. The proof is complete.
\end{proof}

\begin{proposition}
\label{P3-3} Let $2<p<3$ and $\beta >0.$ Then for $\gamma >0$ small enough,
there holds%
\begin{equation}
\sigma _{\gamma ,\beta }(c)<\sigma _{\gamma ,\beta }(c^{\prime })+\sigma
_{\gamma ,\beta }(c-c^{\prime })\text{ for any }0<c^{\prime }<c.
\label{e2-6}
\end{equation}
\end{proposition}

\begin{proof}
We follow the arguments in \cite{CDSS} by making some necessary
modifications. By Lemma \ref{L3-2}, we have proved that $\sigma _{\gamma
,\beta }(c)<0$ for all $c>0$. Let $(\varphi ,\omega )\in S_{1}$ and define
\begin{equation*}
(\varphi _{c},\omega _{c}):=(c^{\frac{2}{10-3p}}\varphi (c^{\frac{p-2}{10-3p}%
}x),c^{\frac{2}{10-3p}}\omega (c^{\frac{p-2}{10-3p}}x)).
\end{equation*}%
It is obvious that $(\varphi _{c},\omega _{c})\in S_{c}$ and
\begin{equation*}
I_{\gamma ,\beta }(\varphi _{c},\omega _{c})=\frac{c^{\frac{6-p}{10-3p}}}{2}%
A(\varphi ,\omega )+\frac{\gamma c^{\frac{18-5p}{10-3p}}}{4}B(\varphi
,\omega )-\frac{c^{\frac{6-p}{10-3p}}}{p}C(\varphi ,\omega ).
\end{equation*}%
Then we have
\begin{equation}
\sigma _{\gamma ,\beta }(c)=c^{\frac{6-p}{10-3p}}J_{1}^{\gamma ,c},
\label{e2-7}
\end{equation}%
where
\begin{equation*}
J_{\rho }^{\gamma ,c}:=\inf_{(u,v)\in S_{\rho }}\left\{ \frac{1}{2}A(u,v)+%
\frac{\gamma c^{\frac{4(3-p)}{10-3p}}}{4}B(u,v)-\frac{1}{p}C(u,v)\right\}
\text{ for }\rho >0.
\end{equation*}%
Using the same scaling argument, we also have
\begin{equation}
J_{\rho }^{\gamma ,c}=\rho ^{\frac{6-p}{10-3p}}J_{1}^{\gamma ,\rho c}.
\label{e2-8}
\end{equation}%
Take $\rho =c^{\prime }/c$, and it follows from (\ref{e2-7}) and (\ref{e2-8}%
) that (\ref{e2-6}) is equivalent to
\begin{equation}
J_{1}^{\gamma ,c}<J_{\rho }^{\gamma ,c}+J_{1-\rho }^{\gamma ,c},\quad
\forall \rho \in (0,1).  \label{e2-9}
\end{equation}%
Note that $\lim_{\gamma \rightarrow 0}J_{1}^{\gamma ,c}=J_{1}^{0,c}$.
Similar to the argument in \cite{GJ1}, for each $\kappa >0,$ $J_{\kappa
}^{0,c}$ satisfies the strict sub-additive inequality, that is,
\begin{equation}
J_{\kappa }^{0,c}<J_{\rho }^{0,c}+J_{\kappa -\rho }^{0,c},\quad \forall \rho
\in (0,\kappa ).  \label{e2-10}
\end{equation}

Next, we prove that $J_{1}^{\gamma ,c}$ satisfies (\ref{e2-9}) for $\gamma
>0 $ small enough. Assume on the contrary. Then there exist two sequences $%
\left\{ \gamma _{n}\right\} $ satisfying $\gamma _{n}\rightarrow 0$ as $%
n\rightarrow \infty $ and $\left\{ \kappa _{n}\right\} \subset (0,1)$ such
that
\begin{equation}
J_{1}^{\gamma _{n},c}=J_{\kappa _{n}}^{\gamma _{n},c}+J_{1-\kappa
_{n}}^{\gamma _{n},c}.  \label{e2-11}
\end{equation}%
Without loss of generality, we assume that $\frac{1}{2}\leq \kappa _{n}<1$.
Moreover, we conclude that $\kappa _{n}\rightarrow 1$ as $n\rightarrow
\infty $, otherwise we get a contradiction with (\ref{e2-10}). Now, we prove
the following claim:
\begin{equation}
J_{\kappa _{n}}^{\gamma _{n},c}<J_{\rho }^{\gamma _{n},c}+J_{\kappa
_{n}-\rho }^{\gamma _{n},c},\quad \forall \rho \in \left( 0,\kappa
_{n}\right) ,  \label{e2-12}
\end{equation}%
for $n$ large enough. Suppose that the claim is not true, then there exists
a sequence $\left\{ \rho _{n}\right\} $ with $\rho _{n}\in \left( \frac{1}{2}%
\kappa _{n},\kappa _{n}\right) $ such that
\begin{equation}
J_{\kappa _{n}}^{\gamma _{n},c}=J_{\rho _{n}}^{\gamma _{n},c}+J_{\kappa
_{n}-\rho _{n}}^{\gamma _{n},c}.  \label{e2-13}
\end{equation}%
It follows from (\ref{e2-11}) and (\ref{e2-13}) that
\begin{eqnarray*}
J_{\rho _{n}}^{\gamma _{n},c}+J_{1-\rho _{n}}^{\gamma _{n},c} \geq J_{1}^{\gamma _{n},c}=J_{\rho _{n}}^{\gamma _{n},c}+J_{\kappa _{n}-\rho
_{n}}^{\gamma _{n},c}+J_{1-\kappa _{n}}^{\gamma _{n},c}\geq J_{\rho _{n}}^{\gamma _{n},c}+J_{1-\rho _{n}}^{\gamma _{n},c},
\end{eqnarray*}%
where we have used the fact of $J_{1-\rho _{n}}^{\gamma _{n},c}\leq
J_{1-\kappa _{n}}^{\gamma _{n},c}+J_{1-\rho _{n}-\left( 1-\kappa _{n}\right)
}^{\gamma _{n},c}$. So we have $J_{\rho _{n}}^{\gamma _{n},c}+J_{1-\rho _{n}}^{\gamma _{n},c}=J_{1}^{\gamma
_{n},c}$. Since $\rho _{n}\geq \frac{\kappa _{n}}{2}$ and $\kappa _{n}\geq \frac{1}{2}$%
, we have $\frac{1}{4}\leq \rho _{n}<\frac{1}{2}$. Then we may assume that $%
\rho _{n}\rightarrow \rho ,$ where $\frac{1}{4}\leq \rho \leq \frac{1}{2}$.
Then we deduce that $J_{\rho }^{0,c}+J_{1-\rho }^{0,c}=J_{1}^{0,c}$ as $%
n\rightarrow \infty $, which is a contradiction with (\ref{e2-10}). Hence the
strict inequality (\ref{e2-12}) holds.

For $n$ large enough, there exists a minimizer $(u_{n},v_{n})$ of $J_{\kappa
_{n}}^{\gamma _{n},c}$ such that $\{(\kappa _{n})^{-1/2}(u_{n},v_{n})\}$ is
a minimizing sequence for $J_{1}^{0,c}$. It follows from (\ref{e2-10}) that $%
(\kappa _{n})^{-1/2}(u_{n},v_{n})\rightarrow (u_{\infty },v_{\infty })$ in $\mathbf{H}$%
, where $(u_{\infty },v_{\infty })$ is a minimizer of $J_{1}^{0,c}$. The
same holds for $\left\{ u_{n},v_{n}\right\} $, given that $\kappa
_{n}\rightarrow 1$. Without loss of generality, we may assume that $%
(u_{n},v_{n})$ and $(u_{\infty },v_{\infty })$ satisfy Euler-Lagrange
equations in $\mathbb{R}^{3},$ respectively,
\begin{equation*}
\left\{
\begin{array}{ll}
-\Delta u_{n}+\lambda _{n}u_{n}+\gamma _{n}c^{\frac{4(3-p)}{10-3p}}\phi
_{u_{n},v_{n}}u_{n}=|u_{n}|^{p-2}u_{n}+\beta |v_{n}|^{\frac{p}{2}}|u_{n}|^{%
\frac{p}{2}-2}u_{n} & \quad \text{in}\quad \mathbb{R}^{3}, \\
-\Delta v_{n}+\lambda _{n}v_{n}+\gamma _{n}c^{\frac{4(3-p)}{10-3p}}\phi
_{u_{n},v_{n}}v_{n}=|v_{n}|^{p-2}v_{n}+\beta |u_{n}|^{\frac{p}{2}}|v_{n}|^{%
\frac{p}{2}-2}u_{n} & \quad \text{in}\quad \mathbb{R}^{3},%
\end{array}%
\right.
\end{equation*}%
with $\Vert u_{n}\Vert _{2}^{2}+\Vert v_{n}\Vert _{2}^{2}=\kappa _{n},$ and
\begin{equation*}
\left\{
\begin{array}{ll}
-\Delta u_{\infty }+\lambda _{\infty }u_{\infty }=|u_{\infty
}|^{p-2}u_{\infty }+\beta |v_{\infty }|^{\frac{p}{2}}|u_{\infty }|^{\frac{p}{%
2}-2}u_{\infty } & \quad \text{in}\quad \mathbb{R}^{3}, \\
-\Delta v_{\infty }+\lambda _{\infty }v_{\infty }=|v_{\infty
}|^{p-2}v_{\infty }+\beta |u_{\infty }|^{\frac{p}{2}}|v_{\infty }|^{\frac{p}{%
2}-2}v_{\infty } & \quad \text{in}\quad \mathbb{R}^{3},%
\end{array}%
\right.
\end{equation*}%
with $\Vert u_{\infty }\Vert _{2}^{2}+\Vert v_{\infty }\Vert _{2}^{2}=1$ and
$\lambda _{\infty }>0$. Note that
\begin{equation*}
\frac{J_{1}^{\gamma _{n},c}-J_{\kappa _{n}}^{\gamma _{n},c}}{1-\kappa _{n}}=%
\frac{J_{1-\kappa _{n}}^{\gamma _{n},c}}{1-\kappa _{n}},
\end{equation*}%
and as $\kappa _{n}\rightarrow 1$, we have
\begin{equation*}
-\lambda _{\infty }\geq \frac{J_{1}^{\gamma _{n},c}-J_{\kappa _{n}}^{\gamma
_{n},c}}{1-\kappa _{n}}.
\end{equation*}%
But it follows from (\ref{e2-8}) that
\begin{equation*}
\frac{J_{1-\kappa _{n}}^{\gamma _{n},c}}{1-\kappa _{n}}=\left( 1-\kappa
_{n}\right) ^{\frac{2(p-2)}{10-3p}}J_{1}^{\gamma _{n},\left( 1-\kappa
_{n}\right) c}\rightarrow 0\text{ as }\kappa _{n}\rightarrow 1.
\end{equation*}%
Hence we arrive at a contradiction. This completes the proof.
\end{proof}

\begin{lemma}
\label{L3-10}(\cite{R,R1}) If $u$ is a radially symmetric function, then
there holds
\begin{equation*}
\Vert u\Vert _{p}^{p}\leq \mathcal{C}(p,s)\Vert u\Vert _{\dot{H}^{s}}^{\frac{%
\theta p}{2-\theta }}\left( \int_{\mathbb{R}^{3}}\int_{\mathbb{R}^{3}}\frac{%
u(x)^{2}u(y)^{2}}{|x-y|}dxdy\right) ^{\frac{(1-\theta )p}{4-2\theta }},
\end{equation*}%
where $\theta :=\frac{6-\frac{5}{2}p}{3-ps-p}$ and
$\frac{16s+2}{6s+1}<p\leq \frac{6}{3-2s}$ for $1/2<s<3/2.$ In particular, for $\frac{18}{7}<p<6$, there holds
\begin{equation}
\Vert u\Vert _{p}^{p}\leq \mathcal{C}(p)\Vert \nabla u\Vert _{2}^{\frac{5p-12%
}{3}}\left( \int_{\mathbb{R}^{3}}\int_{\mathbb{R}^{3}}\frac{u(x)^{2}u(y)^{2}%
}{|x-y|}dxdy\right) ^{\frac{6-p}{6}}.  \label{e3-23}
\end{equation}
\end{lemma}

\begin{lemma}
\label{L3-3} Let $\frac{18}{7}<p<3.$ Then the energy functional $I_{\gamma
,\beta }$ is bounded from below by a constant not depending on $c,$ on $S_{r,c}:=\{(u,v)\in \mathbf{H}_{r}:\Vert u\Vert _{2}^{2}+\Vert v\Vert _{2}^{2}=c\}.$
\end{lemma}

\begin{proof}
For $(u,v)\in S_{r,c},$ it follows from (\ref{e3-23}) that
\begin{eqnarray}
C(u,v) &=&\int_{\mathbb{R}^{3}}(|u|^{p}+|v|^{p}+2\beta |u|^{\frac{p}{2}}|v|^{%
\frac{p}{2}})dx  \notag \\
&\leq &(1+\beta )\int_{\mathbb{R}^{3}}(|u|^{p}+|v|^{p})dx  \notag \\
&\leq &\mathcal{C}(p)(1+\beta )\left[ \Vert \nabla u\Vert _{2}^{\frac{5p-12}{%
3}}\left( \int_{\mathbb{R}^{3}}\int_{\mathbb{R}^{3}}\frac{u(x)^{2}u(y)^{2}}{%
|x-y|}dxdy\right) ^{\frac{6-p}{6}}\right.  \notag \\
&&\left. +\Vert \nabla v\Vert _{2}^{\frac{5p-12}{3}}\left( \int_{\mathbb{R}%
^{3}}\int_{\mathbb{R}^{3}}\frac{v(x)^{2}v(y)^{2}}{|x-y|}dxdy\right) ^{\frac{%
6-p}{6}}\right]  \notag \\
&\leq &\mathcal{C}(p)(1+\beta )\left( \Vert \nabla u\Vert _{2}^{\frac{5p-12}{%
3}}+\Vert \nabla v\Vert _{2}^{\frac{5p-12}{3}}\right) \left[ \left( \int_{%
\mathbb{R}^{3}}\int_{\mathbb{R}^{3}}\frac{u(x)^{2}u(y)^{2}}{|x-y|}%
dxdy\right) ^{\frac{6-p}{6}}\right.  \notag \\
&&\left. +\left( \int_{\mathbb{R}^{3}}\int_{\mathbb{R}^{3}}\frac{%
v(x)^{2}v(y)^{2}}{|x-y|}dxdy\right) ^{\frac{6-p}{6}}\right]  \notag \\
&\leq &4\mathcal{C}(p)(1+\beta )A(u,v)^{\frac{5p-12}{6}%
}B(u,v)^{\frac{6-p}{6}}.  \label{2-15}
\end{eqnarray}%
Then by (\ref{2-15}) we have%
\begin{equation*}
I_{\gamma ,\beta }(u,v)\geq \frac{1}{2}A(u,v)+\frac{\gamma }{4}B(u,v)-\frac{4%
\mathcal{C}(p)(1+\beta )}{p}A(u,v)^{\frac{5p-12}{6}}B(u,v)^{\frac{6-p}{6}%
},
\end{equation*}%
which implies that $I_{\gamma ,\beta }$ is bounded from below by a constant
not depending on $c$. In fact, we set the function of two variables
\begin{equation*}
f(x,y):=\frac{1}{2}x+\frac{\gamma }{4}y-\frac{4\mathcal{C}(p)(1+\beta )}{p}%
x^{\frac{5p-12}{6}}y^{\frac{6-p}{6}},\text{ }\forall (x,y)\in \mathbb{R}%
^{+}\times \mathbb{R}^{+}.
\end{equation*}%
By a direct calculation, there exists stationary point $(x_{0},y_{0})\in
\mathbb{R}^{+}\times \mathbb{R}^{+}$ not depending on $c$ such that $%
f_{x}(x_{0},y_{0})=f_{y}(x_{0},y_{0})=0$. Furthermore, we have%
\begin{equation*}
f_{xx}(x_{0},y_{0})=\frac{\mathcal{C}(p)(1+\beta )(5p-12)(18-5p)}{9p}x_{0}^{-%
\frac{24-5p}{6}}y_{0}^{\frac{6-p}{6}}
\end{equation*}%
and
\begin{equation*}
f_{xx}(x_{0},y_{0})f_{yy}(x_{0},y_{0})-f_{xy}^{2}(x_{0},y_{0})=\frac{8[%
\mathcal{C}(p)(1+\beta )]^{2}(5p-12)(6-p)(3-p)}{27p^{2}}x_{0}^{-\frac{18-5p}{%
3}}y_{0}^{-\frac{p}{6}}>0,
\end{equation*}%
which implies that the function $f(x,y)$ achieves its minimum at point $%
(x_{0},y_{0})$, and so the energy functional $I_{\gamma ,\beta }$ is bounded
from below. This completes the proof.
\end{proof}

\begin{lemma}
\label{L3-4} Let $2<p<3.$ Then it holds $\lim_{c\rightarrow 0}\frac{\sigma
_{\gamma ,\beta }(c)}{c}=0$.
\end{lemma}

\begin{proof}
Notice that
\begin{equation*}
\frac{\sigma _{0,\beta }(c)}{c}\leq \frac{\sigma _{\gamma ,\beta }(c)}{c}<0,
\end{equation*}%
where
\begin{equation*}
\sigma _{0,\beta }(c):=\inf_{(u,v)\in S_{c}}I_{0,\beta }(u,v)
\end{equation*}%
with
\begin{equation*}
I_{0,\beta }(u,v):=\frac{1}{2}A(u,v)-\frac{1}{p}C(u,v).
\end{equation*}%
Due to \cite{GJ1}, there exists a minimizer $(u_{0},v_{0})$ for $\sigma
_{0,\beta }(c)$ satisfying
\begin{equation}
\left\{
\begin{array}{ll}
-\Delta u_{0}+\lambda _{c}u_{0}=|u_{0}|^{p-2}u_{0}+\beta |v_{0}|^{\frac{p}{2}%
}|u_{0}|^{\frac{p}{2}-2}u_{0} & \quad \text{in}\quad \mathbb{R}^{3}, \\
-\Delta v_{0}+\lambda _{c}v_{0}=|v_{0}|^{p-2}v_{0}+\beta |u_{0}|^{\frac{p}{2}%
}|v_{0}|^{\frac{p}{2}-2}u_{0} & \quad \text{in}\quad \mathbb{R}^{3},%
\end{array}%
\right.  \label{e2-16}
\end{equation}%
which indicates that
\begin{equation*}
\frac{\lambda _{c}}{2}=\frac{C(u_{0},v_{0})-A(u_{0},v_{0})}{2\int_{\mathbb{R}%
^{3}}(|u_{0}|^{2}+|v_{0}|^{2})dx}\geq \frac{\frac{1}{p}C(u_{0},v_{0})-\frac{1%
}{2}A(u_{0},v_{0})}{2\int_{\mathbb{R}^{3}}(|u_{0}|^{2}+|v_{0}|^{2})dx}=\frac{%
-\sigma _{0,\beta }(c)}{c}>0,  \label{e2-17}
\end{equation*}%
where $\lambda _{c}$ is the Lagrange multiplier associated to the minimizer $%
(u_{0},v_{0})$. Next, we prove that $\lim_{c\rightarrow 0}\lambda _{c}=0$.
Suppose on the contrary. Then we can assume that there exists a sequence $%
\{c_{n}\}\subset \mathbb{R}$ satisfying $c_{n}\rightarrow 0$ as $%
n\rightarrow \infty $ such that $\lambda _{c_{n}}>M$ for some $M\in (0,1)$.
Since the minimizers $(u_{n},v_{n}):=(u_{c_{n}},v_{c_{n}})$ satisfy system (%
\ref{e2-16}), we have
\begin{eqnarray*}
M(\Vert u_{n}\Vert _{H^{1}}^{2}+\Vert v_{n}\Vert _{H^{1}}^{2}) &\leq &\Vert
\nabla u_{n}\Vert _{2}^{2}+\Vert \nabla v_{n}\Vert _{2}^{2}+M(\Vert
u_{n}\Vert _{2}^{2}+\Vert v_{n}\Vert _{2}^{2}) \\
&\leq &\Vert \nabla u_{n}\Vert _{2}^{2}+\Vert \nabla v_{n}\Vert
_{2}^{2}+\lambda _{c_{n}}(\Vert u_{n}\Vert _{2}^{2}+\Vert v_{n}\Vert
_{2}^{2}) \\
&=&C(u_{n},v_{n}) \\
&\leq &K(p,\beta )(\Vert u_{n}\Vert _{H^{1}}^{p}+\Vert v_{n}\Vert
_{H^{1}}^{p})
\end{eqnarray*}%
for some $K(p,\beta )>0$. This implies that there exists $M_{1}>0$ such that
$A(u_{n},v_{n})\geq M_{1}$. But as $n\rightarrow +\infty $, by (\ref{e2-3}%
), we get
\begin{equation*}
0\geq I_{0,\beta }(u_{n},v_{n})\geq \frac{1}{2}A(u_{n},v_{n})-\frac{%
(4+2\beta )\mathcal{S}_{p}}{p}c_{n}^{\frac{6-p}{4}}A(u_{n},v_{n})^{\frac{%
3(p-2)}{4}}\geq \frac{1}{2}M_{1}-o_{n}(1),
\end{equation*}%
which is a contradiction. This completes the proof.
\end{proof}

\begin{lemma}
\label{L3-5} Let $2<p<3$. Then it holds $\inf_{c>0}\sigma _{\gamma ,\beta
}(c)=-\infty $.
\end{lemma}

\begin{proof}
Now we argue by contradiction and assume that there exists $K>0$ such that $%
\sigma _{\gamma ,\beta }(c)>-K$ for all $c>0$. In this case we would have $%
\lim_{c\rightarrow \infty }\frac{\sigma _{\gamma ,\beta }(c)}{c}=0$. Since $%
\lim_{c\rightarrow 0}\frac{\sigma _{\gamma ,\beta }(c)}{c}=0$ and $\sigma
_{\gamma ,\beta }(c)<0$, the function $c\rightarrow \frac{\sigma _{\gamma
,\beta }(c)}{c}$ attains a global minimum, i.e. there exists $c_{0}$ such
that
\begin{equation}
\frac{\sigma _{\gamma ,\beta }(c_{0})}{c_{0}}\leq \frac{\sigma _{\gamma
,\beta }(c)}{c}\text{ for all }c>0.  \label{e2-15}
\end{equation}%
On the other hand, by Lemma \ref{L3-1}, one has
\begin{equation}
\frac{\sigma _{\gamma ,\beta }(2c_{0})}{2c_{0}}\leq \frac{\sigma _{\gamma
,\beta }(c_{0})+\sigma _{\gamma ,\beta }(c_{0})}{2c_{0}}=\frac{\sigma
_{\gamma ,\beta }(c_{0})}{c_{0}}.  \label{e2-18}
\end{equation}%
It follows from (\ref{e2-15}) and (\ref{e2-18}) that the function $%
c\rightarrow \frac{\sigma _{\gamma ,\beta }(c)}{c}$ attains a global minimum
also at $2c_{0}$. Repeating such step, we deduce that%
\begin{equation*}
\frac{\sigma _{\gamma ,\beta }(kc_{0})}{kc_{0}}=\frac{\sigma _{\gamma ,\beta
}(c_{0})}{c_{0}},\quad \forall k\in \mathbb{N},
\end{equation*}%
which implies that $\liminf_{c\rightarrow \infty }\frac{\sigma _{\gamma
,\beta }(c)}{c}<0.$ This contradicts with the fact that $\lim_{c\rightarrow
\infty }\frac{\sigma _{\gamma ,\beta }(c)}{c}=0$. This completes the proof.
\end{proof}

\textbf{Now we give the proof of Theorem \ref{t1}:} Let $\left\{
(u_{n},v_{n})\right\} \subset S_{c}$ be a minimizing sequence to $\sigma
_{\gamma ,\beta }(c)$. Our aim is to prove that $\left\{
(u_{n},v_{n})\right\} $ is compact in $\mathbf{H}$ up to translations. By
Lemma \ref{L3-1}$(i)$, we obtain that $\left\{ (u_{n},v_{n})\right\} $ is
bounded in $\mathbf{H}$. If
\begin{equation*}
\sup_{y\in \mathbb{R}^{3}}\int_{B_{R}(y)}(u_{n}^{2}+v_{n}^{2})dx=o_{n}(1)%
\text{ for some }R>0,
\end{equation*}%
then it follows from Vanishing lemma in \cite[Lemma I.1]{L1} that $\left\{
(u_{n},v_{n})\right\} \rightarrow (0,0)$ in $L^{p}(\mathbb{R}^{3})\times
L^{p}(\mathbb{R}^{3})$ for $2<p<6$. This is incompatible with the fact that $%
I_{\gamma ,\beta }(u_{n},v_{n})<0$. Thus there exist a constant $D_{0}>0$
and a sequence $\left\{ y_{n}\right\} \subset \mathbb{R}^{3}$ such that
\begin{equation*}
\int_{B_{R}(y_{n})}(u_{n}^{2}+v_{n}^{2})dx\geq D_{0}>0.
\end{equation*}%
Since $\left\{ (u_{n},v_{n})\right\} $ is bounded in $\mathbf{H}$, we obtain
that $\left\{ (u_{n},v_{n})\right\} $ is weakly convergent in $\mathbf{H}$
and is locally compact in $L^{p}(\mathbb{R}^{3})\times L^{p}(\mathbb{R}^{3})$
for $2\leq p<6$, and so
\begin{equation*}
(u_{n}(x-y_{n}),v_{n}(x-y_{n}))\rightharpoonup (\hat{u},\hat{v})\neq
(0,0)\quad \text{in}\, \mathbf{H}.
\end{equation*}%
Hence, $\rho :=\Vert \hat{u}\Vert _{2}^{2}+\Vert \hat{v}\Vert _{2}^{2}>0$.
Next, we prove that $(u,v)\in S_{c}$. Suppose that $\rho <c$. Since
\begin{equation*}
\Vert u_{n}\Vert _{2}^{2}+\Vert v_{n}\Vert _{2}^{2}=\Vert u_{n}-\hat{u}\Vert
_{2}^{2}+\Vert v_{n}-\hat{v}\Vert _{2}^{2}+\Vert \hat{u}\Vert _{2}^{2}+\Vert
\hat{v}\Vert _{2}^{2}+o_{n}(1),
\end{equation*}%
by Lemma \ref{L2-5} and the translational invariance, we conclude that
\begin{equation*}
\sigma _{\gamma ,\beta }(c)=\lim_{n\rightarrow \infty }I_{\gamma ,\beta
}(u_{n},v_{n})\geq I_{\gamma ,\beta }(\hat{u},\hat{v})+\lim_{n\rightarrow
\infty }I_{\gamma ,\beta }(u_{n}-u,v_{n}-v)\geq \sigma _{\gamma ,\beta
}(\rho )+\sigma _{\gamma ,\beta }(c-\rho ),
\end{equation*}%
which contradicts with (\ref{e2-6}). So, we have $\Vert \hat{u}\Vert
_{2}^{2}+\Vert \hat{v}\Vert _{2}^{2}=c$ and $(\hat{u},\hat{v})\in S_{c}$.
Moreover, using (\ref{e2-3-1}), we deduce that $(u_{n},v_{n})\rightarrow (%
\hat{u},\hat{v})$ in $L^{p}(\mathbb{R}^{3})\times L^{p}(\mathbb{R}^{3})$.
Then there holds
\begin{equation*}
\sigma _{\gamma ,\beta }(c)\leq I_{\gamma ,\beta }(\hat{u},\hat{v})\leq
\lim_{n\rightarrow \infty }I_{\gamma ,\beta }(u_{n},v_{n})=\sigma _{\gamma
,\beta }(c).
\end{equation*}%
Hence $(\hat{u},\hat{v})\in S_{c}$ is a nontrivial minimizer of $\sigma
_{\gamma ,\beta }(c)$.

Next, we claim that $\hat{u}\neq 0$ and $\hat{v}\neq 0$. If not, we may
assume that $\hat{v}\equiv 0$. Then by Lemma \ref{L2-4}, there exists $%
s_{\beta }\in (0,1)$ such that $(\sqrt{s_{\beta }}\hat{u},\sqrt{1-s_{\beta }}%
\hat{u})\in S_{c}$ and
\begin{equation*}
I_{\gamma ,\beta }(\sqrt{s_{\beta }}\hat{u},\sqrt{1-s_{\beta }}\hat{u}%
)<I_{\gamma ,\beta }(\hat{u},0)=I_{\gamma ,\beta }(0,\hat{u})=\sigma
_{\gamma ,\beta }(c),
\end{equation*}%
which is a contradiction. So, $(\hat{u},\hat{v})\in S_{c}$ is a vectorial
minimizer of $\sigma _{\gamma ,\beta }(c)$.

Finally, we study the symmetry breaking phenomenon for minimizer $(\hat{u},%
\hat{v})$ when $\frac{18}{7}<p<3$. By Proposition \ref{P3-3} and Lemma \ref%
{L3-5}, we know that the ground state energy $\sigma _{\gamma ,\beta }(c)$
is strictly decreasing on $c$ and $\inf_{c>0}\sigma _{\gamma ,\beta
}(c)=-\infty $. However, it follows Lemma \ref{L3-3} that $\inf_{c>0}\sigma
_{\gamma ,\beta }^{r}(c)>-\infty ,$ where $\sigma _{\gamma ,\beta
}^{r}(c):=\inf_{(u,c)\in S_{r,c}}I_{\gamma ,\beta }(u,v).$ Then the symmetry
breaking phenomenon occurs for $c>0$ sufficiently large.

\subsection{The case of $3\leq p<\frac{10}{3}$}

Following the idea of Lions \cite{L}, we have the following results:
\begin{eqnarray}
\int_{\mathbb{R}^{3}}(|u|^{3}+|u|v^{2})dx &=&\int_{\mathbb{R}^{3}}(-\Delta
\phi _{u,v})|u|dx  \notag \\
&=&\int_{\mathbb{R}^{3}}\langle \nabla \phi _{u,v},\nabla |u|\rangle dx
\notag \\
&\leq &\left( \int_{\mathbb{R}^{3}}|\nabla u|^{2}dx\right) ^{\frac{1}{2}%
}\left( \int_{\mathbb{R}^{3}}|\nabla \phi _{u,v}|^{2}dx\right) ^{\frac{1}{2}}
\notag \\
&\leq &\Vert \nabla u\Vert _{2}B(u,v)^{\frac{1}{2}},  \label{e2-28}
\end{eqnarray}%
and
\begin{eqnarray}
\int_{\mathbb{R}^{3}}(|v|^{3}+|v|u^{2})dx &=&\int_{\mathbb{R}^{3}}(-\Delta
\phi _{u,v})|v|dx  \notag \\
&=&\int_{\mathbb{R}^{3}}\langle \nabla \phi _{u,v},\nabla |v|\rangle dx
\notag \\
&\leq &\left( \int_{\mathbb{R}^{3}}|\nabla v|^{2}dx\right) ^{\frac{1}{2}%
}\left( \int_{\mathbb{R}^{3}}|\nabla \phi _{u,v}|^{2}dx\right) ^{\frac{1}{2}}
\notag \\
&\leq &\Vert \nabla v\Vert _{2}B(u,v)^{\frac{1}{2}}.  \label{e2-29}
\end{eqnarray}%
For $(u,v)\in \mathbf{H}$ and $3\leq p<\frac{10}{3}$, by (\ref{e2-1}), (\ref%
{e2-28}) and (\ref{e2-29}), we have
\begin{eqnarray*}
&&\int_{\mathbb{R}^{3}}(|u|^{p}+|v|^{p})dx  \notag \\
&\leq & \left( \int_{\mathbb{R}^{3}}|u|^{3}dx\right) ^{10-3p}\left(
\int_{\mathbb{R}^{3}}|u|^{\frac{10}{3}}dx\right) ^{3(p-3)}+\left( \int_{%
\mathbb{R}^{3}}|v|^{3}dx\right) ^{10-3p}\left( \int_{\mathbb{R}^{3}}|v|^{%
\frac{10}{3}}dx\right) ^{3(p-3)}  \notag \\
&\leq &\left( \Vert u\Vert _{3}^{3(10-3p)}+\Vert v\Vert
_{3}^{3(10-3p)}\right) \left( \Vert u\Vert _{10/3}^{10(p-3)}+\Vert v\Vert
_{10/3}^{10(p-3)}\right)  \notag \\
&\leq &16\mathcal{S}_{10/3}\left( \Vert u\Vert _{2}^{2}+\Vert v\Vert
_{2}^{2}\right) ^{2(p-3)}A(u,v)^{\frac{3p-8}{2}}B(u,v)^{\frac{10-3p}{2}},  \label{e2-20}
\end{eqnarray*}%
which implies that%
\begin{eqnarray}
&&\frac{\frac{p}{3p-8}\left( \frac{\gamma (3p-8)}{2(10-3p)}\right) ^{\frac{%
10-3p}{2}}c^{-2(p-3)}A(u,v)^{\frac{3p-8}{2}}B(u,v)^{\frac{10-3p}{2}}\left( \Vert u\Vert _{2}^{2}+\Vert v\Vert _{2}^{2}\right)
^{2(p-3)}-\int_{\mathbb{R}^{3}}(|u|^{p}+|v|^{p})dx}{2\int_{\mathbb{R}%
^{3}}|u|^{\frac{p}{2}}|v|^{\frac{p}{2}}dx}  \notag \\
&\geq &\frac{\left[ \frac{p}{16\mathcal{S}_{10/3}(3p-8)}\left( \frac{\gamma
(3p-8)}{2(10-3p)}\right) ^{\frac{10-3p}{2}}c^{-2(p-3)}-1\right] \int_{%
\mathbb{R}^{3}}(|u|^{p}+|v|^{p})dx}{\int_{\mathbb{R}^{3}}(|u|^{p}+|v|^{p})dx}
\notag \\
&\geq &\frac{p}{16\mathcal{S}_{10/3}(3p-8)}\left( \frac{\gamma (3p-8)}{%
2(10-3p)}\right) ^{\frac{10-3p}{2}}c^{-2(p-3)}-1.  \label{e2-21}
\end{eqnarray}%
Thus we can define
\begin{eqnarray*}
\beta _{\star } &=&\beta _{\star }(c,\gamma ,p) \\
&:=&\inf_{(u,v)\in \mathbf{H}\backslash \{(0,0)\}}\left[\frac{\frac{p}{3p-8}\left(
\frac{\gamma (3p-8)}{2(10-3p)}\right) ^{\frac{10-3p}{2}}c^{-2(p-3)}A\left(
u,v\right)^{\frac{3p-8}{2}}B\left( u,v\right)^{\frac{10-3p}{2}}\left(
\Vert u\Vert _{2}^{2}+\Vert v\Vert _{2}^{2}\right) ^{2(p-3)}}{2\int_{\mathbb{R}^{3}}|u|^{\frac{p}{2}}|v|^{\frac{p%
}{2}}dx}\right. \\
&&\left.-\frac{\int_{\mathbb{R}%
^{3}}(|u|^{p}+|v|^{p})dx}{2\int_{\mathbb{R}^{3}}|u|^{\frac{p}{2}}|v|^{\frac{p%
}{2}}dx}\right]>-\infty .
\end{eqnarray*}%
In particular, for $p=3$, it follows from (\ref{e2-28}) and (\ref{e2-29})
that
\begin{equation*}
\int_{\mathbb{R}^{3}}(|u|^{3}+|v|^{3})dx\leq \sqrt{2}A(u,v)^{\frac{1}{2}}B(u,v)^{\frac{1}{2}},
\end{equation*}%
and further
\begin{eqnarray*}
\beta _{\star } &=&\inf_{(u,v)\in \mathbf{H}\backslash \{(0,0)\}}\frac{\frac{%
3\sqrt{2\gamma }}{2}A(u,v)^{\frac{1}{2}}B(u,v)^{\frac{1}{2}%
}-\int_{\mathbb{R}^{3}}(|u|^{3}+|v|^{3})dx}{2\int_{\mathbb{R}^{3}}|u|^{\frac{%
3}{2}}|v|^{\frac{3}{2}}dx} \\
&\geq &\frac{\left( \frac{3\sqrt{\gamma }}{2}-1\right) \int_{\mathbb{R}%
^{3}}(|u|^{3}+|v|^{3})dx}{2\int_{\mathbb{R}^{3}}|u|^{\frac{3}{2}}|v|^{\frac{3%
}{2}}dx} \\
&\geq &\frac{\left( \frac{3\sqrt{\gamma }}{2}-1\right) \int_{\mathbb{R}%
^{3}}(|u|^{3}+|v|^{3})dx}{\int_{\mathbb{R}^{3}}(|u|^{3}+|v|^{3})dx} \\
&=&\frac{3\sqrt{\gamma }}{2}-1.
\end{eqnarray*}

\begin{lemma}
\label{L3-6} Let $3\leq p<\frac{10}{3}$ and $\gamma >0.$ Then for $\beta
>\beta _{\star },$ it holds $\sigma _{\gamma ,\beta }(c)<0$ for all $c>0$.
\end{lemma}

\begin{proof}
Let $(u,v)\in \mathbf{H}\backslash \{(0,0)\}$ and set the scaled function%
\begin{equation}
(\bar{u}_{s},\bar{v}_{s}):=(s^{3/2}u(sx),s^{3/2}v(sx))\text{ for }s>0.
\label{e2-20}
\end{equation}%
Clearly, $(\bar{u}_{s},\bar{v}_{s})\in S_{c}.$ Define the function $f_{\star
}(s)$ by
\begin{eqnarray*}
f_{\star }(s):&=&\frac{1}{s}I_{\gamma ,\beta }(\bar{u}_{s},\bar{v}_{s}) \\
&=&\frac{s}{2}A\left( u,v\right) +\frac{\gamma }{4}B(u,v)-\frac{s^{\frac{3p-8%
}{2}}}{p}\int_{\mathbb{R}^{3}}(|u|^{p}+|v|^{p}+2\beta |u|^{\frac{p}{2}}|v|^{%
\frac{p}{2}})dx.
\end{eqnarray*}%
By a direct calculation, it is easy to see that
\begin{eqnarray*}
\min_{s\geq 0}f_{\star }(s) &=&f_{\star }(\overline{s}) \\
&=&-\frac{10-3p}{2(3p-8)}\left( \frac{3p-8}{p}\right) ^{\frac{2}{10-3p}%
}A\left( u,v\right)^{\frac{8-3p}{10-3p}}\left( \int_{\mathbb{R}%
^{3}}(|u|^{p}+|v|^{p}+2\beta |u|^{\frac{p}{2}}|v|^{\frac{p}{2}})dx\right) ^{%
\frac{2}{10-3p}} \\
&&+\frac{\gamma }{4}B(u,v),
\end{eqnarray*}%
where%
\begin{equation*}
\overline{s}:=\left( \frac{(3p-8)\int_{\mathbb{R}^{3}}(|u|^{p}+|v|^{p}+2%
\beta |u|^{\frac{p}{2}}|v|^{\frac{p}{2}})dx}{pA\left( u,v\right) }\right) ^{%
\frac{2}{3p-10}}.
\end{equation*}%
This indicates that $\min_{s\geq 0}f_{\star }(s)<0$ for $\beta >\beta
_{\star }$, and so $\sigma _{\gamma ,\beta }(c)<0$ for $\beta >\beta _{\star
}$. This completes the proof.
\end{proof}

\begin{remark}
\label{r4} According to the definition of $\beta _{\star }$ and the proof of
Lemma \ref{L3-6}, we see that $\beta _{\star }$ can be characterized by
\begin{equation*}
\beta _{\star }=\inf \left\{ \beta >0:\sigma _{\gamma ,\beta }(c)<0\right\}
\text{ for }3\leq p<\frac{10}{3}.
\end{equation*}
\end{remark}

\begin{lemma}
\label{L3-8} Let $3\leq p<\frac{10}{3}$ and $\gamma >0.$ Then for $\beta
>\beta _{\star },$ we have
\begin{equation*}
\sigma _{\gamma ,\beta }(sc)<s\sigma _{\gamma ,\beta }(c)\text{ for any }s>1,
\end{equation*}%
and
\begin{equation*}
\sigma _{\gamma ,\beta }(c)<\sigma _{\gamma ,\beta }(c^{\prime })+\sigma
_{\gamma ,\beta }(c-c^{\prime })\text{ for all }0<c^{\prime }<c.
\end{equation*}
\end{lemma}

\begin{proof}
Let $(u,v)\in S_{c}.$ By (\ref{e2-19}) one has $\Vert u^{s}\Vert
_{2}^{2}+\Vert v^{s}\Vert _{2}^{2}=s(\Vert u\Vert _{2}^{2}+\Vert v\Vert
_{2}^{2})$ and
\begin{equation*}
I_{\gamma ,\beta }\left( u^{s},v^{s}\right) =s^{3}I_{\gamma ,\beta }(u,v)+%
\frac{s^{3}-s^{2p-3}}{p}C(u,v).
\end{equation*}%
Note that $s^{3}-s^{2p-3}\leqslant 0$ for $s>1$. Then we get $I_{\gamma
,\beta }\left( u^{s},v^{s}\right) \leq s^{3}I_{\gamma ,\beta }(u,v)$,
leading to $\sigma _{\gamma ,\beta }(sc)\leq s^{3}\sigma _{\gamma ,\beta
}(c).$ Furthermore, since $\sigma _{\gamma ,\beta }(c)<0$ for $\beta >\beta
_{\star }$ by Lemma \ref{L3-6}, we have
\begin{equation*}
\sigma _{\gamma ,\beta }(sc)\leq s^{3}\sigma _{\gamma ,\beta }(c)<s\sigma
_{\gamma ,\beta }(c).
\end{equation*}%
Without loss of generality, we assume that $c-c^{\prime }>c^{\prime }$. Then
we have
\begin{equation*}
\sigma _{\gamma ,\beta }(c)<\frac{c}{c-c^{\prime }}\sigma _{\gamma ,\beta
}(c-c^{\prime })=\sigma _{\gamma ,\beta }(c-c^{\prime })+\frac{c^{\prime }}{%
c-c^{\prime }}\sigma _{\gamma ,\beta }(c-c^{\prime })<\sigma _{\gamma ,\beta
}(c-c^{\prime })+\sigma _{\gamma ,\beta }(c^{\prime }).
\end{equation*}%
This completes the proof.
\end{proof}

\textbf{Now we give the proof of Theorem \ref{t2}:} $(i)$ The proof of the
existence of the minimizer is similar to that of Theorem \ref{t1}, we omit
it here. Next we discuss the sign of the Lagrange multiplier $\tilde{\lambda}
$ associated with the minimizer $(\tilde{u},\tilde{v})\in S_{c}.$ Since $%
\sigma _{\gamma ,\beta }(c)<0$ and
\begin{equation*}
(5p-12)\sigma _{\gamma ,\beta }(c)=(5p-12)I_{\gamma ,\beta }(\tilde{u},%
\tilde{v})=2(p-3)A(\tilde{u},\tilde{v})-\frac{(3p-8)\tilde{\lambda}c}{2},
\end{equation*}%
we have $\tilde{\lambda}>0$.\newline
$(ii)$ Note that
\begin{equation}
\frac{\sqrt{\gamma }}{2}\int_{\mathbb{R}^{3}}\left( |u|^{3}+v^{2}|u|\right)
dx\leq \frac{1}{2}\int_{\mathbb{R}^{3}}|\nabla u|^{2}dx+\frac{\gamma }{8}%
\int_{\mathbb{R}^{3}}\phi _{u,v}\left( u^{2}+v^{2}\right) dx  \label{e3-16}
\end{equation}%
and
\begin{equation}
\frac{\sqrt{\gamma }}{2}\int_{\mathbb{R}^{3}}\left( u^{2}|v|+|v|^{3}\right)
dx\leq \frac{1}{2}\int_{\mathbb{R}^{3}}|\nabla v|^{2}dx+\frac{\gamma }{8}%
\int_{\mathbb{R}^{3}}\phi _{u,v}\left( u^{2}+v^{2}\right) dx  \label{e3-17}
\end{equation}%
for all $(u,v)\in \mathbf{H}$. When $p=3,$ for any $(u,v)\in S_{c},$ it
follows from (\ref{e3-16}) and (\ref{e3-17}) that
\begin{eqnarray*}
I_{\gamma ,\beta }(u,v) &=&\frac{1}{2}A(u,v)+\frac{\gamma }{4}B(u,v)-\frac{1%
}{3}\int_{\mathbb{R}^{3}}C(u,v) \\
&\geq &\frac{\sqrt{\gamma }}{2}\int_{\mathbb{R}^{3}}(|u|^{3}+|v|^{3})dx-%
\frac{1+\beta }{3}\int_{\mathbb{R}^{3}}(|u|^{3}+|v|^{3})dx \\
&\geq &\frac{3\sqrt{\gamma }-2(1+\beta )}{6}\int_{\mathbb{R}%
^{3}}(|u|^{3}+|v|^{3})dx \\
&>&0\text{ if }\beta <\frac{3\sqrt{\gamma }}{2}-1.
\end{eqnarray*}%
On the other hand, for fixed $(u,v)\in \mathbf{H}\backslash \{(0,0)\},$ by (%
\ref{e2-20}) we have
\begin{equation*}
\sigma _{\gamma ,\beta }(c)\leq I_{\gamma ,\beta }(u_{s},v_{s})=\frac{s^{2}}{%
2}A(u,v)+\frac{\gamma s}{4}B(u,v)-\frac{s^{\frac{3(p-2)}{2}}}{3}C(u,v)\text{
for }s>0,
\end{equation*}%
which implies that $\sigma _{\gamma ,\beta }(c)\leq 0$ as $s\rightarrow 0$.
So, $\sigma _{\gamma ,\beta }(c)$ has no minimizer for $\beta <\frac{3\sqrt{%
\gamma }}{2}-1$.

For $3<p<\frac{10}{3}$, we follow the idea in \cite{CW}. Suppose by
contradiction that there exist $c>0$ and $\widetilde{\beta }<\beta _{\star
}(c,\gamma ,p)$ such that $\sigma _{\gamma ,\widetilde{\beta }}(c)$ has a
minimizer. According to the definition of $\beta _{\star }(c,\gamma ,p)$, we
find that $\beta _{\star }(c,\gamma ,p)$ is continuous and decreasing on $c$%
. Thus one can choose $0<c<\widetilde{c}$ so that
\begin{equation*}
\widetilde{\beta }<\beta _{\star }(\widetilde{c},\gamma ,p)<\beta _{\star
}(c,\gamma ,p).
\end{equation*}%
On the other hand, by the characterization of $\beta _{\star }(c,\gamma ,p)$
, it follows that $\sigma _{\gamma ,\widetilde{\beta }}(c)=0$. Then by Lemma %
\ref{L3-8}, we get
\begin{equation*}
\sigma _{\gamma ,\widetilde{\beta }}(\widetilde{c})<\frac{\widetilde{c}}{c}%
\sigma _{\gamma ,\widetilde{\beta }}(c)=0,
\end{equation*}%
which implies that $\widetilde{\beta }>\beta _{\star }(\widetilde{c},\gamma
,p)$. This is a contradiction. So, $\sigma _{\gamma ,\beta }(c)$ has no
minimizer for $\beta <\beta _{\star }.$ We complete the proof.

\section{The relation between action ground states and energy ground states}

\begin{lemma}
\label{L3-9} Let $p=3$ and $\gamma ,\lambda >0$. Then for $\beta >\max
\{\beta _{\star },0\}$, system (\ref{e1-3}) has a vectorial action ground
state $(u_{\lambda },v_{\lambda })\in \mathbf{H}$.
\end{lemma}

\begin{proof}
By Theorem \ref{t2}, for $\beta >\max \{\beta _{\star },0\}$, system (\ref%
{e1-3}) admits an energy ground state $(\tilde{u},\tilde{v})\in S_{c}$ with
the corresponding Lagrange multiplier $\tilde{\lambda}>0.$ In other words, $(%
\tilde{u},\tilde{v})$ satisfies the system:%
\begin{equation}
\left\{
\begin{array}{ll}
-\Delta u+\tilde{\lambda}u+\gamma \phi _{u,v}u=|u|^{p-2}u+\beta |v|^{\frac{p%
}{2}}|u|^{\frac{p}{2}-2}u & \quad \text{in}\quad \mathbb{R}^{3}, \\
-\Delta v+\tilde{\lambda}v+\gamma \phi _{u,v}v=|v|^{p-2}v+\beta |u|^{\frac{p%
}{2}}|v|^{\frac{p}{2}-2}v & \quad \text{in}\quad \mathbb{R}^{3}.%
\end{array}%
\right.  \label{e3-19}
\end{equation}%
First of all, we claim that $(\tilde{u},\tilde{v})$ is a vectorial action
ground state of system (\ref{e3-19}). Indeed, let $(\phi _{1},\phi _{2})\in
\mathbf{H}$ be a nontrivial solution of system (\ref{e3-19}). We define%
\begin{equation*}
\tilde{c}:=\frac{c}{\Vert \phi _{1}\Vert _{2}^{2}+\Vert \phi _{2}\Vert
_{2}^{2}}\text{ and }(\tilde{\phi}_{1},\tilde{\phi}_{2}):=(\tilde{c}^{2}\phi
_{1}(\tilde{c}x),\tilde{c}^{2}\phi _{2}(\tilde{c}x))
\end{equation*}%
so that $\Vert \tilde{\phi}_{1}\Vert _{2}^{2}+\Vert \tilde{\phi}_{2}\Vert
_{2}^{2}=c$. By using the characterization of $(\tilde{u},\tilde{v})$, one
has
\begin{equation*}
-\frac{\tilde{\lambda}c}{6}=I_{\gamma ,\beta }(\tilde{u},\tilde{v})\leq
I_{\gamma ,\beta }(\tilde{\phi}_{1},\tilde{\phi}_{2})=\tilde{c}^{3}I_{\gamma
,\beta }(\phi _{1},\phi _{2})=-\frac{\tilde{\lambda}\tilde{c}^{3}}{6}(\Vert
\phi _{1}\Vert _{2}^{2}+\Vert \phi _{2}\Vert _{2}^{2})=-\frac{\tilde{\lambda}%
c}{6}\tilde{c}^{2},
\end{equation*}%
which shows that $\tilde{c}\leq 1,$ and hence $c\leq \frac{c}{\tilde{c}}%
=\Vert \phi _{1}\Vert _{2}^{2}+\Vert \phi _{2}\Vert _{2}^{2}$. Then we get
\begin{equation*}
E_{\tilde{\lambda},\gamma ,\beta }(\tilde{u},\tilde{v})=\frac{\tilde{\lambda}%
c}{3}\leq \frac{\tilde{\lambda}}{3}(\Vert \phi _{1}\Vert _{2}^{2}+\Vert \phi
_{2}\Vert _{2}^{2})=E_{\tilde{\lambda},\gamma ,\beta }(\phi _{1},\phi _{2}),
\end{equation*}%
which implies that $(\tilde{u},\tilde{v})$ is a vectorial action ground
state of system (\ref{e3-19}).

Set $(u_{\lambda },v_{\lambda }):=(\varpi ^{2}\tilde{u}(\varpi x),\varpi ^{2}%
\tilde{v}(\varpi x)),$ where $\varpi :=\sqrt{\frac{\lambda }{\tilde{\lambda}}%
}$. A direct computation shows that $(u_{\lambda },v_{\lambda })$ is a
vectorial solution of system (\ref{e1-3}). Next, we show that $(u_{\lambda
},v_{\lambda })$ is a vectorial action ground state of system (\ref{e1-3}).
Indeed, let $(u,v)\in \mathbf{H}$ be any nontrivial solution of system (\ref%
{e1-3}) and define
\begin{equation*}
(u^{\ast },v^{\ast }):=\left( \varpi ^{-2}u(x/\varpi),\varpi ^{-2}v(x/\varpi)\right) .
\end{equation*}%
Similarly, we can prove that $(u^{\ast },v^{\ast })$ is a nontrivial
solution of system (\ref{e3-19}) and hence we have
\begin{eqnarray*}
\frac{\tilde{\lambda}}{3\varpi }(\Vert u_{\lambda }\Vert _{2}^{2}+\Vert
v_{\lambda }\Vert _{2}^{2}) &=&\frac{\tilde{\lambda}}{3}(\Vert \tilde{u}%
\Vert _{2}^{2}+\Vert \tilde{v}\Vert _{2}^{2}) \\
&=&E_{\tilde{\lambda},\gamma ,\beta }(\tilde{u},\tilde{v}) \\
&\leq &E_{\tilde{\lambda},\gamma ,\beta }(u^{\ast },v^{\ast }) \\
&=&\frac{\tilde{\lambda}}{3}(\Vert u^{\ast }\Vert _{2}^{2}+\Vert v^{\ast
}\Vert _{2}^{2}) \\
&=&\frac{\tilde{\lambda}}{3\varpi }(\Vert u\Vert _{2}^{2}+\Vert v\Vert
_{2}^{2}),
\end{eqnarray*}%
which implies that $\Vert u_{\lambda }\Vert _{2}^{2}+\Vert v_{\lambda }\Vert
_{2}^{2}\leq \Vert u\Vert _{2}^{2}+\Vert v\Vert _{2}^{2}$. Then we have
\begin{equation*}
E_{\lambda ,\gamma ,\beta }(u_{\lambda },v_{\lambda })=\frac{\lambda }{3}%
(\Vert u_{\lambda }\Vert _{2}^{2}+\Vert v_{\lambda }\Vert _{2}^{2})\leq
\frac{\lambda }{3}(\Vert u\Vert _{2}^{2}+\Vert v\Vert _{2}^{2})=E_{\lambda
,\gamma ,\beta }(u,v),
\end{equation*}%
which implies that $(u_{\lambda },v_{\lambda })$ is a vectorial action
ground state of system (\ref{e1-3}). We complete the proof.
\end{proof}

Next, we give a characterization of the vectorial action ground state $%
(u_{\lambda },v_{\lambda })$ of system (\ref{e1-3}) when $p=3$.

\begin{proposition}
\label{P3-10} Let $p=3$ and $\gamma ,\lambda >0$. Then for $\beta >\max
\{\beta _{\star },0\},$ the following statements are true.\newline
$(i)$ For any vectorial action ground state $(u,v)$ of system (\ref{e1-3}),
there exists a unique $c=c(\lambda )>0$ such that $\Vert u\Vert
_{2}^{2}+\Vert v\Vert _{2}^{2}=c(\lambda )$.\newline
$(ii)$ $(u_{\lambda },v_{\lambda })$ is a vectorial action ground state of
system (\ref{e1-3}) if and only if $(u_{\lambda },v_{\lambda })$ is a
minimizer of $\sigma _{\gamma ,\beta }(c(\lambda ))$. In particular, we have
\begin{equation*}
I_{\gamma ,\beta }(u_{\lambda },v_{\lambda })=\sigma _{\gamma ,\beta
}(c(\lambda ))=\inf_{(u,v)\in S_{c(\lambda )}}I_{\gamma ,\beta }(u,v).
\end{equation*}
\end{proposition}

\begin{proof}
$(i)$ Let $(u_{1},v_{1})$ and $(u_{2},v_{2})$ be two vectorial action ground
states of system (\ref{e1-3}). Then we have
\begin{equation*}
\frac{\lambda }{3}(\Vert u_{1}\Vert _{2}^{2}+\Vert v_{1}\Vert
_{2}^{2})=I_{\gamma ,\beta }(u_{1},v_{1})=I_{\gamma ,\beta }(u_{2},v_{2})=%
\frac{\lambda }{3}(\Vert u_{2}\Vert _{2}^{2}+\Vert v_{2}\Vert _{2}^{2}),
\end{equation*}%
which implies the conclusion holds.\newline
$(ii)$ First of all, we show that if $(u_{\lambda },v_{\lambda })$ is a
vectorial action ground state of system (\ref{e1-3}), then $(u_{\lambda
},v_{\lambda })$ is a minimizer of $\sigma _{\gamma ,\beta }(c(\lambda ))$.
To this aim, suppose by contradiction that there exists $(u_{1},v_{1})\in
\mathbf{H}$ such that $\Vert u_{1}\Vert _{2}^{2}+\Vert v_{1}\Vert
_{2}^{2}=c(\lambda )$ and
\begin{equation}
\sigma _{\gamma ,\beta }(c(\lambda ))=I_{\gamma ,\beta
}(u_{1},v_{1})<I_{\gamma ,\beta }(u_{\lambda },v_{\lambda }).  \label{e3-20}
\end{equation}%
Moreover, there exists a Lagrange multiplier $\lambda _{1}>0$ such that $%
(u_{1},v_{1})$ satisfies the system:%
\begin{equation*}
\left\{
\begin{array}{ll}
-\Delta u+\lambda _{1}u+\gamma \phi _{u,v}u=|u|^{p-2}u+\beta |v|^{\frac{p}{2}%
}|u|^{\frac{p}{2}-2}u & \quad \text{in}\quad \mathbb{R}^{3}, \\
-\Delta v+\lambda _{1}v+\gamma \phi _{u,v}v=|v|^{p-2}v+\beta |u|^{\frac{p}{2}%
}|v|^{\frac{p}{2}-2}v & \quad \text{in}\quad \mathbb{R}^{3}.%
\end{array}%
\right.  \label{e3-21}
\end{equation*}%
Then by (\ref{e3-20}) and item $(i)$, we have
\begin{equation*}
-\frac{\lambda _{1}}{6}c(\lambda )=I_{\gamma ,\beta }(u_{1},v_{1})<I_{\gamma
,\beta }(u_{\lambda },v_{\lambda })=-\frac{\lambda }{6}c(\lambda ),
\end{equation*}%
which implies that $\lambda <\lambda _{1}$. Set $\varpi :=\sqrt{\frac{%
\lambda }{\lambda _{1}}}$ and $(\tilde{u}_{1},\tilde{v}_{1}):=(\varpi
^{2}u_{1}(\varpi x),\varpi ^{2}v_{1}(\varpi x))$. It is clear that $(\tilde{u%
}_{1},\tilde{v}_{1})$ is a vectorial solution of system (\ref{e1-3}). Then
we have
\begin{equation*}
\frac{\lambda }{3}c(\lambda )=E_{\lambda ,\gamma ,\beta }(u_{\lambda
},v_{\lambda })\leq E_{\lambda ,\gamma ,\beta }(\tilde{u}_{1},\tilde{v}_{1})=%
\frac{\lambda }{3}(\Vert \tilde{u}_{1}\Vert _{2}^{2}+\Vert \tilde{v}%
_{1}\Vert _{2}^{2})=\frac{\lambda \varpi }{3}(\Vert u_{1}\Vert
_{2}^{2}+\Vert v_{1}\Vert _{2}^{2})=\frac{\lambda \varpi }{3}c(\lambda ),
\end{equation*}%
which implies that $\varpi \geq 1$. This contradicts with $\lambda <\lambda
_{1}$. So, $(u_{\lambda },v_{\lambda })$ is a minimizer of $\sigma _{\gamma
,\beta }(c(\lambda ))$.

Finally we show that if $(u_{\lambda },v_{\lambda })$ is a minimizer of $%
\sigma _{\gamma ,\beta }(c(\lambda ))$, then $(u_{\lambda },v_{\lambda })$
is a vectorial action ground state of system (\ref{e1-3}). Indeed, we note
that there exists a Lagrange multiplier $\tilde{\lambda}>0$ such that $%
(u_{\lambda },v_{\lambda })$ satisfies the system
\begin{equation}
\left\{
\begin{array}{ll}
-\Delta u+\tilde{\lambda}u+\gamma \phi _{u,v}u=|u|^{p-2}u+\beta |v|^{\frac{p%
}{2}}|u|^{\frac{p}{2}-2}u & \quad \text{in}\quad \mathbb{R}^{3}, \\
-\Delta v+\tilde{\lambda}v+\gamma \phi _{u,v}v=|v|^{p-2}v+\beta |u|^{\frac{p%
}{2}}|v|^{\frac{p}{2}-2}v & \quad \text{in}\quad \mathbb{R}^{3}.%
\end{array}%
\right.  \label{e3-22}
\end{equation}%
Similar to the argument in Lemma \ref{L3-9}, one can show that $(u_{\lambda
},v_{\lambda })$ is a vectorial action ground state of system (\ref{e3-22}).
It remains to prove that $\tilde{\lambda}=\lambda $. Let $(u,v)$ be a
vectorial action ground state of system (\ref{e1-3}). It follows from item $%
(i)$ that
\begin{equation*}
-\frac{\tilde{\lambda}}{6}c(\lambda )=I_{\gamma ,\beta }(u_{\lambda
},v_{\lambda })\leq I_{\gamma ,\beta }(u,v)=-\frac{\lambda }{6}c(\lambda ),
\end{equation*}%
which implies that $\lambda \leq \tilde{\lambda}$. Set $(u_{\lambda }^{\ast
},v_{\lambda }^{\ast }):=(\varpi ^{2}u_{\lambda }(\varpi x),\varpi
^{2}u_{\lambda }(\varpi x))$, where $\varpi :=\sqrt{\frac{\lambda }{\tilde{%
\lambda}}}$. Then we have
\begin{equation*}
\frac{\lambda }{3}c(\lambda )=E_{\lambda ,\gamma ,\beta }(u,v)\leq
E_{\lambda ,\gamma ,\beta }(u_{\lambda }^{\ast },v_{\lambda }^{\ast })=\frac{%
\lambda }{3}(\Vert u_{\lambda }^{\ast }\Vert _{2}^{2}+\Vert u_{\lambda
}^{\ast }\Vert _{2}^{2})=\frac{\lambda \varpi }{3}c(\lambda ),
\end{equation*}%
which implies that $\tilde{\lambda}\leq \lambda .$ Hence, $\tilde{\lambda}%
=\lambda $. We complete the proof.
\end{proof}

\textbf{Now we give the proof of Theorem \ref{t4}:} This is an immediate
consequence of Lemma \ref{L3-9} and Proposition \ref{P3-10}.

\section{The existence of local minimizer}

For each $(u,v)\in S_{1}$ and $t>0$, we set $%
(u,v)_{t}:=(u_{t},v_{t})=(t^{3/2}u(tx),t^{3/2}v(tx))$. Then $%
c^{1/2}(u,v)_{t}\in S_{c}$. Define the fibering map $\Phi _{c,(u,v)}(t)$
given by
\begin{equation*}
\Phi _{c,(u,v)}(t):=I_{\gamma ,\beta }(c^{1/2}(u,v)_{t})=\frac{t^{2}}{2}%
cA(u,v)+\frac{\gamma t}{4}c^{2}B(u,v)-\frac{t^{\frac{3p-6}{2}}}{p}c^{\frac{p%
}{2}}C(u,v)\text{ for }t>0.
\end{equation*}%
By calculating the first and second derivatives of $\Phi _{c,(u,v)}(t)$, we
have%
\begin{equation*}
\Phi _{c,(u,v)}^{\prime }(t)=tcA(u,v)+\frac{\gamma }{4}c^{2}B(u,v)-\frac{%
3(p-2)}{2p}c^{\frac{p}{2}}C(u,v)t^{\frac{3p-8}{2}}
\end{equation*}%
and
\begin{equation*}
\Phi _{c,(u,v)}^{\prime \prime }(t)=cA(u,v)-\frac{3(p-2)(3p-8)}{4p}c^{\frac{p%
}{2}}C(u,v)t^{\frac{3p-10}{2}}.
\end{equation*}%
Notice that $\frac{d}{dt}I_{\gamma ,\beta }(c^{1/2}(u,v)_{t})=\frac{d}{dt}%
\Phi _{c,(u,v)}(t)=\frac{\mathcal{P}_{\gamma ,\beta }((u,v)_{t})}{t},$ where
$\mathcal{P}_{\gamma ,\beta }(u,v)$ corresponds to a Pohozaev type identity (%
\ref{e2-4}). We also recognize that for any $(u,v)\in S_{1}$, the dilated
function $c^{1/2}(u,v)_{t}$ belongs to the constraint manifold $\mathcal{M}%
_{\gamma ,\beta }(c)$ if and only if $t\in \mathbb{R}$ is a critical value
of the fibering map $t\in (0,\infty )\mapsto \Phi _{c,(u,v)}(t)$, namely, $%
\Phi _{c,(u,v)}^{\prime }(t)=0$. Thus, following the idea of Soave \cite{S2}
(or Siciliano and Silva \cite{SS}), it is natural to split $\mathcal{M}%
_{\gamma ,\beta }(c)$ into three parts corresponding to local minima, points
of inflection and local maxima. Thus we define
\begin{eqnarray*}
\mathcal{M}_{\gamma ,\beta }^{+}(c) &:&=\left\{ c^{1/2}(u,v)\in S_{c}:\Phi
_{c,(u,v)}^{\prime }(1)=0,\Phi _{c,(u,v)}^{\prime \prime }(1)>0\right\} ; \\
\mathcal{M}_{\gamma ,\beta }^{-}(c) &:&=\left\{ c^{1/2}(u,v)\in S_{c}:\Phi
_{c,(u,v)}^{\prime }(1)=0,\Phi _{c,(u,v)}^{\prime \prime }(1)<0\right\} ; \\
\mathcal{M}_{\gamma ,\beta }^{0}(c) &:&=\left\{ c^{1/2}(u,v)\in S_{c}:\Phi
_{c,(u,v)}^{\prime }(1)=0,\Phi _{c,(u,v)}^{\prime \prime }(1)=0\right\} .
\end{eqnarray*}

Furthermore, we have the following lemma.

\begin{lemma}
\label{L4-1} If $\mathcal{M}_{\gamma ,\beta }^{0}(c)=\emptyset $, then $%
\mathcal{M}_{\gamma ,\beta }(c)$ is a smooth submanifold in $S_{c}$ with
codimension $1$, and it is a natural constraint in $\mathbf{H}$.
\end{lemma}

For $(u,v)\in \mathbf{H},$ by (\ref{e2-2}), the Cauchy-Schwarz inequality
and the semigroup property of the Riesz potential, we have
\begin{eqnarray}
&&B(u,v)  \notag \\
&=&\int_{\mathbb{R}^{3}}\phi _{u,v}(u^{2}+v^{2})dx  \notag \\
&=&\int_{\mathbb{R}^{3}}\int_{\mathbb{R}^{3}}\frac{u^{2}(x)u^{2}(y)}{|x-y|}%
dxdy+\int_{\mathbb{R}^{3}}\int_{\mathbb{R}^{3}}\frac{v^{2}(x)v^{2}(y)}{|x-y|}%
dxdy+2\int_{\mathbb{R}^{3}}\int_{\mathbb{R}^{3}}\frac{u^{2}(x)v^{2}(y)}{|x-y|%
}dxdy  \notag \\
&\leq &\mathcal{B}\left( \Vert u\Vert _{2}^{3}\Vert \nabla u\Vert _{2}+\Vert
v\Vert _{2}^{3}\Vert \nabla v\Vert _{2}\right) +2\left( \int_{\mathbb{R}%
^{3}}\int_{\mathbb{R}^{3}}\frac{u^{2}(x)u^{2}(y)}{|x-y|}dxdy\right) ^{\frac{1%
}{2}}\left( \int_{\mathbb{R}^{3}}\int_{\mathbb{R}^{3}}\frac{v^{2}(x)v^{2}(y)%
}{|x-y|}dxdy\right) ^{\frac{1}{2}}  \notag \\
&\leq &\mathcal{B}\left( \Vert u\Vert _{2}^{3}+\Vert v\Vert _{2}^{3}\right)
\left( \Vert \nabla u\Vert _{2}+\Vert \nabla v\Vert _{2}\right) +2\mathcal{B}%
\Vert \nabla u\Vert _{2}^{\frac{1}{2}}\Vert u\Vert _{2}^{\frac{3}{2}}\Vert
\nabla v\Vert _{2}^{\frac{1}{2}}\Vert v\Vert _{2}^{\frac{3}{2}}  \notag \\
&\leq &4\mathcal{B}\left( \Vert u\Vert _{2}^{2}+\Vert v\Vert _{2}^{2}\right)
^{\frac{3}{2}}A(u,v)^{\frac{1}{2}}+\frac{\mathcal{B}}{2}%
\left( \Vert u\Vert _{2}^{2}+\Vert v\Vert _{2}^{2}\right) ^{\frac{3}{2}%
}A(u,v)^{\frac{1}{2}}  \notag \\
&\leq &\frac{9\mathcal{B}}{2}\left( \Vert u\Vert _{2}^{2}+\Vert v\Vert
_{2}^{2}\right) ^{\frac{3}{2}}A(u,v)^{\frac{1}{2}}.
\label{e4-9}
\end{eqnarray}%
It follows from (\ref{e2-3}) and (\ref{e4-9}) that there exists a constant $K(\gamma
,\beta ,p)>0$ such that%
\begin{eqnarray*}
\inf_{(u,v)\in S_{1}}\mathcal{R}_{p}(u,v) &=&\inf_{(u,v)\in S_{1}}\frac{%
A(u,v)^{\frac{3p-8}{4(p-3)}}(\gamma B(u,v))^{\frac{10-3p}{4(p-3)}}}{C(u,v)^{%
\frac{1}{2(p-3)}}} \\
&\geq &K(\gamma ,\beta ,p)\text{ for }\frac{10}{3}\leq p<6,
\end{eqnarray*}%
which implies that $c_{\ast }$ and $c^{\ast }$ defined as (\ref{e1-7}) are both positive. Then
we have the following two results.

\begin{lemma}
\label{L4-2} Let $\gamma ,\beta >0$. The following statements are true.%
\newline
$(i)$ If $p=10/3$, then $\mathcal{M}_{\gamma ,\beta }(c)=\mathcal{M}_{\gamma
,\beta }^{-}(c)\neq \emptyset $ for $c>c_{\ast }$ while $\mathcal{M}_{\gamma
,\beta }(c)=\emptyset $ for $0<c<c_{\ast }$.\newline
$(ii)$ If $10/3<p<6$, then $\mathcal{M}_{\gamma ,\beta }(c)=\mathcal{M}%
_{\gamma ,\beta }^{-}(c)\neq \emptyset $.
\end{lemma}

\begin{proof}
$(i)$ Since $c>c_{\ast }$, there exists $(u,v)\in S_{1}$ such that
\begin{equation*}
\frac{c}{2}A(u,v)<\frac{3c^{5/3}}{10}C(u,v),
\end{equation*}%
and so $\varphi _{c,(u,v)}$ has only one critical point at $t_{c}^{-}(u,v)$
which is a global maximum with $\varphi _{c,(u,v)}^{\prime \prime
}(t_{c}^{-}(u,v))<0$. This implies that $\mathcal{M}_{\gamma ,\beta }(c)=%
\mathcal{M}_{\gamma ,\beta }^{-}(c)\neq \emptyset $. While if $0<c<c_{\ast }$%
, then for any $(u,v)\in S_{1}$, we have
\begin{equation*}
\frac{c}{2}A(u,v)-\frac{3c^{5/3}}{10}C(u,v)\geq 0,
\end{equation*}%
and $\varphi _{c,(u,v)}$ is strictly increasing and has no critical points.%
\newline
$(ii)$ The proof is similar to that of item $(ii)$, we omit it here.
\end{proof}

\begin{lemma}
\label{L4-3} Let $\gamma ,\beta >0$. The following statements are true.%
\newline
$(i)$ The functional $I_{\gamma ,\beta }$ is bounded from below by a
positive constant on $\mathcal{M}_{\gamma ,\beta }(c)=\mathcal{M}_{\gamma
,\beta }^{-}(c)$ for either $p=\frac{10}{3}$ and $c>c_{\ast },$ or $\frac{10%
}{3}<p<6$ and $c>0.$\newline
$(ii)$ For every minimizing sequence $\left\{ u_{n},v_{n}\right\} \subset
\mathcal{M}_{\gamma ,\beta }(c)$ of $m_{\gamma ,\beta }(c)$, there holds $%
D_{1}\leq A(u_{n},v_{n})\leq D_{2}$ for some $D_{1},D_{2}>0.$
\end{lemma}

\begin{proof}
$(i)$ For $(u,v)\in \mathcal{M}_{\gamma ,\beta }(c)$, we have
\begin{equation}
A(u,v)+\frac{\gamma }{4}B(u,v)-\frac{3(p-2)}{2p}C(u,v)=0.  \label{e4-4}
\end{equation}%
When $p=\frac{10}{3}$, for $c>c_{\ast }$ it follows from (\ref{e4-4}) that
\begin{equation*}
I_{\gamma ,\beta }(u,v)=\frac{3}{10}C(u,v)-\frac{1}{2}A(u,v)>0.
\end{equation*}%
When $\frac{10}{3}<p<6,$ by (\ref{e2-3}), one has
\begin{equation*}
A(u,v)\geq \left[ \frac{p}{3(2+\beta )(p-2)\mathcal{S}_{p}}\right] ^{\frac{4%
}{3p-10}}c^{-\frac{6-p}{3p-10}},
\end{equation*}%
and together with (\ref{e4-4}), leading to%
\begin{eqnarray*}
I_{\gamma ,\beta }(u,v) &=&\frac{3p-10}{6(p-2)}A(u,v)+\frac{3p-8}{12(p-2)}%
\gamma B(u,v)  \notag \\
&\geq &\frac{3p-10}{6(p-2)}A(u,v)  \notag \\
&>&\frac{3p-10}{6(p-2)}\left[ \frac{p}{3(2+\beta )(p-2)\mathcal{S}_{p}}%
\right] ^{\frac{4}{3p-10}}c^{-\frac{6-p}{3p-10}}>0.  \label{e4-5}
\end{eqnarray*}%
\newline
$(ii)$ Let $\left\{ (u_{n},v_{n})\right\} \subset \mathcal{M}_{\gamma ,\beta
}(c)$ be a minimizing sequence of $m_{\gamma ,\beta }(c)$. For $\frac{10}{3}%
<p<6$, $\left\{ (u_{n},v_{n})\right\} $ is bounded in $\mathbf{H},$ since $%
I_{\gamma ,\beta }$ is coercive on $\mathcal{M}_{\gamma ,\beta }(c)$. For $p=%
\frac{10}{3}$, we assume on the contrary that $A(u_{n},v_{n})\rightarrow
+\infty $ as $n\rightarrow \infty $. Since $\mathcal{P}_{\gamma ,\beta
}(u_{n},v_{n})=o_{n}(1)$, we have
\begin{equation}
m_{\gamma ,\beta }(c)=I_{\gamma ,\beta }(u_{n},v_{n})=\frac{1}{8}%
B(u_{n},v_{n}),  \label{e4-6}
\end{equation}%
and
\begin{equation*}
\lim_{n\rightarrow \infty }\frac{\frac{3}{5}C(u_{n},v_{n})}{A(u_{n},v_{n})}%
=1.
\end{equation*}%
Set $\varepsilon _{n}:=[A(u_{n},v_{n})]^{-1/2}$ and $(\tilde{u}_{n}(x),%
\tilde{v}_{n}(x)):=\left(\varepsilon _{n}^{3/2}u_{n}(\varepsilon
_{n}x),\varepsilon _{n}^{3/2}v_{n}(\varepsilon _{n}x)\right)$. It is clear that $%
\varepsilon _{n}\rightarrow 0$ as $n\rightarrow \infty $ and $(\tilde{u}%
_{n}(x),\tilde{v}_{n}(x))\in S_{c}.$ Moreover, we have
\begin{equation}
A(\tilde{u}_{n},\tilde{v}_{n})=1\quad \text{and }C(\tilde{u}_{n},\tilde{v}%
_{n})\rightarrow \frac{5}{3}\text{ as }n\rightarrow \infty .  \label{e4-7}
\end{equation}%
Then $\left\{ (\tilde{u}_{n},\tilde{v}_{n})\right\} $ is bounded in $\mathbf{%
H}$, and from (\ref{e4-6}) it follows that
\begin{equation}
B(\tilde{u}_{n},\tilde{v}_{n})=\varepsilon _{n}B(u_{n},v_{n})\rightarrow 0%
\text{ as }n\rightarrow \infty .  \label{e4-8}
\end{equation}%
Let
\begin{equation*}
\delta :=\lim_{n\rightarrow \infty }\sup_{y\in \mathbb{R}^{3}}%
\int_{B_{R}(y)}(\tilde{u}_{n}^{2}+\tilde{v}_{n}^{2})dx\geq 0.
\end{equation*}%
If $\delta =0$, then it follows from Vanishing lemma in \cite[Lemma I.1]{L1}
that $\left\{ (\tilde{u}_{n},\tilde{v}_{n})\right\} \rightarrow (0,0)$ in $%
L^{\frac{10}{3}}(\mathbb{R}^{3})\times L^{\frac{10}{3}}(\mathbb{R}^{3})$,
which contradicts with (\ref{e4-7}). So $\delta >0$ and there exists a
sequence $\left\{ y_{n}\right\} \subset \mathbb{R}^{3}$ such that
\begin{equation*}
\int_{B_{1}(y_{n})}(\tilde{u}_{n}^{2}+\tilde{v}_{n}^{2})dx\geq \frac{\delta
}{2}>0.
\end{equation*}%
Let $(u_{n}^{\ast },v_{n}^{\ast }):=(\tilde{u}_{n}(x+y_{n}),\tilde{v}%
_{n}(x+y_{n}))$. Then we have $A(u_{n}^{\ast },v_{n}^{\ast })=A(\tilde{u}%
_{n},\tilde{v}_{n})$ and
\begin{equation*}
\int_{B_{1}(0)}((u_{n}^{\ast })^{2}+(u_{n}^{\ast })^{2})dx\geq \frac{\delta
}{2},
\end{equation*}%
which implies that $(u_{n}^{\ast },v_{n}^{\ast })\rightharpoonup (u^{\ast
},v^{\ast })\neq (0,0)$ in $\mathbf{H}$. By the Fatou's lemma and (\ref{e4-8}%
), we obtain that
\begin{equation*}
0<B(u^{\ast },v^{\ast })\leq \lim \inf_{n\rightarrow \infty }B(u_{n}^{\ast
},v_{n}^{\ast })=\lim \inf_{n\rightarrow \infty }B(\tilde{u}_{n},\tilde{v}%
_{n})=0,
\end{equation*}%
which is impossible and so $\left\{ (u_{n},v_{n})\right\} $ is bounded in $%
\mathbf{H}$. Moreover, it is easy to prove that the sequence $\left\{
(u_{n},v_{n})\right\} $ can not vanish. In fact, if $\left\{
(u_{n},v_{n})\right\} $ can vanish, then $I_{\gamma ,\beta
}(u_{n},v_{n})\rightarrow 0$ as $n\rightarrow \infty $, which contradicts
with $m_{\gamma ,\beta }(c)>0$. We complete the proof.
\end{proof}

\begin{lemma}
\label{L4-4} (\cite[Proposition 3.1]{SS}) Assume that $b\neq 0$, $dl-bg\neq
0 $, $(be-al)/(dl-bg)>0$, $(ag-de)(dl-bg)>0$, $A,B,G>0$ and $p\in
(2,6)\backslash \left\{ 3\right\} $. Then the following system
\begin{equation*}
\left\{
\begin{array}{ll}
aAt+bBr+dGr^{\frac{p-2}{2}}t^{\frac{3p-8}{2}}=0 &  \\
eAt+lBr+gGr^{\frac{p-2}{2}}t^{\frac{3p-8}{2}}=0 &
\end{array}%
\right.
\end{equation*}%
admits a unique positive solution $(r,t)$. Moreover, we have
\begin{equation*}
r=\left( \frac{be-al}{dl-bg}\right) ^{\frac{1}{2(p-3)}}\left( \frac{ag-de}{%
dl-bg}\right) ^{\frac{3p-10}{4(p-3)}}\frac{A^{\frac{3p-8}{4(p-3)}}B^{\frac{%
3p-10}{4(p-3)}}}{G^{\frac{1}{2(p-2)}}}.
\end{equation*}
\end{lemma}

Fix $(u,v)\in S_{1}$ and consider the map $\hat{f}(c):=\Phi _{c,(u,v)}\left( t_{c}^{-}(u,v)\right),$ where $t_{c}^{-}(u,v)$ is the unique critical point of $\Phi _{c,(u,v)}$.
For simplicity, we denote $t_{c}=t_{c}^{-}(u,v)$. By Lemma \ref{L4-2}, the
map $\hat{f}$ is well-defined and
\begin{equation*}
\hat{f}^{\prime }(c)=\Phi _{c,(u,v)}^{\prime }\left( t_{c}\right) +\frac{1}{2%
}\left[ (t_{c})^{2}A(u,v)+ct_{c}(u)\gamma B(u,v)-c^{\frac{p-2}{2}}(t_{c})^{%
\frac{3p-6}{2}}C(u,v)\right] .
\end{equation*}%
Thus $\hat{f}^{\prime }(c)=0$ if and only if
\begin{equation}
\left\{
\begin{array}{l}
A(u,v)t_{c}+\frac{\gamma }{4}B(u,v)c-\frac{3(p-2)}{2p}C(u,v)c^{\frac{p-2}{2}%
}(t_{c})^{\frac{3p-8}{2}}=0 \\
A(u,v)t_{c}+\gamma B(u,v)c-C(u,v)c^{\frac{p-2}{2}}(t_{c})^{\frac{3p-8}{2}}=0.%
\end{array}%
\right.  \label{e4-1}
\end{equation}%
By Lemma \ref{L4-4}, system (\ref{e4-1}) has a unique positive solution $%
(c^{\ast },t^{\ast })$, where $c^{\ast }$ is as (\ref{e1-7}).
In particular, when $p=10/3$, it is clear that
\begin{equation*}
c^{\ast }=\left( \frac{15}{7}\right) ^{\frac{3}{2}}\inf_{(u,v)\in S_{1}}%
\mathcal{R}_{10/3}(u,v)=\left( \frac{9}{7}\right) ^{\frac{3}{2}}c_{\ast }.
\end{equation*}%

\begin{lemma}
\label{L4-5} Let $\gamma ,\beta >0$. For $(u,v)\in S_{1}$, the following
statements are true.\newline
$(i)$ If $p=\frac{10}{3}$, then the function $\hat{f}(c)$ is decreasing on $%
c\in (c_{\ast },c^{\ast }).$\newline
$(ii)$ If $\frac{10}{3}<p<6$, then the function $\hat{f}(c)$ is decreasing
on $c\in (0,c^{\ast }).$
\end{lemma}

\begin{proof}
For $(u,v)\in S_{1}$, by system (\ref{e4-1}), we have
\begin{equation}
t_{c}^{2}A(u,v)+ct_{c}\gamma B(u,v)-c^{\frac{p-2}{2}}t_{c}^{\frac{3p-6}{2}%
}C(u,v)=t_{c}^{2}h(c),  \label{e4-2}
\end{equation}%
where
\begin{equation*}
h(c):=-3A(u,v)+\frac{5p-12}{p}c^{\frac{p-2}{2}}t_{c}^{\frac{3p-10}{2}}C(u,v).
\end{equation*}%
According to (\ref{e4-2}), we note that $2\hat{f}^{\prime }(c)=t_{c}^{2}h(c)$%
, which implies that the monotonicity of $\hat{f}(c)$ depends on the sign of
$h(c)$.\newline
$(i)$ When $p=\frac{10}{3}$, we have $h(c)<0$ for $c_{\ast }<c<c^{\ast }$.
This shows that the function $\hat{f}(c)$ is decreasing on $c\in (c_{\ast
},c^{\ast }).$\newline
$(ii)$ When $\frac{10}{3}<p<6,$ since $t_{c}$ is continuous and $\hat{f}$
has a unique critical point, it is sufficient to show that there exists $%
0<c_{1}<c^{\ast }$ such that $h\left( c_{1}\right) <0$. Now we claim that
\begin{equation}
\lim_{c\rightarrow 0}h(c)<0.  \label{e4-3}
\end{equation}%
Indeed, if $t_{c}$ is bounded from above as $c\rightarrow 0$, then (\ref%
{e4-3}) is obvious. If $t_{c}\rightarrow \infty $ as $c\rightarrow 0$, by (%
\ref{e4-1}), we have
\begin{equation*}
A(u,v)-\frac{3(p-2)}{2p}c^{\frac{p-2}{2}}t_{c}^{\frac{3p-10}{2}%
}C(u,v)=o_{c}(1),
\end{equation*}%
which shows that
\begin{eqnarray*}
h(c) &=&-3A(u,v)+\frac{5p-12}{p}c^{\frac{p-2}{2}}t_{c}^{\frac{3p-10}{2}%
}C(u,v) \\
&=&-3A(u,v)+\frac{5p-12}{p}\frac{2p}{3(p-2)}A(u,v)+o_{c}(1) \\
&=&\frac{p-6}{3p-6}A(u,v)+o_{c}(1) \\
&<&0\text{ as }c\rightarrow 0.
\end{eqnarray*}%
So, the claim holds. This implies that there exists $0<c_{1}<c^{\ast }$ such
that $h\left( c_{1}\right) <0.$ The proof is complete.
\end{proof}

\begin{lemma}
\label{L4-6} Let $\gamma ,\beta >0$. The following statements are true.%
\newline
$(i)$ If $\frac{10}{3}<p<6$, then the function $m_{\gamma ,\beta }(c)$ is
decreasing on $c\in (0,c^{\ast }).$\newline
$(ii)$ If $p=\frac{10}{3}$, then the function $m_{\gamma ,\beta }(c)$ is
decreasing on $c\in \left( c_{\ast },c^{\ast }\right) $.
\end{lemma}

\begin{proof}
$(i)$ Fix $0<\rho <\tau <c^{\ast }$ and let $\left\{ (u_{n},v_{n})\right\}
\subset \mathcal{M}_{\gamma ,\beta }^{-}(\rho )$ be a minimizing sequence to
$m_{\gamma ,\beta }(\rho )$. By Lemma \ref{L4-5}, we have $^{\prime
}(c)<0$ on $c\in (0,c^{\ast })$. Next, we further prove that $\hat{f}%
^{\prime }(c)<-D$ for some $D>0$. Suppose on the contrary, then there exists
a sequence $\left\{ (\bar{u}_{n},\bar{v}_{n})\right\} \subset S_{1}$
satisfying $\Vert \nabla \bar{u}_{n}\Vert _{2}^{2}+\Vert \nabla \bar{v}%
_{n}\Vert _{2}^{2}=1$ and corresponding sequences $\left\{ t_{n}\right\}
\subset (0,+\infty ),\left\{ c_{n}\right\} \subset \lbrack \rho ,\tau ]$
such that
\begin{equation*}
\left\{
\begin{array}{l}
t_{n}A(\bar{u}_{n},\bar{v}_{n})+\frac{c_{n}}{4}\gamma B(\bar{u}_{n},\bar{v}%
_{n})-\frac{3(p-2)}{2p}c_{n}^{\frac{p-2}{2}}t_{n}^{\frac{3p-8}{2}}C(\bar{u}%
_{n},\bar{v}_{n})=o_{n}(1) \\
t_{n}A(\bar{u}_{n},\bar{v}_{n})+c_{n}\gamma B(\bar{u}_{n},\bar{v}%
_{n})-c_{n}^{\frac{p-2}{2}}t_{n}^{\frac{3p-8}{2}}C(\bar{u}_{n},\bar{v}%
_{n})=o_{n}(1).%
\end{array}%
\right.
\end{equation*}%
By Lemma \ref{L4-3}, we conclude that $\left\{ t_{n}\right\} $ is bounded
and
\begin{equation*}
c_{n}=\left[ \frac{2(6-p)}{5p-12}\right] ^{\frac{3p-10}{4(p-3)}}\left( \frac{%
3p}{5p-12}\right) ^{\frac{1}{2(p-3)}}R_{p}(u_{n},v_{n})\geq c^{\ast },
\end{equation*}%
which is a contradiction. Thus, the claim is proved. So by the mean value
theorem, we have
\begin{eqnarray*}
m_{\gamma ,\beta }(\tau ) &\leq &\Phi _{\tau ,(u_{n},v_{n})}\left( t_{\tau
}^{-}\left( u_{n},v_{n}\right) \right)=\Phi _{\rho ,(u_{n},v_{n})}\left( t_{\rho }^{-}\left( u_{n},v_{n}\right)
\right) +\hat{f}^{\prime }\left( \xi \right) \left( \tau -\rho \right) \\
&<&m_{\gamma ,\beta }(\rho )-D(\tau -\rho ) \\
&<&m_{\gamma ,\beta }(\rho ),
\end{eqnarray*}%
where $\xi \in \left( \rho ,\tau \right) $. This shows that the function $%
m_{\gamma ,\beta }(c)$ is decreasing on $c\in (0,c^{\ast }).$\newline
$(ii)$ The proof is similar to that of $(i)$, we omit it here.
\end{proof}

Next, we recall the so-called profile decomposition of bounded sequences, proposed by G\'{e}rard in \cite{G}, which is crucial to recover the compactness.

\begin{proposition}\label{P3-18}
Let $\left\{(u_{n},v_{n})\right\}$ be bounded in $\mathbf{H}$. Then there are sequences $\left\{(\bar{u}_{i},\bar{v}_{i})\right\}_{i=0}^{\infty}\subset \mathbf{H}$, $\left\{y_{n}^{i}\right\}_{i=0}^{\infty}\subset \mathbb{R}^{3}$ for any $n\geq 1$, such that $y_{n}^{0}=0$, $|y_{n}^{i}-y_{n}^{j}|\rightarrow \infty$ as $n\rightarrow \infty$ for $i\neq j$, and passing to a subsequence, the following conclusions hold for any $i\geq 0$:
\begin{equation}\label{e7-1}
(u_{n}(\cdot+y_{n}^{i}),v_{n}(\cdot+y_{n}^{i}))\rightharpoonup (\bar{u}_{i},\bar{v}_{i}) \quad \text{ as }\, n\rightarrow\infty,
\end{equation}
with $\lim\sup_{n\rightarrow\infty}\int_{\mathbb{R}^{3}}(|z_{1,n}^{i}|^{q}+|z_{2,n}^{i}|^{q})dx\rightarrow 0$ as $i\rightarrow\infty$, where $2<q<6$, $z_{1,n}^{i}:=u_{n}-\sum\limits_{j=0}^{i}\bar{u}_{j}(\cdot-y_{n}^{j})$ and $z_{2,n}^{i}:=v_{n}-\sum\limits_{j=0}^{i}\bar{v}_{j}(\cdot-y_{n}^{j})$. Moreover, we have
\begin{equation}\label{e7-2}
\lim_{n\rightarrow\infty}\int_{\mathbb{R}^{3}}|\nabla u_{n}|^{2}dx=\sum\limits_{j=0}^{i}\int_{\mathbb{R}^{3}}|\nabla \bar{u}_{j}|^{2}dx+\lim_{n\rightarrow\infty}\int_{\mathbb{R}^{3}}|\nabla z_{1,n}^{i}|^{2}dx,\,
\end{equation}
\begin{equation}\label{e7-3}
\lim_{n\rightarrow\infty}\int_{\mathbb{R}^{3}}|\nabla v_{n}|^{2}dx=\sum\limits_{j=0}^{i}\int_{\mathbb{R}^{3}}|\nabla \bar{v}_{j}|^{2}dx+\lim_{n\rightarrow\infty}\int_{\mathbb{R}^{3}}|\nabla z_{2,n}^{i}|^{2}dx,
\end{equation}
\begin{equation}\label{e7-4}
\limsup_{n\rightarrow\infty}\int_{\mathbb{R}^{3}}(|u_{n}|^{p}+|v_{n}|^{p}+2\beta|u_{n}|^{\frac{p}{2}}|v_{n}|
^{\frac{p}{2}})dx=\sum\limits_{j=0}^{\infty}\int_{\mathbb{R}^{3}}(|\bar{u}_{j}|^{p}+|\bar{v}_{j}|^{p}
+2\beta|\bar{u}_{j}|^{\frac{p}{2}}|\bar{v}_{j}|
^{\frac{p}{2}})dx,
\end{equation}
and
\begin{equation}\label{e7-5}
\limsup_{n\rightarrow\infty}\int_{\mathbb{R}^{3}}\phi_{u_{n},v_{n}}(u_{n}^{2}+v_{n}^{2})dx=
\sum\limits_{j=0}^{\infty}\int_{\mathbb{R}^{3}}\phi_{\bar{u}_{j},\bar{v}_{j}}(\bar{u}_{j}^{2}+\bar{v}_{j}^{2})dx.
\end{equation}
\end{proposition}

\begin{proof}
The proof is based on Vanishing Lemma  \cite[Lemma I.1]{L1} and Brezis-Lieb type results. For the proofs of (\ref{e7-1})-(\ref{e7-4}), we refer the reader to \cite{G,HK,MS,Z}. Here, we only need to prove the nonlocal coupled term (\ref{e7-5}) holds. Without loss of generality, we may assume that $\bar{u}_{i},\bar{v}_{i}$ is continuous and compactly supported. For every $i\neq j$, the sequence $\left\{y_{n}^{i}\right\}_{i=0}^{\infty}$ satisfies $|y_{n}^{i}-y_{n}^{j}|\rightarrow \infty$ as $n\rightarrow \infty$, then we call this sequence $\left\{y_{n}^{i}\right\}_{i=0}^{\infty}\subset\mathbb{R}^{3}$ satisfying the orthogonality condition.

Since the sequence $\left\{(u_{n},v_{n})\right\}$ can be written, up to a subsequence, as
\begin{equation*}
(u_{n},v_{n})=\left(\sum\limits_{j=0}^{i}\bar{u}_{j}(\cdot-y_{n}^{j})+z_{1,n}^{i},
\sum\limits_{j=0}^{i}\bar{v}_{j}(\cdot-y_{n}^{j})+z_{2,n}^{i}\right),
\end{equation*}
and $\limsup_{n\rightarrow\infty}\int_{\mathbb{R}^{3}}(|z_{1,n}^{i}|^{q}+|z_{2,n}^{i}|^{q})dx\rightarrow 0$ as $i\rightarrow\infty$, we claim that
\begin{equation}\label{e7-10}
B(u_{n},v_{n})\rightarrow B\left(\sum\limits_{j=0}^{i}\bar{u}_{j}(\cdot-y_{n}^{j}),\sum\limits_{j=0}^{i}\bar{v}_{j}(\cdot-y_{n}^{j})\right)\ \mbox{as}\ n\rightarrow\infty\ \mbox{and}\ i\rightarrow\infty.
\end{equation}
Taking $\bar{u}_{j,n}=\bar{u}_{j}(\cdot-y_{n}^{j})$ and $\bar{v}_{j,n}=\bar{v}_{j}(\cdot-y_{n}^{j})$, to prove (\ref{e7-10}) is equivalent to prove that
\begin{equation}\label{e7-11}
\int_{\mathbb{R}^{3}}\int_{\mathbb{R}^{3}}\frac{|u_{n}|^{2}|u_{n}|^{2}}{|x-y|}dxdy\rightarrow \int_{\mathbb{R}^{3}}\int_{\mathbb{R}^{3}}\frac{|\sum\limits_{j=0}^{i}\bar{u}_{j,n}|^{2}
|\sum\limits_{j=0}^{i}\bar{u}_{j,n}|^{2}}{|x-y|}dxdy,
\end{equation}
\begin{equation}\label{e7-12}
\int_{\mathbb{R}^{3}}\int_{\mathbb{R}^{3}}\frac{|v_{n}|^{2}|v_{n}|^{2}}{|x-y|}dxdy\rightarrow \int_{\mathbb{R}^{3}}\int_{\mathbb{R}^{3}}\frac{|\sum\limits_{j=0}^{i}\bar{v}_{j,n}|^{2}
|\sum\limits_{j=0}^{i}\bar{v}_{j,n}|^{2}}{|x-y|}dxdy,
\end{equation}
and
\begin{equation}\label{e7-13}
\int_{\mathbb{R}^{3}}\int_{\mathbb{R}^{3}}\frac{|u_{n}|^{2}|v_{n}|^{2}}{|x-y|}dxdy\rightarrow \int_{\mathbb{R}^{3}}\int_{\mathbb{R}^{3}}\frac{|\sum\limits_{j=0}^{i}\bar{u}_{j,n}|^{2}
|\sum\limits_{j=0}^{i}\bar{v}_{j,n}|^{2}}{|x-y|}dxdy,
\end{equation}
as $n\rightarrow\infty$ and $i\rightarrow\infty$. Indeed, to prove (\ref{e7-11}), we only need to obtain the following estimates:
\begin{equation}\label{e7-14}
\lim_{i\rightarrow\infty}\limsup_{n\rightarrow\infty}\int_{\mathbb{R}^{3}}\int_{\mathbb{R}^{3}}
\frac{|z_{1,n}^{i}|^{2}|z_{1,n}^{i}|^{2}}{|x-y|}dxdy=0,
\end{equation}
\begin{equation}\label{e7-15}
\lim_{i\rightarrow\infty}\limsup_{n\rightarrow\infty}\int_{\mathbb{R}^{3}}\int_{\mathbb{R}^{3}}
\frac{|\bar{u}_{j,n}|^{2}|z_{1,n}^{i}|^{2}}{|x-y|}dxdy=0,
\end{equation}
\begin{equation}\label{e7-16}
\lim_{i\rightarrow\infty}\limsup_{n\rightarrow\infty}\int_{\mathbb{R}^{3}}\int_{\mathbb{R}^{3}}
\frac{|\bar{u}_{j,n}|^{3}|z_{1,n}^{i}|}{|x-y|}dxdy=0,
\end{equation}
\begin{equation}\label{e7-17}
\lim_{i\rightarrow\infty}\limsup_{n\rightarrow\infty}\int_{\mathbb{R}^{3}}\int_{\mathbb{R}^{3}}
\frac{|\bar{u}_{j,n}||z_{1,n}^{i}|^{3}}{|x-y|}dxdy=0.
\end{equation}
For (\ref{e7-14}), it follows from Lemma \ref{L2-10} that
\begin{equation*}
\int_{\mathbb{R}^{3}}\int_{\mathbb{R}^{3}}
\frac{|z_{1,n}^{i}|^{2}|z_{1,n}^{i}|^{2}}{|x-y|}dxdy\leq \mathcal{C}_{HLS}\|z_{1,n}^{i}\|_{12/5}^{2}\|z_{1,n}^{i}\|_{12/5}^{2}\leq \mathcal{C}_{HLS}\|z_{1,n}^{i}\|_{12/5}^{4}.
\end{equation*}
Since $\limsup_{n\rightarrow\infty}\int_{\mathbb{R}^{3}}(|z_{1,n}^{i}|^{q}+|z_{2,n}^{i}|^{q})dx\rightarrow 0$ as $i\rightarrow\infty$, we have
$\lim_{i\rightarrow\infty}\limsup_{n\rightarrow\infty}\|z_{1,n}^{i}\|_{12/5}^{4}=0,$
and thus (\ref{e7-14}) is true. For (\ref{e7-15}), by Lemma \ref{L2-10} and Sobolev embedding theorem, we have
\begin{equation*}
\int_{\mathbb{R}^{3}}\int_{\mathbb{R}^{3}}
\frac{|\bar{u}_{j,n}|^{2}|z_{1,n}^{i}|^{2}}{|x-y|}dxdy\leq \mathcal{C}_{HLS}\|\bar{u}_{j,n}\|_{12/5}^{2}\|z_{1,n}^{i}\|_{12/5}^{2}\leq C\|\bar{u}_{j,n}\|_{H^{1}}^{2}\|z_{1,n}^{i}\|_{12/5}^{2},
\end{equation*}
which implies that (\ref{e7-15}) holds, since $\|\bar{u}_{j,n}\|_{H^{1}}$ is bounded. For (\ref{e7-16}) and (\ref{e7-17}), we have
\begin{equation*}
\int_{\mathbb{R}^{3}}\int_{\mathbb{R}^{3}}
\frac{|\bar{u}_{j,n}|^{3}|z_{1,n}^{i}|}{|x-y|}dxdy\leq \mathcal{C}_{HLS}\|\bar{u}_{j,n}\|_{12/5}^{3}\|z_{1,n}^{i}\|_{12/5}\leq C\|\bar{u}_{j,n}\|_{H^{1}}^{3}\|z_{1,n}^{i}\|_{12/5},
\end{equation*}
and
\begin{equation*}
\int_{\mathbb{R}^{3}}\int_{\mathbb{R}^{3}}
\frac{|\bar{u}_{j,n}||z_{1,n}^{i}|^{3}}{|x-y|}dxdy\leq \mathcal{C}_{HLS}\|\bar{u}_{j,n}\|_{12/5}\|z_{1,n}^{i}\|_{12/5}^{3}
\leq C\|\bar{u}_{j,n}\|_{H^{1}}\|z_{1,n}^{i}\|_{12/5}^{3}.
\end{equation*}
Then, similar to the arguments of (\ref{e7-14}) and (\ref{e7-15}), we can get (\ref{e7-16}) and (\ref{e7-17}). So (\ref{e7-11}) holds. Moreover, the proofs of (\ref{e7-12}) and (\ref{e7-13}) are similar, we omit it here.
Hence, (\ref{e7-10}) holds.

Finally, to obtain (\ref{e7-5}), it is sufficient to prove that
\begin{equation*}
B\left(\sum\limits_{j=0}^{\infty}\bar{u}_{j}(\cdot-y_{n}^{j}),\sum\limits_{j=0}^{\infty}\bar{v}_{j}(\cdot-y_{n}^{j})\right)=
\sum\limits_{j=0}^{\infty}B(\bar{u}_{j}(\cdot-y_{n}^{j}),\bar{v}_{j}(\cdot-y_{n}^{j}))+o_{n}(1).
\end{equation*}
By the pairwise orthogonality of the family $\left\{y_{n}^{i}\right\}_{i=0}^{\infty}$, and the following the elementary inequality \cite{G}
\begin{equation*}
\left|\left|\sum\limits_{j=0}^{\infty}a_{j}\right|^{2}
-\sum\limits_{j=0}^{\infty}|a_{j}|^{2}\right|\leq C\sum\limits_{j\neq k}|a_{j}||a_{k}|,
\end{equation*}
we have
\begin{eqnarray}
&&\left|\int_{\mathbb{R}^{3}}\int_{\mathbb{R}^{3}}\frac{|\sum\limits_{j=0}^{\infty}\bar{u}_{j}(x-y_{n}^{j})
|^{2}|\sum\limits_{j=0}^{\infty}\bar{u}_{j}(y-y_{n}^{j})|^{2}}{|x-y|}dxdy-\sum\limits_{j=0}^{\infty}\int_{\mathbb{R}^{3}}\int_{\mathbb{R}^{3}}
\frac{|\bar{u}_{j}(x-y_{n}^{j})|^{2}|\bar{u}_{j}(y-y_{n}^{j})|^{2}}{|x-y|}dxdy\right|\notag\\
&\leq&\sum\limits_{j=0}^{\infty}\sum\limits_{k\neq j}\int_{\mathbb{R}^{3}}\int_{\mathbb{R}^{3}}
\frac{|\bar{u}_{j}(x-y_{n}^{j})||\bar{u}_{k}(x-y_{n}^{k})||\sum\limits_{m=0}^{\infty}
\bar{u}_{m}(y-y_{n}^{m})|^{2}}{|x-y|}dxdy\label{e7-6}\\
&&+\sum\limits_{j=0}^{\infty}\sum\limits_{k\neq j}\int_{\mathbb{R}^{3}}\int_{\mathbb{R}^{3}}
\frac{|\bar{u}_{j}(y-y_{n}^{j})||\bar{u}_{k}(y-y_{n}^{k})||\sum\limits_{m=0}^{\infty}
\bar{u}_{m}(x-y_{n}^{m})|^{2}}{|x-y|}dxdy\label{e7-7}\\
&&+\sum\limits_{j=0}^{\infty}\sum\limits_{k\neq j}\int_{\mathbb{R}^{3}}\int_{\mathbb{R}^{3}}
\frac{|\bar{u}_{j}(x-y_{n}^{j})|^{2}|\bar{u}_{k}(y-y_{n}^{k})|^{2}}{|x-y|}dxdy\label{e7-8}.
\end{eqnarray}
Next, we estimate (\ref{e7-6}), (\ref{e7-7}) and (\ref{e7-8}), respectively. It follows from Lemma \ref{L2-10} and orthogonality that
\begin{equation*}
(\ref{e7-6})\leq \sum\limits_{j=0}^{\infty}\sum\limits_{k\neq j}\left\||\bar{u}_{j}(\cdot-y_{n}^{j})||\bar{u}_{k}(\cdot-y_{n}^{k})|\right\|_{6/5}
\left\|\left|\sum\limits_{m=0}^{\infty}\bar{u}_{m}(\cdot-y_{n}^{m})\right|^{2}\right\|_{6/5}\rightarrow 0\ \text{ as }\, n\rightarrow\infty.
\end{equation*}
Similarly, we get $(\ref{e7-7})\rightarrow 0$ as $n\rightarrow\infty$. For (\ref{e7-8}), by using the transformation $\tilde{x}:=x-y_{n}^{j}$ and $\tilde{y}:=y-y_{n}^{j}$,
we have
\begin{eqnarray*}
(\ref{e7-8})&=&\sum\limits_{j=0}^{\infty}\sum\limits_{k\neq j}\int_{\mathbb{R}^{3}}\int_{\mathbb{R}^{3}}
\frac{|\bar{u}_{j}(\tilde{x})|^{2}|\bar{u}_{k}(\tilde{y}-(y_{n}^{k}-y_{n}^{j}))|^{2}}{|\tilde{y}-\tilde{x}|}d\tilde{x}d\tilde{y}\\
&\leq&\sum\limits_{j=0}^{\infty}\sum\limits_{k\neq j}\|\bar{u}_{j}\|_{12/5}^{2}\|\bar{u}_{k}(\cdot-(y_{n}^{k}-y_{n}^{j}))\|_{12/5}^{2}\rightarrow 0\ \text{ as }\ n\rightarrow\infty.
\end{eqnarray*}
Hence, we get
\begin{eqnarray*}
\int_{\mathbb{R}^{3}}\int_{\mathbb{R}^{3}}\frac{|\sum\limits_{j=0}^{\infty}\bar{u}_{j}(x-y_{n}^{j})
|^{2}|\sum\limits_{j=0}^{\infty}\bar{u}_{j}(y-y_{n}^{j})|^{2}}{|x-y|}dxdy\rightarrow\sum\limits_{j=0}^{\infty}\int_{\mathbb{R}^{3}}\int_{\mathbb{R}^{3}}
\frac{|\bar{u}_{j}(x-y_{n}^{j})|^{2}|\bar{u}_{j}(y-y_{n}^{j})|^{2}}{|x-y|}dxdy
\end{eqnarray*}
as $n\rightarrow\infty.$
Similarly,
\begin{eqnarray*}
\int_{\mathbb{R}^{3}}\int_{\mathbb{R}^{3}}\frac{|\sum\limits_{j=0}^{\infty}\bar{v}_{j}(x-y_{n}^{j})
|^{2}|\sum\limits_{j=0}^{\infty}\bar{v}_{j}(y-y_{n}^{j})|^{2}}{|x-y|}dxdy\rightarrow\sum\limits_{j=0}^{\infty}\int_{\mathbb{R}^{3}}\int_{\mathbb{R}^{3}}
\frac{|\bar{v}_{j}(x-y_{n}^{j})|^{2}|\bar{v}_{j}(y-y_{n}^{j})|^{2}}{|x-y|}dxdy,
\end{eqnarray*}
and
\begin{eqnarray*}
\int_{\mathbb{R}^{3}}\int_{\mathbb{R}^{3}}\frac{|\sum\limits_{j=0}^{\infty}\bar{u}_{j}(x-y_{n}^{j})
|^{2}|\sum\limits_{j=0}^{\infty}\bar{v}_{j}(y-y_{n}^{j})|^{2}}{|x-y|}dxdy\rightarrow\sum\limits_{j=0}^{\infty}\int_{\mathbb{R}^{3}}\int_{\mathbb{R}^{3}}
\frac{|\bar{u}_{j}(x-y_{n}^{j})|^{2}|\bar{v}_{j}(y-y_{n}^{j})|^{2}}{|x-y|}dxdy
\end{eqnarray*}
as $n\rightarrow\infty.$ Therefore, (\ref{e7-5}) holds. We complete the proof.
\end{proof}

\textbf{Now we give the proofs of Theorems \ref{t5} and \ref{t6}:} Let $\left\{(u_{n},v_{n})\right\}\subset \mathcal{M}_{\gamma,\beta}^{-}(c)$ be a bounded and non-vanishing minimizing sequence to $m_{\gamma,\beta}(c)$. By applying Proposition \ref{P3-18}, we can find a profile decomposition of $\left\{(u_{n},v_{n})\right\}$ satisfying (\ref{e7-1})-(\ref{e7-5}). Define the index set
\begin{equation*}
I:=\left\{i\geq 0: (\bar{u}_{i},\bar{v}_{i})\neq (0,0)\right\}.
\end{equation*}
By using Lemma \ref{L4-3}, (\ref{e7-4}) and (\ref{e7-5}), we get that $I\neq \emptyset$. Now, we claim that
$\mathcal{P}_{\gamma,\beta}(\bar{u}_{i},\bar{v}_{i})\leq 0$ for some $i\geq 0$. Otherwise, if $\mathcal{P}_{\gamma,\beta}(\bar{u}_{i},\bar{v}_{i})>0$ for all $i\in I$. Then it follows from (\ref{e7-4}) and (\ref{e7-5}) that
\begin{eqnarray*}
\lim\sup_{n\rightarrow\infty}\left[\frac{3(p-2)}{2p}C(u_{n},v_{n})-\frac{\gamma}{4}B(u_{n},v_{n})\right]
&=&\lim\sup_{n\rightarrow\infty}A(u_{n},v_{n})\\&\geq& \sum\limits_{i=0}^{\infty}A(\bar{u}_{i},\bar{v}_{i})=\sum\limits_{i\in I}A(\bar{u}_{i},\bar{v}_{i})\\
&>&\sum\limits_{i=0}^{\infty}\left[\frac{3(p-2)}{2p}C(\bar{u}_{i},\bar{v}_{i})-\frac{\gamma}{4}
B(\bar{u}_{i},\bar{v}_{i})\right]\\
&=&\lim\sup_{n\rightarrow\infty}\left[\frac{3(p-2)}{2p}C(u_{n},v_{n})-\frac{\gamma}{4}B(u_{n},v_{n})\right],
\end{eqnarray*}
which is a contradiction.

Let $i\in I$ and, for simplicity, let us denote $(\bar{u},\bar{v}):=(\bar{u}_{i},\bar{v}_{i})\neq (0,0)$. Now we try to prove that $(\bar{u},\bar{v})\in S_{c}$. Assume on the contrary. Then $\Vert \bar{u}\Vert _{2}^{2}+\Vert \bar{v}\Vert _{2}^{2}=:\rho <c<c^{\ast}$. Moreover, we claim that $\rho>c_{\ast}$ for $p=10/3$. Indeed, since $\mathcal{P}_{\gamma,\beta}(\bar{u},\bar{v})\leq 0$, we have
\begin{equation*}
A(\bar{u},\bar{v})<A(\bar{u},\bar{v})+\frac{\gamma}{4}B(\bar{u},\bar{v})\leq \frac{3}{5}C(\bar{u},\bar{v}),
\end{equation*}
which implies that
\begin{equation*}
\frac{A(\bar{u},\bar{v})}{C(\bar{u},\bar{v})}<\frac{3}{5}.
\end{equation*}
Set
\begin{equation*}
(\check{u},\check{v}):=\frac{1}{\sqrt{\rho}}(\bar{u},\bar{v})\in S_{1}.
\end{equation*}
Then we have
\begin{eqnarray*}
\rho&>&\rho\left(\frac{5}{3}\right)^{\frac{3}{2}}
\left[\frac{A(\bar{u},\bar{v})}{C(\bar{u},\bar{v})}
\right]^{\frac{3}{2}}
=\left(\frac{5}{3}\right)^{\frac{3}{2}}
\left[\frac{A(\check{u},\check{v})}{C(\check{u},\check{v})}
\right]^{\frac{3}{2}}\\
&\geq&\left(\frac{5}{3}\right)^{\frac{3}{2}}
\inf_{(u,v)\in S_{1}}\left[\frac{A(u,v)}{C(u,v)}
\right]^{\frac{3}{2}}=c_{\ast}.
\end{eqnarray*}
Thus, the claim is true. Again using $\mathcal{P}_{\gamma,\beta}(\bar{u},\bar{v})\leq 0$, there exists $t^{-}\in(0,1]$ such that $(\bar{u},\bar{v})_{t^{-}}\in \mathcal{M}_{\gamma,\beta}^{-}(\rho)$. Then we have
\begin{eqnarray*}
m_{\gamma,\beta}(\rho)\leq I_{\gamma,\beta}((\bar{u},\bar{v})_{t^{-}})&=&\frac{3p-10}{6(p-2)}
A((\bar{u},\bar{v})_{t^{-}})+\frac{3p-8}{12(p-2)}B((\bar{u},\bar{v})_{t^{-}})\\
&=&(t^{-})^{2}\frac{3p-10}{6(p-2)}A(\bar{u},\bar{v})+(t^{-})\frac{3p-8}{12(p-2)}B(\bar{u},\bar{v})\\
&\leq&\frac{3p-10}{6(p-2)}A(\bar{u},\bar{v})+\frac{3p-8}{12(p-2)}B(\bar{u},\bar{v})\\
&\leq&\lim\inf_{n\rightarrow \infty}\left[\frac{3p-10}{6(p-2)}A(u_{n},v_{n})+\frac{3p-8}{12(p-2)}B(u_{n},v_{n})\right]\\
&=&m_{\gamma,\beta}(c),
\end{eqnarray*}
which is a contradiction. So $\rho=c$, $t^{-}=1$ and $I_{\gamma,\beta}(\bar{u},\bar{v})=m_{\gamma,\beta}(c)$. This implies that $(u_{n}(\cdot+y_{n}^{i}),v_{n}(\cdot+y_{n}^{i}))\rightarrow (\bar{u},\bar{v})$ in $\mathbf{H}$ and therefore $I$ is a singleton set.

Finally, we claim that $\bar{u}\neq 0$ and $\bar{v}\neq 0$. If not, we may
assume that $\bar{v}\equiv 0$. Then by Lemma \ref{L2-4}, there exists $%
s_{\beta }\in (0,1)$ such that $(\sqrt{s_{\beta }}\bar{u},\sqrt{1-s_{\beta }}%
\bar{u})\in S_{c}$. If $p\in \lbrack 10/3,4)$ and $\beta >0$, or $p\in
\lbrack 4,6)$ and $\beta >2^{(p-2)/2}$, then for any $t>0$, we have
\begin{eqnarray*}
I_{\gamma ,\beta }((\bar{u},0)_{t}) &=&\frac{t^{2}}{2}\Vert \nabla \bar{u}%
\Vert _{2}^{2}+\frac{\gamma t}{4}\int_{\mathbb{R}^{3}}\int_{\mathbb{R}^{3}}%
\frac{\bar{u}(x)^{2}\bar{u}(y)^{2}}{|x-y|}dxdy-\frac{t^{\frac{3(p-2)}{2}}}{p}%
\int_{\mathbb{R}^{3}}|\bar{u}|^{p}dx \\
&>&\frac{t^{2}}{2}A(\sqrt{s_{\beta }}\bar{u},\sqrt{1-s_{\beta }}\bar{u})+%
\frac{\gamma t}{4}B(\sqrt{s_{\beta }}\bar{u},\sqrt{1-s_{\beta }}\bar{u})-%
\frac{t^{\frac{3(p-2)}{2}}}{p}C(\sqrt{s_{\beta }}\bar{u},\sqrt{1-s_{\beta }}%
\bar{u}) \\
&=&I_{\gamma ,\beta }((\sqrt{s_{\beta }}\bar{u},\sqrt{1-s_{\beta }}\bar{u}%
)_{t}).
\end{eqnarray*}%
Passing to the maximum on $t>0$, then there exists $t_{\beta }^{-}>0$ such
that $(\sqrt{s_{\beta }}\bar{u},\sqrt{1-s_{\beta }}\bar{u})_{t_{\beta
}^{-}}\in \mathcal{M}_{\gamma ,\beta }^{-}(c)$ and
\begin{equation*}
I_{\gamma ,\beta }((\sqrt{s_{\beta }}\bar{u},\sqrt{1-s_{\beta }}\bar{u}%
)_{t_{\beta }^{-}})<I_{\gamma ,\beta }(\bar{u},0)=m_{\gamma ,\beta }(c),
\end{equation*}%
which is a contradiction. So, $(\bar{u},\bar{v})$ is a vectorial type
minimizer of $m_{\gamma ,\beta }(c)$.

Finally, we determine the sign of Lagrange multiplier $\bar{\lambda}$. Since
$(\bar{u},\bar{v})$ is a weak solution of system (\ref{e1-3}), we have
\begin{equation*}
A(\bar{u},\bar{v})+\bar{\lambda}c+B(\bar{u},\bar{v})-C(\bar{u},\bar{v})=0
\end{equation*}%
and
\begin{equation*}
\mathcal{P}_{\gamma ,\beta }(\bar{u},\bar{v})=A(\bar{u},\bar{v})+\frac{%
\gamma }{4}B(\bar{u},\bar{v})-\frac{3(p-2)}{2p}C(\bar{u},\bar{v})=0.
\end{equation*}%
For $p=\frac{10}{3}$, we deduce that
\begin{equation*}
\bar{\lambda}=\frac{1}{c}\left[ 3A(\bar{u},\bar{v})-\frac{7}{5}C(\bar{u},%
\bar{v})\right] >0,
\end{equation*}%
where we have used the fact of $c<c^{\ast }$. For $\frac{10}{3}<p<6$, it
follows from (\ref{e2-3}) that
\begin{eqnarray*}
A(\bar{u},\bar{v})-\frac{6(1+\beta )(p-2)\mathcal{S}_{p}}{p}c^{\frac{6-p}{4}%
}A(\bar{u},\bar{v})^{\frac{3(p-2)}{4}} &\leq &A(\bar{u},\bar{v})-\frac{3(p-2)%
}{2p}C(\bar{u},\bar{v}) \\
&=&-\frac{\gamma }{4}B(\bar{u},\bar{v}) \\
&<&0,
\end{eqnarray*}%
which implies that
\begin{equation*}
A(\bar{u},\bar{v})>\left[ \frac{6(1+\beta )(p-2)\mathcal{S}_{p}}{p}\right]
^{-\frac{4}{3p-10}}c^{-\frac{6-p}{3p-10}}.
\end{equation*}%
This shows that $A(\bar{u},\bar{v})$ is large when $c>0$ sufficiently small.
Using this fact, together with (\ref{e4-9}), we have
\begin{eqnarray*}
\bar{\lambda}c &=&\frac{6-p}{3p-6}A(\bar{u},\bar{v})-\frac{5p-12}{2(3p-6)}B(%
\bar{u},\bar{v}) \\
&\geq &\frac{6-p}{3p-6}A(\bar{u},\bar{v})-\frac{9(5p-12)\mathcal{B}}{4(3p-6)}%
c^{\frac{3}{2}}A(\bar{u},\bar{v})^{\frac{1}{2}} \\
&>&0\text{ for }c>0\text{ sufficiently small.}
\end{eqnarray*}%
So, $\bar{\lambda}>0$. The proof is complete.

\begin{proposition}
\label{P6-1}Let $(u,v)$ be the energy ground state in Theorem \ref{t5} or %
\ref{t6} with the positive Lagrange multiplier $\lambda $, then $(u,v)$ has
exponential decay:
\begin{equation*}
|u(x)|\leq Ce^{-\zeta |x|},\,|v(x)|\leq Ce^{-\zeta |x|}\text{ for every }%
x\in \mathbb{R}^{3},
\end{equation*}%
for some $C>0$, $\zeta >0$.
\end{proposition}

\begin{proof}
Following the idea in \cite{ACS}. Set $\vartheta (r):=(\vartheta
_{1}(r),\vartheta _{2}(r))$, where
\begin{equation*}
\vartheta _{1}(r)=\int_{\mathbb{R}^{3}\backslash B_{r}(0)}(|\nabla
u|^{2}+|u|^{2})dx\quad \text{and }\vartheta _{2}(r)=\int_{\mathbb{R}%
^{3}\backslash B_{r}(0)}(|\nabla v|^{2}+|v|^{2})dx.
\end{equation*}%
Let $\chi :\mathbb{R}^{3}\rightarrow \mathbb{R}$ be given by
\begin{equation*}
\chi (r):=\left\{
\begin{array}{ll}
0\quad & \text{if }r\leq 0, \\
r\quad & \text{if }0<r<1, \\
1\quad & \text{if }r\geq 1.%
\end{array}%
\right.
\end{equation*}%
Set $u^{r}:=\chi (|x|-r)u(x)$ and $v^{r}:=\chi (|x|-r)v(x)$ for $r\geq 0$.
Then for $x\in B_{r+1}(0)\backslash B_{r}(0)$, we have
\begin{equation*}
u^{r}(x)=(|x|-r)u(x),\text{ }v^{r}(x)=(|x|-r)v(x),
\end{equation*}%
and
\begin{equation*}
\nabla u^{r}(x)=(|x|-r)\nabla u(x)+\frac{x}{|x|}u(x),\quad \nabla
v^{r}(x)=(|x|-r)\nabla v(x)+\frac{x}{|x|}v(x).
\end{equation*}%
Next, we only consider one component of $\vartheta (r)$ because the other is
similar. Set $\varrho :=\min \left\{ \lambda ,1\right\} $. Then we get
\begin{eqnarray}
&&\int_{\mathbb{R}^{3}}(\nabla u\nabla u^{r}+\lambda uu^{r})dx  \notag \\
&\geq &\varrho \vartheta _{1}(r+1)+\int_{B_{r+1}(0)\backslash B_{r}(0)}\left[
(|x|-r)(|\nabla u|^{2}+\lambda |u|^{2})+u\frac{x}{|x|}\cdot \nabla u\right]
dx  \notag \\
&\geq &\varrho \vartheta _{1}(r+1)-\frac{1}{2}\int_{B_{r+1}(0)\backslash
B_{r}(0)}(|\nabla u|^{2}+|u|^{2})dx  \notag \\
&\geq &\frac{2\varrho +1}{2}\vartheta _{1}(r+1)-\frac{1}{2}\vartheta _{1}(r).
\label{e5-1}
\end{eqnarray}%
Since $(u,v)$ solves system (\ref{e1-2}), we have
\begin{eqnarray*}
\int_{\mathbb{R}^{3}}(\nabla u\nabla u^{r}+\lambda uu^{r})dx&\leq &\int_{\mathbb{R}^{3}}(|u|^{p-1}u^{r}+\beta |v|^{\frac{p}{2}}|u|^{%
\frac{p}{2}-1}u^{r})dx  \notag \\
&\leq &\int_{\mathbb{R}^{3}\backslash B_{r}}(|u|^{p}+\beta |v|^{\frac{p}{2}%
}|u|^{\frac{p}{2}})dx  \notag \\
&\leq &\left( 1+\frac{\beta }{2}\right) \int_{\mathbb{R}^{3}\backslash
B_{r}}(|u|^{p}+|v|^{p})dx  \notag \\
&\leq &\left( 1+\frac{\beta }{2}\right) \left[ \int_{\mathbb{R}%
^{3}\backslash B_{r}}(|\nabla u|^{2}+|u|^{2})dx\right] ^{\frac{p}{2}}\left[
\int_{\mathbb{R}^{3}\backslash B_{r}}(|\nabla v|^{2}+|v|^{2})dx\right] ^{%
\frac{p}{2}}.  \label{e5-2}
\end{eqnarray*}%
Using this, together with (\ref{e5-1}) one has
\begin{equation*}
\frac{2\varrho +1}{2}|\vartheta (r+1)|_{1}-\frac{1}{2}|\vartheta
(r)|_{1}\leq C_{0}(|\vartheta _{1}(r)|^{\frac{p}{2}}+|\vartheta _{2}(r)|^{%
\frac{p}{2}})\leq C_{1}|\vartheta (r)|_{1}^{\frac{p}{2}}.
\end{equation*}%
Thus we have
\begin{equation*}
\frac{|\vartheta (r+1)|_{1}}{|\vartheta (r)|_{1}}\leq \frac{1}{2\varrho +1}%
\left( 1+C_{2}\frac{|\vartheta (r)|_{1}^{\frac{p}{2}}}{|\vartheta (r)|_{1}}%
\right) =\frac{1}{2\varrho +1}\left( 1+C_{2}|\vartheta (r)|_{1}^{\frac{p-2}{2%
}}\right) .
\end{equation*}%
Since $|\vartheta (r)|_{1}\rightarrow 0$ as $r\rightarrow \infty $, there
exits $r_{0}>0$ such that for $r>r_{0}+1$,
\begin{equation*}
|\vartheta (r)|_{1}\leq |\vartheta (\lfloor r\rfloor )|_{1}=|\vartheta
(r_{0})|_{1}\prod_{k=r_{0}}^{\lfloor r\rfloor -1}\frac{|\vartheta (k+1)|_{1}%
}{|\vartheta (k)|_{1}}\leq \left( C_{3}+1\right) ^{r_{0}-r+1}(\Vert u\Vert
_{H^{1}}^{2}+\Vert v\Vert _{H^{1}}^{2}).
\end{equation*}%
As $|\vartheta (r)|_{1}\leq \left( C_{3}+1\right) ^{r_{0}-r+1}(\Vert u\Vert
_{H^{1}}^{2}+\Vert v\Vert _{H^{1}}^{2})$ for $r\leq r_{0}+1$, then for all $%
r\geq 0$, we get
\begin{eqnarray*}
|\vartheta (r)|_{1} &\leq &\left( C_{3}+1\right) ^{r_{0}-r+1}(\Vert u\Vert
_{H^{1}}^{2}+\Vert v\Vert _{H^{1}}^{2}) \\
&=&(\Vert u\Vert _{H^{1}}^{2}+\Vert v\Vert _{H^{1}}^{2})\left(
C_{3}+1\right) ^{r_{0}+1}e^{-\ln (C_{3}+1)r}.
\end{eqnarray*}

For $x\in \mathbb{R}^{3}$ with $|x|\geq 2$, let $r=\frac{1}{2}|x|$. Then $%
B_{1}(x)\subset \mathbb{R}^{3}\backslash B_{r}(0)$, and there holds
\begin{equation*}
\Vert u\Vert _{H^{1}(B_{1}(x))}^{2}\leq \Vert u\Vert _{H^{1}(\mathbb{R}%
^{3}\backslash B_{1}(x))}^{2}=\vartheta _{1}(r)\leq C_{4}e^{-D_{1}r},
\end{equation*}%
where $D_{1}>0$. By the standard elliptic regularity theory, we have $u,v\in
W^{2,s}(\mathbb{R}^{3})\cap \mathcal{C}^{2}(\mathbb{R}^{3})$ for $s>2$. Thus
for $x\in \mathbb{R}^{3}\backslash B_{2}(0)$, by using the embedding $%
W^{2,s}(B_{1/2}(x))\hookrightarrow L^{\infty }(B_{1/2}(x))$ and \cite[Lemma
2.3]{ACS}, one has
\begin{equation*}
|u(x)|\leq |u(x)|_{L^{\infty }(B_{1/2}(x))}\leq C_{5}\Vert u\Vert
_{W^{2,s}(B_{1/2}(x))}\leq C_{6}e^{-D_{2}|x|},
\end{equation*}%
where $D_{2}>0$. Since $u$ is continuous, we have $|u(x)|\leq
C_{6}e^{-D_{2}|x|}$ for $x\in B_{2}(0)$. Therefore,
\begin{equation*}
|u(x)|\leq Ce^{-\zeta |x|}\text{ for every }x\in \mathbb{R}^{3}.
\end{equation*}%
Similarly, we have
\begin{equation*}
|v(x)|\leq Ce^{-\zeta |x|}\text{ for every }x\in \mathbb{R}^{3}.
\end{equation*}
The proof is complete.
\end{proof}

\section{Dynamic behavior of standing waves}

\subsection{The local well-posedness and orbital stability}

We firstly consider the local well-posedness for system (\ref{e1-2}) with
the initial data. Set
\begin{equation*}
g_{1}(\psi _{1},\psi _{2}):=|\psi _{1}|^{p-2}\psi _{1}+\beta |\psi _{2}|^{%
\frac{p}{2}}|\psi _{1}|^{\frac{p}{2}-2}\psi _{1}-\gamma \left( |x|^{-1}\ast
\sum\limits_{j=1}^{2}|\psi _{j}|^{2}\right) \psi _{1}
\end{equation*}%
and
\begin{equation*}
g_{2}(\psi _{1},\psi _{2}):=|\psi _{2}|^{p-2}\psi _{2}+\beta |\psi _{1}|^{%
\frac{p}{2}}|\psi _{2}|^{\frac{p}{2}-2}\psi _{2}-\gamma \left( |x|^{-1}\ast
\sum\limits_{j=1}^{2}|\psi _{j}|^{2}\right) \psi _{2}.
\end{equation*}%
Then the Cauchy problem (\ref{e1-2}) is rewritten as
\begin{equation*}
\left\{
\begin{array}{ll}
i\partial _{t}\psi _{1}+\Delta \psi _{1}+g_{1}(\psi _{1},\psi _{2})=0 &
\quad \text{in}\quad \mathbb{R}^{3}, \\
i\partial _{t}\psi _{2}+\Delta \psi _{2}+g_{2}(\psi _{1},\psi _{2})=0 &
\quad \text{in}\quad \mathbb{R}^{3}, \\
\psi _{1}(0,x)=\psi _{1}(0)\in H^{1}(\mathbb{R}^{3}),\text{ }\psi
_{2}(0,x)=\psi _{2}(0)\in H^{1}(\mathbb{R}^{3}). &
\end{array}%
\right.  \label{e3-30}
\end{equation*}

\begin{definition}
\label{D3-1} The pair $(s,r)$ is referred to be as an Strichartz admissible
if
\begin{equation*}
\frac{2}{s}+\frac{3}{r}=\frac{3}{2}\quad \text{for }s,r\in \lbrack 2,\infty ].
\end{equation*}
\end{definition}

We introduce the spaces $L_{t}^{s}([0,T);L^{r})$ and $%
L_{t}^{s}([0,T);W^{1,r})$ (as $L_{t}^{s}L^{r}$ and $%
L_{t}^{s}W^{1,r}$ respectively to simplify the notation) equipped with
the Strichartz norms:
\begin{equation*}
\Vert w(t,x)\Vert _{L_{t}^{s}L^{r}}=\left( \int_{0}^{T}\Vert w(t,\cdot
)\Vert _{r}^{s}dt\right) ^{\frac{1}{s}}\text{ and }\Vert w(t,x)\Vert
_{L_{t}^{s}W^{1,r}}=\left( \int_{0}^{T}\Vert w(t,\cdot )\Vert _{W^{1,r}(%
\mathbb{R}^{3})}^{s}dt\right) ^{\frac{1}{s}},
\end{equation*}%
where the function $w(t,x)$ is defined on the time-space strip $[0,T)\times
\mathbb{R}^{N}.$

\begin{definition}
\label{D3-2}Let $T>0$. We say that $(\psi _{1},\psi _{2})$ is an integral
solution of the Cauchy problem (\ref{e3-30}) on the time interval $[0,T)$ if%
\newline
$(i)$ $\psi _{1},\psi _{2}\in \mathcal{C}([0,T);H^{1})\cap
L_{t}^{s}([0,T);W^{1,r}).$\newline
$(ii)$ For all $t\in \lbrack 0,T)$ it holds
\begin{equation*}
(\psi _{1}(t),\psi _{2}(t))=e^{it\Delta }(\psi _{1}(0),\psi
_{2}(0))+i\int_{0}^{t}e^{i(t-\rho )\Delta }(g_{1}(\psi _{1}(\rho ),\psi
_{2}(\rho )),g_{2}(\psi _{1}(\rho ),\psi _{2}(\rho )))d\rho .
\end{equation*}
\end{definition}

Let us recall the following well-known Strichartz's estimates (see \cite%
{C,KT}).

\begin{proposition}
\label{P3-4} Let $T>0.$ For every admissible pairs $(s,r)$ and $(\tilde{s},%
\tilde{r})$, there exists a constant $C>0$ such that the following
properties hold:\newline
$(i)$ For every $\varphi \in L^{2}(\mathbb{R}^{3})$, the function $%
t\mapsto e^{it\Delta }\varphi $ belongs to $L_{t}^{s}([0,T);L^{r})\cap
\mathcal{C(}[0,T);L^{2}(\mathbb{R}^{3}))$ and
$\left\Vert e^{it\Delta }\varphi \right\Vert _{L_{t}^{s}L^{r}}\leq C\Vert
\varphi \Vert _{L^{2}}.$\newline
$(ii)$ Let $F\in L_{t}^{\tilde{s}^{\prime }}([0,T);L^{\tilde{r}^{\prime
}})$, where we use a prime to denote conjugate indices. Then the function $t\mapsto \Phi _{F}(t):=\int_{0}^{t}e^{i(t-\rho )\Delta }F(\rho )d\rho$
belongs to $L_{t}^{s}([0,T);L^{r})\cap \mathcal{C}([0,T);L^{2}(%
\mathbb{R}^{3}))$ and
$\left\Vert \Phi _{F}\right\Vert _{L_{t}^{s}L^{r}}\leq C\Vert F\Vert
_{L_{t}^{\tilde{s}^{\prime }}L^{\tilde{r}^{\prime }}}.$\newline
$(iii)$ For every $\varphi \in H^{1}(\mathbb{R}^{3})$, the function $%
t\mapsto e^{it\Delta }\varphi $ belongs to $L_{t}^{s}([0,T);W^{1,r})\cap
\mathcal{C}([0,T);L^{2}(\mathbb{R}^{3}))$ and $\left\Vert e^{it\Delta }\varphi \right\Vert _{L_{t}^{s}W^{1,r}}\leq
C\Vert \varphi \Vert _{H^{1}}.$
\end{proposition}

We now choose the $L^{2}$-admissible pair $(s_{1},r_{1})$ and $(s_{2},r_{2}),$
respectively, as
$(s_{1},r_{1})=\left( \frac{4p}{3(p-2)},p\right)$  and $(s_{2},r_{2})=\left( 8,\frac{12}{5}\right).$
\newline
\textbf{Now we give the proof of Theorem \ref{t8}:} Define the solution map $%
\Psi $  by%
\begin{equation*}
\Psi (\psi _{1},\psi _{2}):=e^{it\Delta }(\psi _{1}(0),\psi
_{2}(0))+i\int_{0}^{t}e^{i(t-\rho )\Delta }(g_{1}(\psi _{1}(\rho ),\psi
_{2}(\rho )),g_{2}(\psi _{1}(\rho ),\psi _{2}(\rho )))d\rho .
\end{equation*}%
For $T>0$, we define the complete metric space
\begin{equation*}
\mathcal{B}_{R,T}:=\left\{ (\psi _{1},\psi _{2})\in \mathcal{H}\times
\mathcal{H}\text{ }|\text{\ }\vartheta (\psi _{1},\psi _{2})\leq R\right\}
\text{ for some }R>0,
\end{equation*}%
where
\begin{equation*}
\mathcal{H}:=\mathcal{C}([0,T);H^{1})\cap
L_{t}^{s_{1}}([0,T);W^{1,r_{1}})\cap L_{t}^{s_{2}}([0,T);W^{1,r_{2}})
\end{equation*}%
and
\begin{equation*}
\vartheta (\psi _{1},\psi _{2}):=\max \left\{ \sup_{t\in \lbrack 0,T)}\Vert
(\psi _{1},\psi _{2})\Vert _{H^{1}},\Vert (\psi _{1},\psi _{2})\Vert
_{L_{t}^{s_{1}}W^{1,r_{1}}},\Vert (\psi _{1},\psi _{2})\Vert
_{L_{t}^{s_{2}}W^{1,r_{2}}}\right\}
\end{equation*}%
with the metric
\begin{equation*}
d((\psi _{1},\psi _{2})-(\omega _{1},\omega _{2})):=\Vert (\psi _{1},\psi
_{2})-(\omega _{1},\omega _{2})\Vert _{L_{t}^{s_{1}}L^{r_{1}}\cap
L_{t}^{s_{2}}L^{r_{2}}}.
\end{equation*}%
By Strichartz estimates, we have
\begin{eqnarray}
\Vert \Psi (\psi _{1},\psi _{2})\Vert _{L_{t}^{\tilde{s}}W^{1,\tilde{r}%
}} &\lesssim &\Vert (\psi _{1}(0),\psi _{2}(0))\Vert _{H^{1}}+\Vert
(f_{1}(\psi _{1},\psi _{2}),f_{2}(\psi _{1},\psi _{2}))\Vert
_{L_{t}^{s_{1}^{\prime }}W^{1,r_{1}^{\prime }}}  \notag \\
&&+\Vert (h_{1}(\psi _{1},\psi _{2}),h_{2}(\psi _{1},\psi _{2}))\Vert
_{L_{t}^{s_{2}^{\prime }}W^{1,r_{2}^{\prime }}},  \label{e3-36}
\end{eqnarray}%
where
\begin{equation*}
f_{1}:=|\psi _{1}|^{p-2}\psi _{1}+\beta |\psi _{2}|^{\frac{p}{2}}|\psi
_{1}|^{\frac{p}{2}-2}\psi _{1},\text{ }f_{2}:=|\psi _{2}|^{p-2}\psi
_{2}+\beta |\psi _{1}|^{\frac{p}{2}}|\psi _{2}|^{\frac{p}{2}-2}\psi _{2},
\end{equation*}%
and
\begin{equation*}
h_{1}:=\gamma \left( |x|^{-1}\ast \sum\limits_{j=1}^{2}|\psi
_{j}|^{2}\right) \psi _{1},\text{ }h_{2}:=\gamma \left( |x|^{-1}\ast
\sum\limits_{j=1}^{2}|\psi _{j}|^{2}\right) \psi _{2}.
\end{equation*}%
By using the H\"{o}lder inequality in time in (\ref{e3-36}), we obtain
\begin{eqnarray*}
\Vert \Psi (\psi _{1},\psi _{2})\Vert _{L_{t}^{\tilde{s}}W^{1,\tilde{r}%
}} &\lesssim &\Vert (\psi _{1}(0),\psi _{2}(0))\Vert _{H^{1}}+T^{\theta
}\Vert (f_{1}(\psi _{1},\psi _{2}),f_{2}(\psi _{1},\psi _{2}))\Vert
_{L_{t}^{s_{1}}W^{1,r_{1}^{\prime }}} \\
&&+T^{\delta }\Vert (h_{1}(\psi _{1},\psi _{2}),h_{2}(\psi _{1},\psi
_{2}))\Vert _{L_{t}^{s_{2}}W^{1,r_{2}^{\prime }}}
\end{eqnarray*}%
where $\theta :=\frac{6-p}{2p}$ and $\delta :=\frac{3}{4}$.

Using H\"{o}lder and Sobolev inequalities, we have
\begin{equation*}
\Vert f_{1}(\psi _{1},\psi _{2})\Vert _{L_{t}^{s_{1}}W^{1,r_{1}^{\prime
}}}\lesssim \Vert \psi _{1}\Vert _{L_{t}^{s_{1}}W^{r_{1}}}^{p-1}+2\Vert
\psi _{2}\Vert _{L_{t}^{s_{1}}W^{r_{1}}}^{\frac{p}{2}}\Vert \psi
_{1}\Vert _{L_{t}^{s_{1}}W^{r_{1}}}^{\frac{p-2}{2}}.
\end{equation*}%
Similarly, we get
\begin{equation*}
\Vert f_{2}(\psi _{1},\psi _{2})\Vert _{L_{t}^{s_{1}}W^{1,r_{1}^{\prime
}}}\lesssim \Vert \psi _{2}\Vert _{L_{t}^{s_{1}}W^{r_{1}}}^{p-1}+2\Vert
\psi _{1}\Vert _{L_{t}^{s_{1}}W^{r_{1}}}^{\frac{p}{2}}\Vert \psi
_{2}\Vert _{L_{t}^{s_{1}}W^{r_{1}}}^{\frac{p-2}{2}}.
\end{equation*}%
Moreover, it follows from H\"{o}lder, Hardy-Littlewood-Sobolev and Sobolev
inequalities that
\begin{eqnarray*}
&&\Vert \nabla h_{1}(\psi _{1},\psi _{2})\Vert
_{L_{t}^{s_{2}}L^{r_{2}^{\prime }}} \\
&\lesssim &\Vert |x|^{-1}\ast (|\psi _{1}|\nabla \psi _{1})\Vert
_{L_{t}^{s_{2}}L^{6}}\Vert \psi _{1}\Vert _{L_{t}^{\infty }L^{\frac{%
12}{5}}}+\Vert |x|^{-1}\ast |\psi _{1}|^{2}\Vert _{L_{t}^{\infty
}L^{6}}\Vert \nabla \psi _{1}\Vert _{L_{t}^{s_{2}}L^{\frac{12}{5}}}
\\
&&+\Vert |x|^{-1}\ast (|\psi _{2}|\nabla \psi _{2})\Vert
_{L_{t}^{s_{2}}L^{6}}\Vert \psi _{1}\Vert _{L_{t}^{\infty }L^{\frac{%
12}{5}}}+\Vert |x|^{-1}\ast |\psi _{2}|^{2}\Vert _{L_{t}^{\infty
}L^{6}}\Vert \nabla \psi _{1}\Vert _{L_{t}^{s_{2}}L^{\frac{12}{5}}}
\\
&\lesssim &\Vert |\psi _{1}|\nabla \psi _{1}\Vert _{L_{t}^{s_{2}}L^{%
\frac{6}{5}}}\Vert \psi _{1}\Vert _{L_{t}^{\infty }L^{r_{2}}}+\Vert \psi
_{1}\Vert _{L_{t}^{\infty }L^{r_{2}}}^{2}\Vert \nabla \psi _{1}\Vert
_{L_{t}^{s_{2}}L^{r_{2}}} \\
&&+\Vert |\psi _{2}|\nabla \psi _{2}\Vert _{L_{t}^{s_{2}}L^{\frac{6}{5}%
}}\Vert \psi _{1}\Vert _{L_{t}^{\infty }L^{r_{2}}}+\Vert \psi _{2}\Vert
_{L_{t}^{\infty }L^{r_{2}}}^{2}\Vert \nabla \psi _{1}\Vert
_{L_{t}^{s_{2}}L^{r_{2}}} \\
&\lesssim &\Vert \nabla \psi _{1}\Vert _{L_{t}^{s_{2}}L^{r_{2}}}\Vert
\psi _{1}\Vert _{L_{t}^{\infty }H^{1}}\Vert \psi _{1}\Vert
_{L_{t}^{\infty }H^{1}}+\Vert \psi _{1}\Vert _{L_{t}^{\infty
}H^{1}}^{2}\Vert \nabla \psi _{1}\Vert _{L_{t}^{s_{2}}L^{r_{2}}} \\
&&+\Vert \nabla \psi _{2}\Vert _{L_{t}^{s_{2}}L^{r_{2}}}\Vert \psi
_{2}\Vert _{L_{t}^{\infty }H^{1}}\Vert \psi _{1}\Vert _{L_{t}^{\infty
}H^{1}}+\Vert \psi _{2}\Vert _{L_{t}^{\infty }H^{1}}^{2}\Vert \nabla
\psi _{1}\Vert _{L_{t}^{s_{2}}L^{r_{2}}},
\end{eqnarray*}%
and
\begin{eqnarray*}
\Vert h_{1}(\psi _{1},\psi _{2})\Vert _{L_{t}^{s_{2}}W^{1,r_{2}^{\prime
}}} &\lesssim &\Vert \psi _{1}\Vert _{L_{t}^{\infty }H^{1}}^{2}\Vert
\psi _{1}\Vert _{L_{t}^{s_{2}}W^{r_{2}}}+\Vert \psi _{2}\Vert
_{L_{t}^{s_{2}}W^{r_{2}}}\Vert \psi _{2}\Vert _{L_{t}^{\infty
}H^{1}}\Vert \psi _{1}\Vert _{L_{t}^{\infty }H^{1}} \\
&&+\Vert \psi _{2}\Vert _{L_{t}^{\infty }H^{1}}^{2}\Vert \psi _{1}\Vert
_{L_{t}^{s_{2}}W^{r_{2}}}.
\end{eqnarray*}%
Similarly, we have
\begin{eqnarray*}
\Vert h_{2}(\psi _{1},\psi _{2})\Vert _{L_{t}^{s_{2}}W^{1,r_{2}^{\prime
}}} &\lesssim &\Vert \psi _{2}\Vert _{L_{t}^{\infty }H^{1}}^{2}\Vert
\psi _{2}\Vert _{L_{t}^{s_{2}}W^{r_{2}}}+\Vert \psi _{1}\Vert
_{L_{t}^{s_{2}}W^{r_{2}}}\Vert \psi _{1}\Vert _{L_{t}^{\infty
}H^{1}}\Vert \psi _{2}\Vert _{L_{t}^{\infty }H^{1}} \\
&&+\Vert \psi _{1}\Vert _{L_{t}^{\infty }H^{1}}^{2}\Vert \psi _{2}\Vert
_{L_{t}^{s_{2}}W^{r_{2}}}.
\end{eqnarray*}%
So, for $(\psi _{1},\psi _{2})\in \mathcal{B}_{R,T}$, we get
\begin{equation*}
\Vert \Psi (\psi _{1},\psi _{2})\Vert _{L_{t}^{\tilde{s}}W^{1,\tilde{r}%
}}\leq C\Vert (\psi _{1}(0),\psi _{2}(0))\Vert _{\mathbf{H}}+CT^{\theta
}R^{p-1}+CT^{\delta }R^{3}.
\end{equation*}%
Now we choose $R=2C\Vert (\psi _{1}(0),\psi _{2}(0))\Vert _{\mathbf{H}}>0$
and $0<T<1$ such that $CT^{\theta }R^{p-2}+CT^{\delta }R^{2}\leq \frac{1}{2}.$ This implies that $\mathcal{B}_{R,T}$ is an invariant set of $\Psi $.
Furthermore, for $(\psi _{1},\psi _{2}),(\omega _{1},\omega _{2})\in
\mathcal{B}_{R,T}$, by using the same argument, we get
\begin{equation*}
d(\Psi (\psi _{1},\psi _{2})-\Psi (\omega _{1},\omega _{2}))\lesssim
(T^{\theta }+T^{\delta })d((\psi _{1},\psi _{2})-(\omega _{1},\omega _{2})),
\end{equation*}%
which leads to
\begin{equation*}
d(\Psi (\psi _{1},\psi _{2})-\Psi (\omega _{1},\omega _{2}))\leq \frac{1}{2}%
d((\psi _{1},\psi _{2})-(\omega _{1},\omega _{2})).
\end{equation*}%
This implies that $\Psi $ is a contraction map. In particular, $\Psi $ has a
unique fixed point in this space. Moreover, by the classical arguments in
\cite{C}, the solution $(\psi _{1},\psi _{2})$ satisfies the mass
conservation law
\begin{equation*}
\Vert \psi _{1}(t)\Vert _{2}^{2}+\Vert \psi _{2}(t)\Vert _{2}^{2}=\Vert \psi
_{1}(0)\Vert _{2}^{2}+\Vert \psi _{2}(0)\Vert _{2}^{2}
\end{equation*}%
and the energy conservation law
\begin{eqnarray*}
I_{\gamma ,\beta }(\psi _{1}(t),\psi _{2}(t)) &=&\frac{1}{2}\int_{\mathbb{R}^{N}}(|\nabla \psi _{1}(t)|^{2}+|\nabla \psi
_{2}(t)|^{2})dx+\frac{\gamma }{4}\int_{\mathbb{R}^{3}}\phi _{\psi
_{1}(t),\psi _{2}(t)}(|\psi _{1}(t)|^{2}+|\psi _{2}(t)|^{2})dx \\
&&-\frac{1}{p}\int_{\mathbb{R}^{3}}(|\psi _{1}(t)|^{p}+|\psi
_{2}(t)|^{p}+2\beta |\psi _{1}(t)|^{\frac{p}{2}}|\psi _{2}(t)|^{\frac{p}{2}%
})dx\\
&=&I_{\gamma ,\beta }(\psi _{1}(0),\psi _{2}(0)).
\end{eqnarray*}

\textbf{We give the proof of Theorem \ref{t3}:} Since the proof is similar
to the classical arguments in \cite{L}, we only give a sketch. Assume on the
contrary. Then there exist $\varepsilon _{0}>0$, $\left\{ \psi
_{1}^{n}(0),\psi _{2}^{n}(0)\right\} $ and $\left\{ t_{n}\right\} \subset
\mathbb{R}^{+}$ such that
\begin{equation*}
\inf_{(u,v)\in \mathcal{Z}(c)}\Vert (\psi _{1}^{n}(0),\psi
_{2}^{n}(0))-(u,v)\Vert _{H^{1}(\mathbb{R}^{3})\times H^{1}(\mathbb{R}%
^{3})}\rightarrow 0
\end{equation*}%
and
\begin{equation}
\inf_{(u,v)\in \mathcal{Z}(c)}\Vert (\psi _{1}^{n}(t_{n}),\psi
_{2}^{n}(t_{n}))-(u,v)\Vert _{H^{1}(\mathbb{R}^{3})\times H^{1}(\mathbb{R}%
^{3})}\geq \varepsilon _{0}.  \label{e3-18}
\end{equation}%
According to mass and energy conservation laws, we have%
\begin{equation*}
\sum_{i=1}^{2}\Vert \psi _{i}^{n}(t_{n})\Vert _{2}^{2}=\sum_{i=1}^{2}\Vert
\psi _{i}^{n}(0)\Vert _{2}^{2}\quad \text{and }I_{\gamma ,\beta }(\psi
_{1}^{n}(t_{n}),\psi _{2}^{n}(t_{n}))=I_{\gamma ,\beta }(\psi
_{1}^{n}(0),\psi _{2}^{n}(0))\text{ for }i=1,2.
\end{equation*}%
Define
\begin{equation*}
\widetilde{\psi }_{i}^{n}=\frac{\psi _{i}^{n}(t_{n})}{\sum_{i=1}^{2}\Vert
\psi _{i}^{n}(t_{n})\Vert _{2}^{2}}c\text{ for }i=1,2.
\end{equation*}%
A direct calculation shows that
\begin{equation*}
\sum_{i=1}^{2}\Vert \widetilde{\psi }_{i}^{n}\Vert _{2}^{2}=c\quad \text{and
}I_{\gamma ,\beta }(\widetilde{\psi }_{1}^{n},\widetilde{\psi }%
_{2}^{n})=\sigma _{\gamma ,\beta }(c)+o_{n}(1),
\end{equation*}%
which indicates that $(\widetilde{\psi }_{1}^{n},\widetilde{\psi }_{2}^{n})$
is a minimizing sequence for $\sigma _{\gamma ,\beta }(c)$. By Theorems \ref%
{t1} and \ref{t2}, it is precompact in $\mathbf{H}$, which is a
contradiction with (\ref{e3-18}). Therefore, the set $\mathcal{Z}(c)$ is
orbitally stable.

\subsection{The global well-posedness}

\textbf{We give the proof of Theorem \ref{t9}:} Since system (\ref{e1-2})
with the initial data is locally well-posed by Theorem \ref{t8}, we obtain
that either $T_{\max }=+\infty $, or $T_{\max }<+\infty $ and $%
\lim_{t\rightarrow T_{\max }^{-}}A(\psi _{1}(t),\psi _{2}(t))\rightarrow
+\infty $. Then the global existence boils down to obtaining a priori bound
on $A(\psi _{1}(t),\psi _{2}(t))$.\newline
$(i)$ $2<p\leq \frac{10}{3}$. By mass and energy conservation laws and (\ref%
{e2-3}), we have
\begin{eqnarray*}
&&A(\psi _{1}(t),\psi _{2}(t)) \\
&\leq &2I_{\gamma ,\beta }(\psi _{1}(0),\psi _{2}(0))+\frac{2}{p}C(\psi
_{1}(t),\psi _{2}(t)) \\
&\leq &2I_{\gamma ,\beta }(\psi _{1}(0),\psi _{2}(0))+4(1+\beta )\mathcal{S}%
_{p}\left( \Vert \psi _{1}(0)\Vert _{2}^{2}+\Vert \psi _{2}(0)\Vert
_{2}^{2}\right) ^{\frac{6-p}{4}}\left( A(\psi _{1}(t),\psi _{2}(t))\right) ^{%
\frac{3(p-2)}{4}}.
\end{eqnarray*}%
If $2<p<\frac{10}{3},$ then there exists a constant $M>0$ independent of $t$
such that $A(\psi _{1}(t),\psi _{2}(t))\leq M$. If $p=\frac{10}{3},$ then
also there exists a constant $M>0$ independent of $t$ such that $A(\psi
_{1}(t),\psi _{2}(t))\leq M$ provided that $\Vert \psi _{1}(0)\Vert
_{2}^{2}+\Vert \psi _{2}(0)\Vert _{2}^{2}$ is small.\newline
$(ii)$ $10/3<p<6$. If by contradiction $T_{\max }<+\infty $, we have $%
\lim_{t\rightarrow T_{\max }^{-}}A(\psi _{1}(t),\psi _{2}(t))\rightarrow
+\infty $. Since
\begin{equation*}
I_{\gamma ,\beta }(\psi _{1},\psi _{2})-\frac{2}{3(p-2)}\mathcal{P}_{\gamma
,\beta }(\psi _{1},\psi _{2})=\frac{3p-10}{6(p-2)}A(\psi _{1},\psi _{2})+%
\frac{3p-8}{12(p-2)}B(\psi _{1},\psi _{2}),
\end{equation*}%
together with the energy conservation law, we have
$\lim_{t\rightarrow T_{\max }^{-}}\mathcal{P}_{\gamma ,\beta }(\psi _{1},\psi
_{2})=-\infty .$ Then by the continuity, these exists $t_{0}\in (0,T_{\max })$ such that $%
\mathcal{P}_{\gamma ,\beta }(\psi _{1}(t_{0}),\psi _{2}(t_{0}))=0$, where we
have used the assumption of $\mathcal{P}_{\gamma ,\beta }(\psi _{1}(0),\psi
_{2}(0))>0.$ Using again the energy conservation law and the assumption of $%
I_{\gamma ,\beta }(\psi _{1}(0),\psi _{2}(0))<m_{\gamma ,\beta }(c)$, we
deduce that%
\begin{equation*}
m_{\gamma ,\beta }(c)>I_{\gamma ,\beta }(\psi _{1}(0),\psi
_{2}(0))=I_{\gamma ,\beta }(\psi _{1}(t_{0}),\psi _{2}(t_{0}))\geq
\inf_{(u,v)\in \mathcal{M}_{\gamma ,\beta }(c)}I_{\gamma ,\beta
}(u,v)=m_{\gamma ,\beta }(c),
\end{equation*}%
which is a contradiction. So, $T_{\max }=+\infty .$

\subsection{Finite time blow-up and strong unstablility}

\begin{lemma}
\label{L6-1} Assume that the conditions of Theorem \ref{t5} or Theorem \ref%
{t6} hold. Let $(\psi _{1}(0),\psi _{2}(0))\in S_{c}$ such that $I_{\gamma
,\beta }(\psi _{1}(0),\psi _{2}(0))<m_{\gamma ,\beta }(c)$. Then
\begin{equation*}
\Phi _{(\psi _{1}(0),\psi _{2}(0))}(t):=\frac{t^{2}}{2}A(\psi _{1}(0),\psi
_{2}(0))+\frac{\gamma t}{4}B(\psi _{1}(0),\psi _{2}(0))-\frac{t^{\frac{3(p-2)%
}{2}}}{p}C(\psi _{1}(0),\psi _{2}(0))
\end{equation*}%
has a unique global maximum point $t^{-}(\psi _{1}(0),\psi _{2}(0))$.
Furthermore, if
\begin{equation*}
0<t^{-}(\psi _{1}(0),\psi _{2}(0))<1\text{ and }(|x|\psi _{1}(0),|x|\psi
_{2}(0))\in L^{2}(\mathbb{R}^{3})\times L^{2}(\mathbb{R}^{3}),
\end{equation*}%
then the solution $(\psi _{1}(t),\psi _{2}(t))$ of system (\ref{e1-2}) with
initial datum $(\psi _{1}(0),\psi _{2}(0))$ blows-up in finite time.
\end{lemma}

\begin{proof}
Let $(\psi _{1}(0),\psi _{2}(0))\in S_{c}$. Since $\frac{10}{3}\leq p<6,$ it
is easy to prove that $\Phi _{(\psi _{1}(0),\psi _{2}(0))}(t)$ has a unique
global maximum point $t^{-}(\psi _{1}(0),\psi _{2}(0))$ and $\Phi _{(\psi
_{1}(0),\psi _{2}(0))}(t)$ is strict decreasing and concave on $(t^{-}(\psi
_{1}(0),\psi _{2}(0)),+\infty )$. We claim that if $t^{-}(\psi _{1}(0),\psi
_{2}(0))\in (0,1)$, then
\begin{equation}
\mathcal{P}_{\gamma ,\beta }(\psi _{1}(0),\psi _{2}(0))\leq I_{\gamma ,\beta
}(\psi _{1}(0),\psi _{2}(0))-m_{\gamma ,\beta }(c).  \label{e6-3}
\end{equation}%
The fact of $t^{-}(\psi _{1}(0),\psi _{2}(0))\in (0,1)$ implies that $%
\mathcal{P}_{\gamma ,\beta }(\psi _{1}(0),\psi _{2}(0))<0$. Then there holds
\begin{eqnarray*}
I_{\gamma ,\beta }(\psi _{1}(0),\psi _{2}(0))&=&\Phi _{(\psi _{1}(0),\psi _{2}(0))}(1) \\
&\geq &\Phi _{(\psi _{1}(0),\psi _{2}(0))}(t^{-}(\psi _{1}(0),\psi
_{2}(0)))-(t^{-}(\psi _{1}(0),\psi _{2}(0))-1)\Phi _{(\psi _{1}(0),\psi
_{2}(0))}^{\prime }(1) \\
&=&I_{\gamma ,\beta }((\psi _{1}(0),\psi _{2}(0))_{t^{-}(\psi _{1}(0),\psi
_{2}(0))})-|\mathcal{P}_{\gamma ,\beta }(\psi _{1}(0),\psi
_{2}(0))|(1-t^{-}(\psi _{1}(0),\psi _{2}(0))) \\
&\geq &m_{\gamma ,\beta }(c)-|\mathcal{P}_{\gamma ,\beta }(\psi _{1}(0),\psi
_{2}(0))| \\
&=&m_{\gamma ,\beta }(c)+\mathcal{P}_{\gamma ,\beta }(\psi _{1}(0),\psi
_{2}(0)).
\end{eqnarray*}%
This proves the claim.

Now let us consider the solution $(\psi _{1},\psi _{2})$ with initial data $%
(\psi _{1}(0),\psi _{2}(0))$. Similar to the argument in Soave \cite[Lemma
5.3]{S2}, the map $(\psi _{1}(0),\psi _{2}(0))\mapsto t^{-}(\psi
_{1}(0),\psi _{2}(0))$ is continuous, and together with $0<t^{-}(\psi
_{1}(0),\psi _{2}(0))<1,$ we have $t^{-}(\psi _{1}(\tau ),\psi _{2}(\tau
))<1 $ for every $|\tau |$ small, say $|\tau |<\overline{\tau }$ for some $%
\overline{\tau }>0$. By (\ref{e6-3}) and the assumption of $I_{\gamma ,\beta
}(\psi _{1}(0),\psi _{2}(0))<m_{\gamma ,\beta }(c)$, we have
\begin{eqnarray*}
\mathcal{P}_{\gamma ,\beta }(\psi _{1}(\tau ),\psi _{2}(\tau )) &\leq
&I_{\gamma ,\beta }(\psi _{1}(\tau ),\psi _{2}(\tau ))-m_{\gamma ,\beta }(c)
\\
&=&I_{\gamma ,\beta }(\psi _{1}(0),\psi _{2}(0))-m_{\gamma ,\beta
}(c):=-\eta <0,
\end{eqnarray*}%
for every such $\tau $, and hence $t^{-}(\psi _{1}(\overline{\tau }),\psi
_{2}(\overline{\tau }))<1$. Appplying the continuity argument yields
\begin{equation*}
\mathcal{P}_{\gamma ,\beta }(\psi _{1}(t),\psi _{2}(t))\leq -\eta \text{ for
}t\in (T_{\min },T_{\max }).
\end{equation*}%
Since $(|x|\psi _{1}(0),|x|\psi _{2}(0))\in L^{2}(\mathbb{R}^{3})\times
L^{2}(\mathbb{R}^{3})$, by the Virial identity \cite[Proposition 6.5.1]{C},
the function
\begin{equation*}
\bar{f}(t):=\sum_{i=1}^{2}\int_{\mathbb{R}^{3}}|x|^{2}|\psi _{i}(t,x)|^{2}dx
\end{equation*}%
is of class $C^{2}$. We now use the method established by Cazenave \cite{C}.
By multiplying the first equation in system (\ref{e1-2}) by $i|x|^{2}%
\overline{\psi }_{1}$ and the second equation in system (\ref{e1-2}) by $%
i|x|^{2}\overline{\psi }_{2}$, taking the real part and integrating over $%
\mathbb{R}^{3}$, we have
\begin{equation*}
\bar{f}^{\prime }(t)=2\mathrm{Re}\sum_{i=1}^{2}\int_{\mathbb{R}^{3}}|x|^{2}%
\overline{\psi }_{i}\partial _{t}\psi _{i}dx=4\mathrm{Im}\sum_{i=1}^{2}\int_{%
\mathbb{R}^{3}}\overline{\psi }_{i}x\cdot \nabla \psi _{i}dx,
\end{equation*}%
and
\begin{eqnarray*}
\bar{f}^{\prime \prime }(t) =-4\mathrm{Im}\sum_{i=1}^{2}\int_{\mathbb{R}%
^{3}}(2x\cdot \nabla \psi _{i}+3\overline{\psi }_{i})\psi _{i}dx =8\mathcal{P}_{\gamma ,\beta }(\psi _{1}(t),\psi _{2}(t))\leq -8\eta <0
\end{eqnarray*}%
for $t\in (T_{\min },T_{\max })$. Integrating twice in time gives
$0\leq \bar{f}(t)\leq \bar{f}(0)+\bar{f}^{\prime }(0)t-4\eta t^{2}.$ Since the right hand side becomes negative for $t$ sufficiently large, it is
necessary that both $T_{\min }$ and $T_{\max }$ are bounded, which in turn
implies finite time blow-up. The proof is complete.
\end{proof}

\textbf{We give the proofs of Theorems \ref{t7} and \ref{t10}:} Theorem \ref%
{t7} directly follows from Lemma \ref{L6-1}.

Next, we prove Theorem \ref{t10}. Let $(\bar{u},\bar{v})$ be an energy ground state
obtained in Theorem \ref{t5} or \ref{t6}. Let $(\psi _{1}(t),\psi _{2}(t))$
be the solution to system (\ref{e1-2}) with the initial data $(\bar{u}_{s},\bar{v}_{s})$%
, where $(\bar{u}_{s},\bar{v}_{s})=(s^{3/2}\bar{u}(sx),s^{3/2}\bar{v}(sx))$ with $s>1.$ Then $%
(\bar{u}_{s},\bar{v}_{s})\rightarrow (\bar{u},\bar{v})$ in $\mathbf{H}$ as $s\searrow 1$, and so
it is sufficient to prove that $(\psi _{1}(t),\psi _{2}(t))$ blows up in
finite time. Clearly, $t^{-}(\bar{u}_{s},\bar{v}_{s})=\frac{1}{s}<1$, and by the
definition of $m_{\gamma ,\beta }(c)$ one has $I_{\gamma ,\beta }(\bar{u}_{s},\bar{v}_{s})<I_{\gamma ,\beta }(\bar{u},\bar{v})=m_{\gamma ,\beta
}(c).$ Therefore, it follows from Lemma \ref{L6-1} that the solution $(\psi
_{1}(t),\psi _{2}(t))$ blows up in finite time and the associated standing
wave is orbitally unstable.

\section*{Acknowledgments}

J. Sun was supported by the National Natural Science Foundation of China
(Grant No. 12371174) and Shandong Provincial Natural Science Foundation
(Grant No. ZR2020JQ01). T.F. Wu was supported by the National Science and
Technology Council, Taiwan (Grant No. 112-2115-M-390-001-MY3).

\end{document}